\documentclass[12pt]{amsart}
\usepackage[usenames,dvipsnames]{color}
\usepackage{amsmath,amssymb,amsthm,graphicx,mathrsfs,url,latexsym,enumerate}
\usepackage{color}
\usepackage{a4wide}
\usepackage{ifthen}
\usepackage{verbatim}
\usepackage{fancyhdr}
\usepackage{url}
\usepackage{bbm}
\usepackage[sans]{dsfont}
\usepackage{MnSymbol}

\usepackage{amsbsy}
\usepackage[colorlinks=true,linkcolor=red,citecolor=green]{hyperref}
\usepackage{wrapfig}
\usepackage{tikz}
\usepackage{geometry}

\definecolor{bleu}{RGB}{27,88,145}
\definecolor{mauve}{RGB}{138,20,79}

\newcommand{\mc}[1]{\mathcal{#1}}
\renewcommand{\Im}{\operatorname{Im}}
\renewcommand{\Re}{\operatorname{Re}}
\newcommand{\dist}{\operatorname{dist}}

\newcommand{\supp}{\operatorname{supp}}

\newcommand{\C}{\mathbb C}

\newcommand{\R}{\mathbb R}

\newcommand{\Hess}{\operatorname{Hess}}

\newcommand{\Ran}{\operatorname{Ran}}

\newcommand{\argmin}{\operatorname{argmin}}

\newcommand{\tr}{\operatorname{tr}}
\renewcommand{\div}{\operatorname{div}}
\renewcommand{\u}{{\bf u}}

\def\<{\langle}
\def\>{\rangle}

\def \jac {{\rm \, Jac }}

\newcommand{\be}{\begin{equation}}
\newcommand{\ee}{\end{equation}}
\newcommand{\bes}{\begin{equation*}}
\newcommand{\ees}{\end{equation*}}

\newcommand{\ve}{\varepsilon}
\newcommand {\pa}{\partial}

\def\Im {{\rm \, Im\;}}
\def\Re {{\rm \,Re\;}}

\def \tr{{\rm \, Tr\;}}
\def\dim {{\rm \; dim  \;}}

\def \jac {{\rm \, Jac }}

\numberwithin{equation}{section}
\numberwithin{figure}{section}

\newtheorem{theorem}{Theorem}
\newtheorem*{theorem*}{Theorem}
\newtheorem{proposition}{Proposition}
\newtheorem{definition}[proposition]{Definition}
\newtheorem*{definition*}{Definition}
\newtheorem{lemma}[proposition]{Lemma}
\newtheorem{corollary}[proposition]{Corollary}
\newtheorem{remark}[proposition]{Remark}

\newtheorem*{assumption*}{Assumption}

 \title[Exit time and principal eigenvalue of some non-reversible processes]{Exit time and  principal eigenvalue of  non-reversible elliptic diffusions}
\author[D. Le Peutrec]{Dorian Le Peutrec}
\address{D. Le Peutrec, Institut Denis Poisson, Universit\'e d'Orl\'eans}
\email{dorian.le-peutrec@univ-orleans.fr}

\author[L. Michel]{Laurent Michel}
\address{L. Michel, Institut Math\'ematique de Bordeaux, Universit\'e de Bordeaux}
\email{laurent.michel@math.u-bordeaux.fr}

\author[B. Nectoux]{Boris Nectoux}
\address{B. Nectoux, Laboratoire de Mathématiques Blaise Pascal, UCA}
\email{boris.nectoux@uca.fr}

\begin{document} 
 \maketitle
    \begin{abstract}
    In this work, we  analyse the metastability of non-reversible diffusion processes  
   $$dX_t=\boldsymbol{b}(X_t)dt+\sqrt h \, dB_t$$ on a bounded domain $\Omega$ when $\boldsymbol{b}$ admits the decomposition $\boldsymbol{b}=-(\nabla f+\boldsymbol{\ell})$ and $\nabla f \cdot \boldsymbol{\ell}=0$. 
  In this setting, we first show that, when $h\to 0$, the principal eigenvalue of the generator of $(X_t)_{t\ge 0}$  with Dirichlet boundary conditions on the boundary $\pa\Omega$ of $\Omega$ is exponentially close to the inverse of the mean exit time from $\Omega$, uniformly in the initial conditions $X_0=x$ within the compacts of~$\Omega$.
 The asymptotic behavior of the law of the exit time  in this limit  is also obtained. 
 The main novelty of these first results follows from 
 the consideration of non-reversible elliptic diffusions whose associated 
 dynamical systems $\dot X=\boldsymbol{b}(X)$ admit  equilibrium points on $\pa \Omega$.
 In a second time, when in addition $\div \boldsymbol{\ell} =0$, we derive a new sharp asymptotic equivalent  in the limit~$h\to 0$ of the principal eigenvalue of the generator of the process and  of its mean exit time from $\Omega$.   
Our proofs combine tools from large deviations theory  and from semiclassical analysis, and truly relies on the notion of quasi-stationary distribution. 
    \medskip

\noindent \textbf{Keywords.} Metastability, Eyring-Kramers type formulas, mean exit time, principal eigenvalue, non-reversible processes. \\
 \textbf{AMS classification.}  60J60, 35P15,   35Q82, 47F05, 60F10. 
   
 \end{abstract} 

\section{Introduction}\label{sec1}

 \subsection{Purpose of this work} 
Let  $\mathrm L>0$ and   $M= (\mathrm L\mathbb T)^d$, where $\mathbb T=\mathbb R/\mathbb Z$ is  the  one dimensional torus. 
Let $(X_t)_{t\ge 0}$ be the solution on $M$ of the  stochastic differential equation  
\begin{equation}\label{eq.langevin}
dX_t= \boldsymbol{b}(X_t)\, dt +\sqrt h \, dB_t,
\end{equation} 
where $h>0$, $(B_{t})_{t\ge 0}$ denotes the Brownian motion on $M$,  and $\boldsymbol{b}:M\to \mathbb R^d$  is a vector field. Such an equation is one of the most important models in statistical physics.  
In all this work,   $\Omega\subset M$ is a  $\mathcal C^\infty$ domain  
and  we denote by  
$$\tau_{\Omega^c}=\inf\{t\ge 0, X_t\notin \Omega\}$$
 the first exit time from $\Omega$ for the process \eqref{eq.langevin}.

When $h$ is small, due to the existence of stable equilibrium points  of the system $\dot X=\boldsymbol{b}(X)$, the process \eqref{eq.langevin} remains  trapped during a very long time 
in a neighborhood of such a point in~$M$, called a metastable region, before going to another metastable region. For this reason,  the process \eqref{eq.langevin} is said to be  metastable.  
 This phenomenon of   metastability  has been widely studied    through the  asymptotic behavior in the zero white noise limit  $h \to 0$  of the law of $\tau_{\Omega^c}$ and of  the principal eigenvalue $-\lambda_{1,h}^L$ of the infinitesimal generator of the diffusion \eqref{eq.langevin} with Dirichlet boundary conditions on $\pa \Omega$.  When  the $\omega$-limit set of each trajectory of the dynamical system $\dot X=\boldsymbol{b}(X)$   
lying entirely in $\overline \Omega$
 is contained in $\Omega$, the limit of $h\ln \mathbb E[\tau_{\Omega^c}]$ when  $h\to 0$   has been studied in~\cite{FrWe}  (see also~\cite{friedman-75-2,martinelli1989small}). When in addition $\boldsymbol{b}\cdot n_\Omega<0$ on $\pa \Omega$ (where $n_\Omega$ is the unit outward normal vector to $\pa\Omega$),  it is proved in~\cite{day-83}  that $\lambda_{1,h}^L \mathbb E[\tau_{\Omega^c}]\to 1$ when $h\to 0$ (see also~\cite{ishii-souganidis-1,ishii-souganidis-2}). We also mention \cite{MaSc,Schuss83}  where  formulas were obtained through formal computations.

 When the  process~\eqref{eq.langevin} is reversible, i.e. when there exists a function $f$ such that $\boldsymbol{b}=-\nabla f$, we refer to~\cite{sugiura2001,HeNi1,DLLN-saddle1,nectoux2017sharp} for sharp asymptotics formulas on $\lambda_{1,h}^L$ or on $\mathbb E[\tau_{\Omega^c}]$ when the  system  does not have  equilibrium points  on $\pa \Omega$, and to \cite{mathieu1995spectra,DoNe2,NectouxCPDE} when it does (see also~\cite{lelievre2022eyring}). 
 When  $\boldsymbol{b}\cdot n_\Omega=0$, the cycling effect  of a two-dimensional randomly perturbed system
 has been studied in \cite{Day92}.
 We refer
 to~\cite{Ber,di-gesu-le-peutrec-lelievre-nectoux-16,di-gesu-lelievre-le-peutrec-nectoux-17} for a  comprehensive review of the literature on this topic.

{\bf Remark.}
{\it For asymptotic estimates of eigenvalues and transition times  in the boundaryless case, we refer to~\cite{holley-kusuoka-stroock-89,miclo-95,BGK,eckhoff-05,BEGK,Berglund_Gentz_MPRF,galves-olivieri-vares-87,HKN,michel2017small}  when   elliptic reversible processes are considered, 
 and to~\cite{bouchet2016generalisation,seoARMA,LePMi20,seoPTRF} when the considered  process is elliptic, non-reversible, and admits the Gibbs measure \eqref{eq.Gibbs} as invariant measure.}\medskip

The purpose of this work is to investigate  the   asymptotic behaviors  when $h\to 0$
 of $\lambda_{1,h}^L$ and 
 of the law and the expected time of $\tau_{\Omega^c}$ 
for non-reversible processes of the form \eqref{eq.langevin}   when
the smooth vector field $\boldsymbol{b}:M\to  \mathbb R^d$ 
decomposes into the pointwise orthogonal sum of a 
smooth gradient field with a vector field (see \eqref{ortho}).

First, we prove in this case the following: when $\Omega$ is roughly a single well (see \eqref{well}) of the potential energy function~$f$   (see  Theorem~\ref{th:main}, which is the first main result of this work):
\begin{itemize}
\item[R1.]  In the limit $h\to 0$,  $\lambda_{1,h}^L \mathbb E[\tau_{\Omega^c}]$ converges to  $1$ and the  law of   $\lambda_{1,h}^L \tau_{\Omega^c}$ converges   to an exponential law of mean $1$,  both exponentially fast and   uniformly  w.r.t.  the   initial conditions~$x$ living in the (relevant)   compacts of $\Omega$. 
The asymptotic behavior of the spectral gap is also investigated.
\end{itemize}

When in addition
  the  Gibbs measure 
\begin{equation}\label{eq.Gibbs}
\text{$\mu_G(dx)=\frac{e^{-\frac 2h f} }{\int_{M} e^{-\frac 2h f}}\, dx$  }
\end{equation}
is  invariant (see \eqref{div}) and
 under an additional assumption on the shape of~$\pa\Omega$  near its lowest energy points (see \eqref{normal}),
we prove that (see  Theorem~\ref{th:main2}, which is the second main result of this work):
\begin{itemize}
\item[R2.] In the limit $h\to 0$, $\lambda_{1,h}^L$, and thus $\mathbb E[\tau_{\Omega^c}]$, satisfy an Eyring-Kramers type formula. 
\end{itemize}

Concerning item R1 above, 
the main novelty compared to the existing literature arises from  the fact that these results are derived  when, simultaneously,   the process \eqref{eq.langevin} is non-reversible and   the dynamical system $\dot X=\boldsymbol{b}(X)$ is allowed to admit   equilibrium points on   $\pa \Omega$\footnote{We mention that in our setting (more precisely under \eqref{ortho}), every $\omega$-limit set is composed of a single equilibrium point, see Section \ref{sec.DS}.}.  The latter situation,
which is known to introduce several technical difficulties \cite{Day87},
is natural for applications  \cite{Schuss83}. 
For instance, this situation  occurs when one is interested in the so-called state-to-state dynamics associated with \eqref{eq.langevin}. In this case, the set $\Omega$, which is associated with a macroscopic state,  is indeed typically  defined as the basin of attraction of some asymptotically stable equilibrium  point $x_0\in M$ for   the dynamical system $\dot X=\boldsymbol{b}(X)$, so that $\pa \Omega$ contains  equilibrium points of   $\dot X=\boldsymbol{b}(X)$.  
We refer for instance to \cite{perez2009accelerated,le2012mathematical,lelievre2016partial,di-gesu-lelievre-le-peutrec-nectoux-17} for more material and references  on state-to-state dynamics.
Let us also mention that the condition \eqref{normal} is automatically satisfied when $\Omega$ is a basin of attraction, see the discussion after \eqref{normal} on this subject.

Finally, concerning  item R2 above, the Eyring-Kramers type  formula we derive for $\lambda_{1,h}^L$ in Theorem~\ref{th:main2}, which leads to  the inverse formula  for     $\mathbb E[\tau_{\Omega^c}]$ according to item R1, is new when considering such  non-reversible processes, whether or not there are   equilibrium points of $\dot X=\boldsymbol{b}(X)$ on $\pa \Omega$. 
It exhibits the  precise   effect of the boundary $\pa \Omega$ on the sharp equivalent as $h\to 0$ of both $\lambda_{1,h}^L$ and  $\mathbb E[\tau_{\Omega^c}]$.

\subsection{Assumptions} 
For $\mu\in \mathbb R$,
we use the notation 
$$\text{$\{f\le \mu\}:=\{x\in M, \ f(x)\le \mu\}$, $ \{f< \mu\}:=\{x\in M,\  f(x)< \mu\}$, and  $\{f=\mu\}:=\{x\in M, \ f(x)=\mu\}$.} $$
Moreover,  for $r>0$ and $y\in M$, $  B(y,r)$ denotes the open ball of radius $r$ centered at $y$ in $M$:
$$ B(y,r):=\{z\in M, \, \vert y-z\vert <r\}.$$
Throughout this work,  we assume that  
there exist a smooth vector field $\boldsymbol{\ell}:M\to \mathbb R^d$ and a smooth Morse function $f:M\to \mathbb R$ such that
 the vector field $\boldsymbol{b}:M\to  \mathbb R^d$ satisfies the following orthogonal decomposition:
\begin{equation}\label{ortho}
\tag{{\bf Ortho}} 
 \boldsymbol{b}(x)=-(\nabla f(x)  +  \boldsymbol{\ell}(x) )
 \quad\text{and}\quad \boldsymbol{\ell}(x) \cdot \nabla f(x) =0  \qquad \text{for every } x\in M.
\end{equation} 
We recall that   a smooth function is  a Morse function if  all its critical points are non degenerate.

 Let us now define 
\begin{equation}
\label{eq.cmin}
\mathbf{C}_{{\rm min}}:=    
\Omega \cap \{f<\min_{\partial \Omega}f  \}.
\end{equation} 
Notice that $\mathbf{C}_{{\rm min}}=\overline\Omega \cap \{f<\min_{\partial \Omega}f  \}$ and that, when $\mathbf{C}_{{\rm min}}$ is nonempty and connected,  it is a connected component of $\{ f<\min_{\partial \Omega}f \}$.

Our  second main assumption  roughly says that $\Omega$ looks like a single well of the potential~$f$:
\begin{equation}\label{well}
\tag{{\bf One-Well}} 
\text{$f:M\to\R$ admits precisely one critical point $x_0$  in $\Omega$ and $\pa \mathbf{C}_{{\rm min}}\cap \partial \Omega \neq \emptyset$.}
\end{equation}

\noindent  
Note that when \eqref{well} holds, $\mathbf{C}_{{\rm min}}$ is nonempty and connected, $x_{0}$ belongs to $\mathbf{C}_{{\rm min}}$, and
\begin{equation}\label{eq.global}
f(x_0)=\min_{x\in \overline \Omega} f(x).
\end{equation}
We refer to Figure \ref{fig:1} for a schematic representation  of   $\mathbf{C}_{{\rm min}}$    when  \eqref{well} holds. 
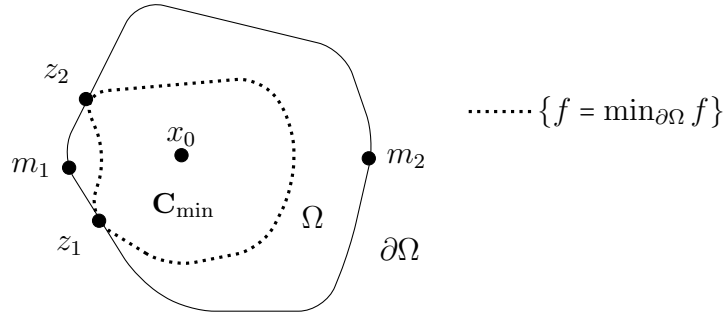
\begin{figure}[h!]
\begin{center}
\begin{tikzpicture}[scale=0.83]
\tikzstyle{vertex}=[draw,circle,fill=black,minimum size=5pt,inner sep=0pt]
\tikzstyle{ball}=[circle, dashed, minimum size=1.3cm, draw]
\tikzstyle{point}=[circle, fill, minimum size=.01cm, draw]
\draw [rounded corners=10pt] (1,0.5) -- (-0.25,2.5) -- (1,5) -- (4.3,4.2) -- (4.8,2.75) -- (4.4,1) -- (4,0) -- (1.8,0) --cycle;
\draw [dotted, very thick,rounded corners=10pt] (1.5,0.7) -- (.2,1.5) -- (0.5,2.5) -- (0.09,3.5) -- (3,3.75) -- (3.5,3) -- (3.5,2) -- (3,1) --cycle; 
\draw [dotted, very thick]   (6.3 ,3.2) -- (7.3 ,3.2);
\draw  (8.9 ,3.2) node[]{$\{f= \min_{\pa \Omega}f\}$}; 
 \draw (1.73,1.7) node[]{{\small $\mathbf{C}_{{\rm min}}$}};
     \draw  (3.8,1.5) node[]{$\Omega$};
    \draw  (5.2,1) node[]{$\pa \Omega$}; 
  \draw (1.7 ,2.8)node[]{$x_0$};
\draw (1.7 ,2.5) node[vertex,label=north west: { }](v){};
\draw (0.38,1.45) node[vertex,label=south west: {$z_1$}](v){};
\draw (0.17,3.4) node[vertex,label=north west: {$z_2$}](v){};
\draw (4.7,2.45) node[vertex,label=  east: {$m_2$}](v){};
\draw (-0.1,2.3) node[vertex,label=  west: {$m_1$}](v){};
\end{tikzpicture}
 \caption{Schematic representation  of   $\mathbf{C}_{{\rm min}}$    when  \eqref{well} holds. On  this figure, $\pa \mathbf{C}_{{\rm min}}\cap \pa \Omega=\{z_1,z_2\}$  and $m_1,m_{2}\in \pa \Omega$ are the local maxima of $f$ in $M$.  } 
 \label{fig:1}
 \end{center}
\end{figure}

The first main result of this work, namely Theorem~\ref{th:main},
only requires the assumptions \eqref{ortho} and \eqref{well}.
Our second main result, namely Theorem~\ref{th:main2}, requires two additional assumptions which are the topic of the rest of this section.
The first one implies the invariance of the Gibbs measure $\mu_G(dx)=\frac{e^{-\frac 2h f} }{\int_{M} e^{-\frac 2h f}}\, dx$
defined in \eqref{eq.Gibbs}:
\begin{equation}
\label{div}
\tag{{\bf Div-free}}
\text{For every $x\in M$,\ \ \   $\div \boldsymbol{\ell} (x)=0$.}
\end{equation}

\noindent
It is well-known that a process solution to an elliptic stochastic differential equation on $M$   with sufficiently smooth coefficients admits a unique invariant probability measure. 
Furthermore, using the standard characterization\footnote{See for instance~\cite[page 259]{varad1980}.} of an  invariant probability  measure with the adjoint   of the operator $-\frac h2 \Delta+\boldsymbol{b}\cdot \nabla$,   the conditions   \eqref{ortho} and \eqref{div} are  necessary and sufficient to ensure that the measure  $\mu_G$ 
is  an (and thus the)  
invariant probability measure of the process~\eqref{eq.langevin} for all $h>0$.

Throughout this work, we say that $z\in M$ is a saddle point of $f$  when $z$ is a critical point of~$f$ of index~$1$, i.e. when the matrix $\Hess f(z)$, which is invertible according to \eqref{ortho}, admits  precisely one negative eigenvalue. 
Our last assumption \eqref{normal} below deals with
the points $z\in \pa \mathbf{C}_{{\rm min}}  \cap \pa \Omega$. These points, which are global minima of~$f|_{\pa \Omega}$,   play a crucial role in the asymptotic equivalents of the mean exit time from $\Omega$
resulting from Theorems~\ref{th:main} and~\ref{th:main2}.
 Let us mention that, according to~\cite[Item (b) in Proposition 12]{DoNe2}, when 
such a $z$ is a critical point of~$f$, it is a saddle point.

For  $x\in M$, we define the  Jacobian matrix
 $$ \mathsf L(x):= \jac   \,  \boldsymbol{\ell} (x).$$
In order to state our last assumption, we need some elements of the following proposition resulting
from \cite[Lemma 1.8]{LePMi20} and \cite[Lemma 1.4]{BoLePMi22} (see also~\cite{landim2018metastability} for a similar result)
on 
the Jacobian matrix of the vector field  $\boldsymbol{b}$ at a saddle point of $f$.

 \begin{lemma} \label{le.lep-michel}
Assume  \eqref{ortho} and let $z\in  M$ be a critical point of $f$ with index $p\in\{0,\dots,d\}$. 
  Then, the matrix $\Hess f(z) + {}^t \mathsf L(z)$ admits  precisely $p$  eigenvalues 
   in $ \{\mathsf z\in\mathbb C, \Re \mathsf z<0 \}$ and $d-p$ eigenvalues
    in $ \{\mathsf z\in\mathbb C, \Re \mathsf z>0 \}$.
  
  When $z$ is a saddle point,   we denote by 
  $\mu(z)$  the  eigenvalue of $\Hess f(z) + {}^t \mathsf L(z)$ in
   $ \{\mathsf z\in\mathbb C, \Re \mathsf z<0 \}$
   and by $\lambda(z)$ the negative eigenvalue of $\Hess f(z) $.
   We have moreover in this case:
  \begin{enumerate}
   \item  
    The eigenvalue  $\mu(z)$   is real, and thus negative.
     \vspace{0.2cm} 
   \item   Let $\xi(z)$ be a real
   unit eigenvector of $\Hess f(z) + {}^t\mathsf L(z)$ associated with $\mu(z)$. Then, the 
   matrix 
  $\Hess f(z) +2\vert \mu(z)\vert\, \xi(z)\xi(z)^t$ is positive definite and
  of determinant $-\det\Hess f(z)$.\vspace{0.2cm}
  \item  It holds $\vert \mu(z)\vert\ge |\lambda(z)|$, with equality if, and only if, $ {}^t\mathsf L(z)\xi(z)=0$. 
    \end{enumerate}
\end{lemma}

Let us now   formulate our last assumption, on the local shape of $f$ near the points of $\pa \mathbf{C}_{{\rm min}}  \cap \pa \Omega$
when \eqref{ortho} holds. In the following,
for any $z\in \ \pa \Omega$,
 $ n _{\Omega}(z)$ denotes the unit outward normal  vector to $\pa\Omega$ at $z$.
\begin{equation}
\label{normal}
\tag{{\bf Normal}}
 \forall z\in \pa \mathbf{C}_{{\rm min}}  \cap \pa \Omega, \ \text{it holds:}
\left\{\begin{array}{l}
\text{when $ \nabla f(z)=0$, $ \xi(z) \in {\rm Span} \, (n _{\Omega}(z))$,}\\[0.2cm]
 \text{when $\nabla f(z)\neq 0$,  $\det\Hess(f|_{\pa \Omega})(z)\neq 0$  and $\boldsymbol{\ell}(z)=0$,}
\end{array}\right.
\end{equation}
where $\xi(z)$ is
an eigenvector
of $\Hess f(z) + {}^t\mathsf L(z)$ associated with its unique
negative eigenvalue, see Lemma~\ref{le.lep-michel}.

We end this section by discussing the geometric consequences of \eqref{normal}.

Let  $z\in \pa \mathbf{C}_{{\rm min}} \cap \pa \Omega$ be such that   $ \nabla f(z)=0$.
When \eqref{normal} holds,  the tangent space $T_z\pa \Omega$ to $\pa \Omega$ at $z$ satisfies $T_z\pa \Omega = z+{\{\xi(z)\}}^{\perp}$.
Since  $\xi(z)$ is
an eigenvector
of $\Hess f(z) + {}^t\mathsf L(z)$ associated with its unique
eigenvalue in $ \{\mathsf z\in\mathbb C, \Re \mathsf z<0 \}$ and,
according to  Lemma~\ref{le.lep-michel},
the $d-1$ remaining eigenvalues of $\Hess f(z) + {}^t\mathsf L(z)$
belong to $ \{\mathsf z\in\mathbb C, \Re \mathsf z>0 \}$,
 it follows 
that the (complexification of the) hyperplane ${\{\xi(z)\}}^{\perp}$ is the sum of the 
generalized eigenspaces of $-\jac\, \boldsymbol{b}(z)= \Hess f(z) + \mathsf L(z)$
corresponding to its eigenvalues in $ \{\mathsf z\in\mathbb C, \Re \mathsf z>0 \}$.
Moreover, it follows from  \cite[Lemma~4.1]{LePMi20} that,   in a neighborhood $\mathcal O_z$ of $z$ in $M$, 
\begin{equation}
\label{eq.OpaOmega}
(\partial \Omega\cap \mathcal O_z)\setminus \{z\}\subset \{f>f(z)\}.   
\end{equation}
In particular, $z$ is a strict global minimum of $f|_{\pa \Omega}$. 
We refer to Figure \ref{fig:f-nearz} for a schematic representation of $\xi(z)$ and $\mathbf{C}_{{\rm min}}$ near  such a point $z$ when \eqref{normal} holds.

Let us also mention here that,
as explained in Section~\ref{sec.DS} below, $ \nabla f(z)=0$ implies that~$z$ is an equilibrium point for   the dynamical system $\dot X=\boldsymbol{b}(X)$, i.e. that $\boldsymbol{b}(z)=0$.
Hence,
from a dynamical point of view, the above discussion simply says that, when 
\eqref{normal} holds: at every $z\in \pa \mathbf{C}_{{\rm min}} \cap \pa \Omega$ such that   $ \nabla f(z)=0$, the boundary $\pa\Omega$
of $\Omega$ is tangent to the stable manifold of   $z$ for   the dynamical system $\dot X=\boldsymbol{b}(X)$, which has dimension $d-1$. We recall that the stable (resp. unstable) manifold of an equilibrium point $z$ is defined as the set of the elements of $ M$ whose trajectories (for the dynamics $\dot X=\boldsymbol{b}(X)$) converge to $z$ in the future (resp. in the past),
and that (the complexification of) its tangent space at $z$
is 
the sum of the 
generalized eigenspaces of $\jac\, \boldsymbol{b}(z) $
corresponding to its eigenvalues in $ \{\mathsf z\in\mathbb C, \Re \mathsf z<0 \}$ (resp. in $ \{\mathsf z\in\mathbb C, \Re \mathsf z>0 \}$).

Let us now consider $z\in \pa \mathbf{C}_{{\rm min}} \cap \pa \Omega$ such that  $ \nabla f(z)\neq 0$. 
Since $z$ is a global minimum of~$f|_{\pa \Omega}$, 
the tangent space $T_z\pa \Omega$  satisfies $T_z\pa \Omega = z+{\{\nabla f(z)\}}^{\perp}$, $\pa_{n}f(z) >0$, and
$\boldsymbol{b}(z)=-\nabla f(z)-\boldsymbol{\ell}(z)$
is inward-pointing.
Thus, according to \eqref{ortho}, the condition $\boldsymbol{\ell}(z)=0$
 in the second part of \eqref{normal} is equivalent to $\boldsymbol{b}(z)\in  {\rm Span} \, (n _{\Omega}(z))$. 
 It is thus in a way the counterpart of the first assumption of \eqref{normal} when $z$ is not an equilibrium point for the dynamics $\dot X=\boldsymbol{b}(X)$,
since it gives the condition for $\boldsymbol{b}(z)$ to be orthogonal to $T_z\pa \Omega$.

In particular, when \eqref{normal} holds, any  $z\in \pa \mathbf{C}_{{\rm min}}\cap \pa \Omega$ is a strict   global minimum of $f|_{\pa \Omega}$,
whether $ \nabla f(z)\neq 0$  or $ \nabla f(z)= 0$. Thus, since $\pa \Omega$ is compact:
\begin{equation}\label{eq.Ccons}
\eqref{normal}  \Rightarrow \text{ Card }(\pa \mathbf{C}_{{\rm min}}\cap \pa \Omega)<+\infty.
\end{equation}

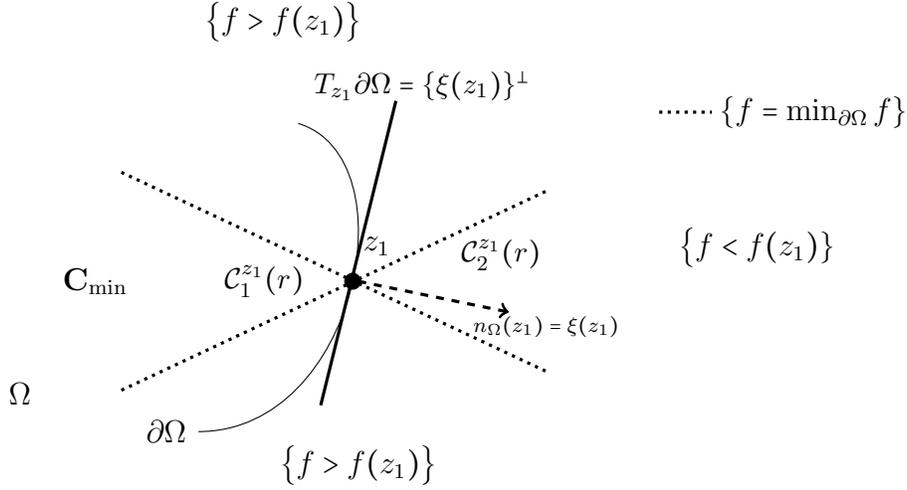
\begin{figure}[h!]
\begin{center}
\begin{tikzpicture}[scale=0.5]
\tikzstyle{vertex}=[draw,circle,fill=black,minimum size=6pt,inner sep=0pt]
\draw[very thick] (-1,-3.3)--(1,4.8); 
   \draw   (-1.6,4.2) ..controls (1.4,3.3)  and   (-0,-4)    .. (-4.25,-4) ;
 \draw (1.7,5.2) node[]{\small {$T_{z_1}\pa \Omega = \{\xi(z_1)\}^{\perp}$}};
\draw (-0.15,0)  node[vertex,label=north east: {}](v){};
  \draw (0.49,0.94)  node[] {$z_1$}; 
   \draw (-7,0) node[]{$\mathbf{C}_{{\rm min}}$};
   \draw (-5.1,-4) node[]{$\pa \Omega$};
  \draw (10.6,0.9) node[]{$ \big\{f<f(z_1)\}$};
      \draw (-2.5,0) node[]{{\small $\mathcal C_1^{z_1}(r)$}};
    \draw (3.8,0.7) node[]{{\small $\mathcal C_2^{z_1}(r)$}};
    \draw (-2,6.9) node[]{$ \big\{f> f(z_1)\big\}$};
     \draw (0,-4.9) node[]{$\big \{f> f(z_1)\big\}$};
   \draw [dotted, very thick]  (-6.3,2.9) to (5,-2.4) ;
   \draw [dotted,  very thick] (-6.3,-2.9) to (5,2.4) ;
      \draw (-9,-3)node[]{$\Omega$};
       \draw [very thick, dashed, ->] (0,0)--(4,-0.82); 
        \draw (5,-1.2)  node[] {\tiny{$n_\Omega(z_1)=\xi(z_1)$}}; 
         \draw (12.1,4.5)node[]{$\{f=\min_{\pa \Omega}f\}$};
          \draw[dotted, very thick] (8,4.5)--(9.4,4.5);
\end{tikzpicture}
\caption{Schematic representation of $\pa \Omega$ near $z_1\in \pa \mathbf{C}_{{\rm min}}\cap \pa \Omega$   
 when \eqref{normal} holds and  $\nabla f(z_1)=0$   (recall that $z_1$ is then a saddle point of $f$).  }  
 \label{fig:f-nearz}
 \end{center}
\end{figure}

\subsection{The deterministic dynamical system}
\label{sec.DS}

We give here  basic properties 
on the $\omega$-limit  sets of the deterministic dynamical system $\dot X=\boldsymbol{b}(X)$ associated with the stochastic differential equation \eqref{eq.langevin}
when \eqref{ortho} holds.

For every $x\in M$, we  denote by~$\varphi_{t}(x)$ the solution on $M$ to the ordinary differential equation   
\begin{equation}\label{eq.flow}
\frac{d}{dt}\varphi_{t}(x)  = \boldsymbol{b}(\varphi_{t}(x)) \text{ with initial condition } \varphi_{0}(x)=x.
   \end{equation}
 Notice that, since $ \boldsymbol{b}$ is (globally) Lipschitz continuous over $M$, such curves are defined globally. 
 
 Let us now describe the $\omega$-limit set of some $x\in M$ for the dynamical system \eqref{eq.flow}. 
This set, denoted by $\omega(x)$, is defined by (see e.g.~\cite[Definition 8.1.1]{wiggins2003introduction})
 $$\omega(x):=\{y\in M, \, \exists (s_n)_{n\in \mathbb N} \in (\mathbb R_+)^{\mathbb N}, \,\lim_{n\to \infty}s_n=+ \infty, \,\lim_{n\to \infty}\varphi_{s_n}(x)=y \}.$$
Let us recall that, for all $x\in M$,  $\omega(x)$ is nonempty, connected,   closed, and invariant under the flow of \eqref{eq.flow} (see e.g. \cite[Proposition 8.1.3]{wiggins2003introduction}). 
Moreover, since $\boldsymbol{\ell}\cdot \nabla f=0$ according to \eqref{ortho}:  for every $x\in M$  and  $t\in\R$,
 \begin{equation}\label{eq.flow-e}
 \frac{d}{dt}f(\varphi_{t}(x))=- \vert \nabla f\vert^2(\varphi_{t}(x))  .
 \end{equation}
 Hence, following the proof of~\cite[Theorem 15.0.3]{wiggins2003introduction}, we have, as for gradient vector fields:  for all $x\in M$,   $
 \omega(x)\subset \{y\in M,\, \nabla f(y) =0\}$. 
 Since the Morse function   $f:M\to \mathbb R$ has a finite number of critical points in $M$   and  $\omega(x)$ is  nonempty and connected:  for all $x\in M$,  there exists a critical point  $y\in M$ of $f$ such that   
$
\omega(x)=\{y\} 
$, so in particular $\lim_{t\to +\infty} \varphi_t(x)=y$.

 Now, recall that an equilibrium point for the dynamical system \eqref{eq.flow} is by definition a point $z\in M$ such that $ \boldsymbol{b}(z)=0$,
that is such that $\omega(z)=\{z\}$.  
It follows that
$$
\{z\in  M, \boldsymbol{b}(z) =0\}\subset \{z\in  M, \nabla f(z)  =0\}.
$$
 Moreover, since $\Hess f$ is invertible at any critical point of $f$, a Taylor expansion of
 $\boldsymbol{\ell} \cdot \nabla f=0$ around such a point shows that $\boldsymbol{\ell}(z)=0$
 whenever $\nabla f(z)  =0$.
 Thus, when \eqref{ortho} holds, we have the equality
$\{z\in  M, \nabla f(z)  =0\}=\{z\in  M, \boldsymbol{b}(z) =0\}$
 and, for all $x\in M$,  there exists  $y\in M$ such that
\begin{equation}\label{eq.incluw}
\omega(x)=\{y\} \subset \{z\in  M, \nabla f(z)  =0\}= \{z\in  M, \boldsymbol{b}(z) =0\}.
 \end{equation}  
With the same reasoning when $t\to-\infty$: for all $x\in M$,  there exist two critical points   $y_{\pm}$ of $f$ such that
\begin{equation}\label{eq.convphi}
\lim_{t\to +\infty} \varphi_t(x)=y_{+}\ \ \text{and}\ \ \lim_{t\to -\infty} \varphi_t(x)=y_{-}.
\end{equation} 

  \begin{definition}
  For every $x\in \Omega$, we set $t_x:=\inf  \{t\ge 0, \ \varphi_{t}(x)\notin \Omega\}>0$. 
  The domain of attraction of $F\subset \Omega$  is  defined by 
\begin{equation}\label{eq.A}
  \mathcal A(F):=\big \{ x\in \Omega, \,t_x=+\infty \text{ and } \omega(x)\subset F \big \}.
\end{equation}
  \end{definition}
  \noindent
Notice that when   \eqref{ortho} and  \eqref{well} hold,~\eqref{eq.flow-e} and \eqref{eq.incluw} imply that 
\begin{equation}\label{eq.Ax0}
\mathbf{C}_{{\rm min}}\subset   \mathcal A( \{x_0\}).
\end{equation}

\subsection{Main results}

We denote by $L^2(\Omega)$ the space of functions which are square integrable on $\Omega$ for the Lebesgue measure on $\Omega$. The associated Sobolev spaces of regularity $k\ge 1$ are denoted by  $H^k(\Omega)$. The  space $H^1_0(\Omega)$ denotes the spaces of functions $w\in H^1(\Omega)$ such that $w=0$ on $\pa \Omega$. 
  We also denote by $L^2_w(\Omega)$ the space of functions which are square integrable on $\Omega$ for the    measure $e^{-\frac 2h f}dx$ on $\Omega$. The notation $w$ indicates that the weight $e^{-\frac 2h f}dx$ appears in the inner product. The associated weighted Sobolev spaces of regularity $k\ge 1$ are denoted by  $H^k_w(\Omega)$.

According to  \eqref{ortho}, it is natural to work in $L^2_w(\Omega)$ to study the spectral properties of (minus) the infinitesimal generator $L_h$ of the process~\eqref{eq.langevin}
with Dirichlet conditions on $\pa\Omega$:
$$ L_h=-\frac h2 \Delta + \nabla f \cdot \nabla +  \boldsymbol{\ell}  \cdot \nabla
\ \ \ \text{ with domain $D(L_h)=H^2_w(\Omega)\cap \{w\in H^1_w(\Omega), w=0 \text{ on }\pa\Omega\}$.}
 $$ 
Its adjoint $L^{*}_h$ on $L^2_w(\Omega)$, whose domain is still $D(L_h)$,   has indeed the rather nice form 
$$
L^{*}_h=-\frac h2 \Delta + \nabla f \cdot \nabla -  \boldsymbol{\ell}  \cdot \nabla-  \div \boldsymbol{\ell}\,.
$$
In particular, when \eqref{div} holds,  $L^{*}_h$ is  $L_{h}$ with $\boldsymbol{\ell}$ replaced by~$-\boldsymbol{\ell}$, and the process \eqref{eq.langevin} is reversible when
$\boldsymbol{\ell}=0$.

 To study the spectral properties of $L_{h}$, we actually use a unitary transformation to work in the flat space $L^2(\Omega)$, where computations such as integrations by parts are easier to perform.  
We  denote by $ \nabla_{f,h}:= h\, e^{-\frac f h}\nabla e^{\frac f h}=h\nabla + \nabla f$ the distorted gradient \textit{\`a  la} Witten and  
\begin{equation}\label{eq.Witten}
 \Delta_{f,h}:= \nabla_{f,h}^*\nabla_{f,h}=-h^2\Delta+ \vert \nabla f\vert^2-h\Delta f
 \end{equation}
 the Witten Laplacian associated with $f$, where adjoints are now taken on $L^2(\Omega)$. 
Let us then define 
\begin{equation}\label{eq.unitary}
P_h:= 2h \, e^{-\frac f h} \,  L_h \, e^{\frac f h} = \Delta_{f,h}+  2\boldsymbol{\ell}\cdot \nabla_{f,h}= 
\Delta_{f,h}+2h\,\boldsymbol{\ell}\cdot \nabla
\end{equation}
with domain  $D(P_h)=H^{2}(\Omega)\cap H^1_0(\Omega)$
 on $L^2(\Omega)$.
According to   \eqref{eq.unitary}, the operators $2h\,L_{h}$  and~$P_{h}$ are unitarily equivalent, 
and thus have the same spectral properties.
In particular, for all $h>0$,  $\lambda \in \sigma(L_h)$ if and only if $ 2h\,\lambda  \in \sigma(P_h)  $, and the algebraic and  geometric multiplicities of $\lambda$ are the same for both $L_h$ and $(2h)^{-1}P_h$.\\

 The following result describes general spectral properties of $(P_h,D(P_h))$, and thus of $(L_h,D(L_h))$,   for every fixed $h>0$.

 \begin{proposition}\label{pr.spectre}
Assume that \eqref{ortho} holds. Then, for every $h>0$:
\begin{itemize}
\item
The operator $P_{h}:D(P_{h})\to L^{2}(\Omega)$ is  maximal quasi-accretive. More precisely, the operator
$P_{h}+h\|\div\boldsymbol{\ell}\|_{\infty}:D(P_{h})\to L^{2}(\Omega)$ is  maximal accretive. 
Furthermore, $P_{h}$  has a compact resolvent and is sectorial. 
\item The adjoint of $P_{h}:D(P_{h})\to L^{2}(\Omega)$
is the operator  
$$P_{h}^*= \Delta_{f,h}- 2\boldsymbol{\ell}\cdot \nabla_{f,h} - 2h\div\boldsymbol{\ell}\  \text{ with domain $D(P_h)$}.$$
It is also maximal quasi-accretive, with a compact resolvent, and sectorial. 
\item
There exists $\Sigma\subset \mathbb C$ such that the  spectra of $P_{h}$ and of $P_{h}^*$ satisfy
$$\sigma(P_{h})=\{\lambda_{1,h}^P\}\, \cup\,\Sigma \ \ \ \text{and}\ \ \   \sigma(P_{h}^*) = \{ \lambda_{1,h}^P \}\, \cup\,\overline  \Sigma,$$
where $\lambda_{1,h}^P\in \mathbb R^*_+$  is simple  (i.e. has algebraic multiplicity $1$) for both $P_{h}$ and $P_{h}^*$ and, 
for every $\lambda \in\Sigma$, $\Re \lambda > \lambda_{1,h}^P$.\\
 Moreover,
 $P_{h}$ (resp. $P_{h}^*$) admits
 an eigenfunction $u_{1,h}^P$  (resp.~$u_{1,h}^{P^*}$) associated with 
   $\lambda_{1,h}^P$  which is positive  within $\Omega$. 
\end{itemize}
\end{proposition}

The proof of Proposition~\ref{pr.spectre} uses standard arguments on elliptic operators with Dirichlet boundary conditions on a smooth  bounded domain. It is proved in the appendix
for the sake of completeness.

The eigenvalue $\lambda_{1,h}^P$ is the so-called principal eigenvalue of $P_h$. According to \eqref{eq.unitary}, the principal eigenvalue $\lambda^L_{1,h}$ of $L_h$ acting on $L^2_w(\Omega)$  thus
satisfies $2h\,\lambda^L_{1,h}= \lambda^P_{1,h}$. 
Moreover, by compacity of the resolvent of $L_{h}$, its spectrum is discrete and can only accumulate at infinity. 
Hence, the sectoriality of $L_{h}$
and the last item of
Proposition~\ref{pr.spectre} imply the existence of a spectral gap for every $h>0$, that is:
$$
\forall h >0\,,\ \exists c_{h}>0\,,\ \ \sigma(L_{h})\cap \left\{\mathsf  z\in\C\,,\ \Re \mathsf  z \in(\lambda_{1,h}^L,\lambda_{1,h}^L+c_{h})\right\}=\emptyset.
$$
Furthermore, the analysis led in Section~\ref{se-Re-Ph} (see Theorem~\ref{th:2}) permits
to specify the behaviour of $\lambda_{1,h}^{L}$ and of this spectral gap with respect to $h$: when $f$
admits  $\mathsf m_0$ local minima in $\Omega$,
there exist $c_{1},c_{2}>0$ and $h_{0}>0$ such that, for every $h\in(0,h_{0}]$,
$L_{h}$ admits $\mathsf m_0$ eigenvalues (counted with multiplicity) in $\{\mathsf  z\in\C,\ |\mathsf  z|\leq e^{-\frac {c_{1}}h}\}$
and its remaining eigenvalues live in $\{\mathsf  z\in\C,\ \Re \mathsf  z\geq  c_{2}\}$.
In particular, when \eqref{well} is also satisfied:
$$
\exists c,h_{0}>0\,,\  \forall h \in(0,h_{0}]\,,\ \ \lambda_{1,h}^L\leq e^{-\frac ch}\ \ \text{and}\ \ 
\sigma(L_{h})\cap \left\{\mathsf  z\in\C\,,\ \Re \mathsf  z \in(\lambda_{1,h}^L,\lambda_{1,h}^L+c)\right\}=\emptyset.
$$

We can now state the two main results of this work.

\begin{theorem}\label{th:main} 
Assume  \eqref{ortho} and   \eqref{well}. 
Let $K$ be a compact subset of   $\mathcal A( \{x_0\})$ (see \eqref{eq.A}). 
Then,  there exist $c>0$ and $h_0>0$ such that  for all $h\in (0,h_0]$:
\begin{enumerate}
\item[a.]  The principal eigenvalue $\lambda_{1,h}^L$ of $L_{h}$ satisfies, 
\begin{equation}\label{eq.info}
\sigma(L_{h})\cap \big\{ \mathsf z\in \mathbb C, \Re \mathsf z\le c\big\}=\{\lambda_{1,h}^L\big\}\  \text{ and }\ 
 \lim_{h\to 0}h\ln \lambda_{1,h}^L= -2\,\big(\min_{\pa \Omega}f -f(x_0)\big).
 \end{equation} 
 
\item[b.] The mean exit time   $\tau_{\Omega^c}$ satisfies, uniformly in $x\in K$,
\begin{equation}\label{eq.ex} 
  \mathbb  E_x[\tau_{\Omega^c}]=\frac{ (1+ O( e^{-\frac{c}{h}}))} {\lambda_{1,h}^L}.
  \end{equation}
\item[c.]  The law of the exit time  $\tau_{\Omega^c}$  satisfies, 
   \begin{equation}
\label{eq.loiP} 
\sup_{t\geq 0\,,\,x\in  K}\, \Big\vert\, \mathbb P_{x}\Big[\tau_{\Omega^c}> \frac{t}{\lambda_{1,h}^L}\Big] - e^{-t}\,\Big \vert \le e^{-\frac ch}. 
\end{equation}
 \end{enumerate}
\end{theorem}

\noindent
Let us make some comments with regard to Theorem~\ref{th:main}:

\begin{itemize}

\item The second statement in~\eqref{eq.info}  is the so-called  Arrhenius law for $\lambda_{1,h}^L$. Together with~\eqref{eq.ex}, it implies the following   Arrhenius law for the mean exit time~$\tau_{\Omega^c}$: 
$$\lim_{h\to 0}h\ln  \mathbb  E_x[\tau_{\Omega^c}]= 2\,\big(\min_{\pa \Omega}f -f(x_0)\big), \ \text{ uniformly in $x\in K$}.$$

\item 
Equation~\eqref{eq.ex}   provides the following    \textit{leveling result} on the mean exit time from $\Omega$: $\mathbb  E_x[\tau_{\Omega^c}]= \mathbb  E_{y}[\tau_{\Omega^c}](1+O(e^{-\frac ch}))$,  uniformly in $x,y$ in the compacts of  $\mathcal A(\mathbf{C}_{{\rm min}})$ (see \eqref{eq.Ax0}). 
As long as   \eqref{ortho} is satisfied,  this leveling result  extends  the one obtained in  \cite[Corollary~1]{day-83} when \eqref{eq.flow} admits equilibrium points  on $\pa \Omega$.  It also extends \cite[Theorem 2]{NectouxCPDE}  when  the underlying process is non-reversible. 
  \vspace{0.2cm}

\item Equation  \eqref{eq.loiP} implies that when $h\to 0$,
 the law of  $\lambda_{1,h}^L\tau_{\Omega^c}$ converges exponentially fast to the exponential law of mean $1$, uniformly  in the compacts of $\mathcal A( \{x_0\})$. Notice that~\eqref{eq.ex}  is not a consequence of \eqref{eq.loiP}.\vspace{0.2cm}

   \item Deriving 
 Theorem~\ref{th:main}
 for all $x\in \mathcal A(\mathbf{C}_{{\rm min}})$ and not only for $x=x_0$ is of real interest  for applications relying on the process \eqref{eq.langevin}. Indeed, ones wants in practice  an estimate on the time this process  remains trapped in the metastable domain $\Omega$. Since  it admits a density with respect  to the   Lebesgue measure  $dx$  on $M$,  the probability that its trajectories   pass through     $x_0$ is zero.     
\end{itemize}

Our second main result states that, under the additional assumptions \eqref{div} and \eqref{normal},
the eigenvalue $\lambda_{1,h}^L$ satisfies an Eyring-Kramers type formula.

\begin{theorem}\label{th:main2} 
Assume \eqref{ortho}, \eqref{well},  \eqref{div}, and \eqref{normal}.
Then, when $h\to 0$, the eigenvalue $\lambda_{1,h}^L$
satisfies the following Eyring-Kramers type formula:
\begin{equation}\label{eq.As-lambda}
\lambda_{1,h}^L=\big (\kappa^L_1\,h^{-\frac12}  +  \kappa_2^L + O(h^{\frac14}))   \, e^{-\frac 2h (\min_{\pa \Omega}f -f(x_0))},
\end{equation} 
 where  \begin{equation}
 \label{eq.Pre1}
 \left\{
    \begin{array}{ll}
         \kappa_1^L &=  \displaystyle{\frac {\sqrt{ \det\Hess f(x_0)} }{\sqrt \pi}  \sum_{\substack{z\in\pa \mathbf{C}_{{\rm min}} \cap \pa \Omega \\ \nabla f(z)\neq 0}}  \,   \frac{\partial_{ n_\Omega}f(z) }{\sqrt{  \det\Hess f_{\vert\partial\Omega}(z)    } }}\\
        \kappa_2^L &=  \displaystyle{ \frac {\sqrt{ \det\Hess f(x_0)} }{2\pi}  \sum_{\substack{z\in\pa \mathbf{C}_{{\rm min}} \cap \pa \Omega \\  \nabla f(z)=0}}  \frac{2\vert \mu(z)\vert }{\sqrt{\vert   \det\Hess f(z)  \vert  }}  }
    \end{array}
\right.,
 \end{equation}
and $\mu(z)$ denotes the negative eigenvalue of $\Hess f(z) + {}^t\mathsf L(z)$ at a saddle point  $z$ of $f$ (see Lemma~\ref{le.lep-michel}).
\end{theorem}

\noindent
Let us now comment the results of Theorem~\ref{th:main2}. 

\begin{itemize}
\item Our analysis actually  shows that the error term $O(h^{\frac14})$ in \eqref{eq.As-lambda}
 is of order $O(h^{\frac12})$  when~$\kappa^L_1=0$ or $\kappa^L_2=0$, see Theorem~\ref{th:VP}.
 It is moreover always of order $O(h^{\frac12})$
when the process is reversible, i.e. when 
$\boldsymbol{\ell}=0$ (see  \cite{DoNe2} or Proposition~\ref{pr.strategy} below).
In addition, whether or not the process is reversible, when the error term in \eqref{eq.As-lambda} is $O(h^{\frac12})$, it is in general optimal 
(see for instance \cite[Remark 25]{DoNe2} for a discussion). 
\vspace{0.2cm} 

\item Let $\lambda_{1,h}^{\Delta}$ be the principal eigenvalue of $-\frac h2\Delta+ \nabla f\cdot\nabla$.
When $\kappa_1^L=0$ (that is when $\nabla f(z)=0$ for every $z\in\pa \mathbf{C}_{{\rm min}}\cap \pa \Omega$), we have:
$$
\frac{\lambda_{1,h}^{\Delta}}{\lambda_{1,h}^L}\sim\frac{\sum\limits_{\substack{z\in\pa \mathbf{C}_{{\rm min}} \cap \pa \Omega }}  \vert \lambda(z)\vert\,\vert   \det\Hess f(z)  \vert^{-\frac12}}{\sum\limits_{\substack{z\in\pa \mathbf{C}_{{\rm min}} \cap \pa \Omega }} \vert \mu(z)\vert \,\vert   \det\Hess f(z)  \vert^{-\frac12}},
$$  
where, for $z\in \pa \mathbf{C}_{{\rm min}} \cap \pa \Omega$, $\lambda(z)$ is the negative eigenvalue of
$\Hess f(z)$.
According to Lemma~\ref{le.lep-michel}, we have $\vert \mu(z)\vert\geq |\lambda(z)|$, with equality if and only if 
${}^t \mathsf L(z)\xi(z)= 0$.
Then, in view of \eqref{eq.ex} and 
of \cite[Theorem~1]{NectouxCPDE}, 
  we accelerate the exit from $\Omega$ by adding, locally around $\pa \mathbf{C}_{{\rm min}} \cap \pa \Omega$,
   a  generic drift term $\boldsymbol{\ell}(X_{t})$ to the reversible process $dX_t=-\nabla f(X_t)dt+\sqrt h \,dB_t$. In the mathematical literature, this acceleration phenomenon has   been   studied  for elliptic non-reversible diffusions  on $\mathbb R^d$ through the analysis  of different quantities:  the rate of convergence to equilibrium at fixed $h>0$ or as $h\to0$, and the  asymptotic equivalents of the  transition times as $h\to 0$, see~\cite{lelievre2013optimal,bouchet2016generalisation,landim2018metastability,LePMi20,seoPTRF} and references therein.
   \vspace{0.2cm}

\item Let us finally mention that combining    the analyses developed  in this work and in \cite{LePMi20,DoNe2,BoLePMi22}, it is clearly possible to extend the results of Theorem \ref{th:main2} to the cases when $f$ has several local minima in $\Omega$  and    $\boldsymbol{\ell}$
admits a classical expansion $\sum_{k\geq 0} h^{k}\boldsymbol{\ell}_{k}$, where  $\boldsymbol{\ell}_{k}$ are smooth vector fields over $M$
such
that  the Gibbs measure~\eqref{eq.Gibbs} remains invariant  for the process~\eqref{eq.langevin} for all $h>0$. 
\end{itemize}

\subsection{Strategy of the proof and organization of the paper}
\label{sec.strategy}
The proof of Theorem~\ref{th:main} relies crucially on the formula   
 \begin{equation}\label{eq.Starting-Point2-intro}
 \frac{1}{\lambda_{1,h}^L}=\mathbb E_{\nu_h}[\tau_{\Omega^c}] =\frac{\int_\Omega  \mathbb E_{x}[\tau_{\Omega^c}] u_{1,h}^{P^*}  \, e^{-\frac f h}}{\int_{\Omega} u_{1,h}^{P^*} e^{-\frac f h}}, \  \text{ where $\nu_h(dx)=\frac{u_{1,h}^{P^*}e^{-\frac f h}}{\int_{\Omega}u_{1,h}^{P^*}e^{-\frac f h}}\ dx$}
 \end{equation} 
 is a  quasi-stationary distribution  for the process \eqref{eq.langevin}  in $\Omega$ (actually it is the  quasi-stationary distribution, see  Section~\ref{eq.secEqx} for more details on $\nu_h$). 
 
 To extract 
$\mathbb E_x[\tau_{\Omega^c}]$ from the integral in \eqref{eq.Starting-Point2-intro}, in order to prove \eqref{eq.ex} for instance, we use a  leveling result on $x\mapsto \mathbb E_x[\tau_{\Omega^c}]$. This is the purpose of Theorem~\ref{th:leveling}, proved in Section~\ref{sec.level} using large deviations techniques. 
Besides, 
 we also need a priori estimates on the principal 
eigenvalue $\lambda_{1,h}^L= {\lambda^P_{1,h}}/{2h}$ of $L_{h}$,
which is the purpose of Theorem~\ref{th:2} in Section~\ref{se-Re-Ph}, relying on the sole assumption  \eqref{ortho}
and proved by semiclassical methods.

We derive  in Section~\ref{sec.SE-P} from these a priori estimates  information on the concentration of the principal 
eigenfunction $u_{1,h}^{P^*}$ of $P_h^*$, see Proposition~\ref{pr.approximation0}.
Afterwards, combining this information with the leveling results on $x\mapsto \mathbb E_x[\tau_{\Omega^c}]$ and the a priori estimates on $\lambda_{1,h}^L$,
we prove Theorem~\ref{th:main} in Section~\ref{eq.secEqx}.

  Finally, when assuming in addition  \eqref{div} and \eqref{normal}, we   prove the sharp asymptotic equivalents on $\lambda_{1,h}^L$ given in Theorem~\ref{th:main2} by constructing a very precise quasi-mode for $P_{h}$. This is done in Section~\ref{sec.VP}, see Theorem~\ref{th:VP}.

\section{Leveling results on the mean exit time from $\Omega$}
\label{sec.level}

The  goal of this section is to prove Theorem~\ref{th:leveling}  below which aims at giving,  when  \eqref{ortho} and \eqref{well}  hold, sharp leveling results on $x\mapsto \mathbb E_x[\tau_{\Omega^c}]$ as well as the limit of $h \ln \mathbb E_x[\tau_{\Omega^c}]$ when $h\to 0$.  To do so, we use techniques from the  large deviations theory. 
This requires 
some care, since these techniques  cannot be used directly on $\Omega$  due to the possible existence of equilibrium points of  $ \boldsymbol{b}$  on $\partial \Omega$ (recall indeed that $\boldsymbol{b}(z)=0$ if and only if $\nabla f(z) =0$, see   \eqref{eq.incluw}).

\subsection{Large deviations and  mean exit time}

In this section we only assume   \eqref{ortho}.

\subsubsection{The quasi-potential on a subset of  $M$}

We now introduce the \textit{quasi-potential} associated with the vector field $ \boldsymbol{b}$ on $\overline D$,
 where $D$ denotes a  smooth bounded subdomain  of $M$ (which is possibly $M$),
 and recall some of its basic properties. 
For $x,y\in \overline D$ and $t_1< t_2 \in \mathbb R$, let us denote by $\mathcal C^{x,y}([t_1,t_2],\overline D)$ the set of continuous curves $\phi: [t_1,t_2]\to \overline D$ such that $\phi(t_1)=x$ and $\phi(t_2)=y$. For $\phi\in \mathcal C^{x,y}([t_1,t_2],\overline D)$, define,
if $\phi$ is absolutely continuous, 
$$S_{t_1,t_2}(\phi)=\frac 12 \int_{t_1}^{t_2}|\dot \phi_s-\boldsymbol{b}(\phi_s)|^2ds\in \mathbb R^{+},$$
where $\dot \phi_s= \frac{d}{ds}\phi_s$,
and, otherwise, $S_{t_1,t_2}(\phi)=+\infty$.   The function  
$$V_D:(x,y)\in \overline D \times \overline D\mapsto  \inf\big\{ S_{0,T}(\phi), \ \, \phi\in \mathcal C^{x,y}([0,T],\overline D) \text{ and } T>0\big\}\in \mathbb R^{+}$$
is the so-called (Freidlin-Wentzell) quasi-potential of the process \eqref{eq.langevin} on $D$. 
  Notice that 
\begin{equation}\label{eq.Vd}
V_D(x,x)=0 \text{ for all } x\in \overline D.
\end{equation}
For every $x,y\in \overline D$ and $S,S'\subset  \overline D$, we  also define
$$
V_D(x,S'):=\inf_{y\in S'} V_D(x,y)\,,\ V_D(S,y):=\inf_{x\in S} V_D(x,y)\,,\ 
\text{and}\ V_D(S,S'):= \inf_{(x,y)\in S\times S'}V_D(x,y) .
$$
In the next lemma, we recall some basic and useful properties of the functional $V_D$.
\begin{lemma}\label{le.Vcontinuous}
One has the following:\\
$\bullet$ $V_D:\overline D \times \overline D\to \mathbb R^{+}$ is  continuous.\\
$\bullet$ Assume that there exists   a subset $S$  of $\overline D$ such that, for any $T\ge 0$ and $\phi \in \mathcal C^{x,y}([0,T],\overline D)$, there exists $t\in [0,T]$ such that $\phi_t\in S$. Then, it holds
$$V_D(x,y)=\inf_{z\in S}[V_D(x,z)+V_D(z,y)].$$
$\bullet$ For every $x\in \overline D$ and every $-\infty \leq t_{-}\leq t_{+}\leq +\infty$
such that the solution $\varphi_{t}(x)$  of \eqref{eq.flow} satisfies $\{\varphi_{t}(x),t\in[t_{-},t_{+}]\}\subset \overline D$,
where $\varphi_{t_{\pm}}(x):=\lim_{t\to \pm\infty}\varphi_{t}(x)$ when $t_{\pm}=\pm\infty$ (see \eqref{eq.convphi}), it holds
$$
V_{D}(\varphi_{t_{-}}(x),\varphi_{t_{+}}(x))=0.
$$
$\bullet$ Let $T>0$ and $G$ be a closed 
nonempty  subset  of $ \mathcal C ([0,T],M)$ (endowed with the uniform convergence topology). Then, the infimum
$$
\inf\big\{ S_{0,T}(\phi), \ \, \phi\in G\big\}
$$
is a minimum. In particular, 
this infimum is strictly positive as soon as 
$G$ does not contain any trajectory of the dynamical system \eqref{eq.flow} defined on $[0,T]$.
 \end{lemma}

 \noindent
 The first item is  a consequence of
  \cite[Lemma~1.1 in Section~1 of Chapter 6]{FrWe} and implies
 the third one, while
 the second item can be proved by straightforward arguments.
For the 
 last one, we refer 
  to  the comments following  the proof of \cite[Theorem 1.1 in Chapter 4]{FrWe}.

  \begin{lemma}\label{le.ff}
  Assume   \eqref{ortho}. 
Then, for all $\phi\in \mathcal C^{x,y}([t_1,t_2],M)$,   $S_{t_1,t_2}(\phi)\ge 2(f(y)-f(x))$.   
 \end{lemma}

 \begin{proof}
Using   \eqref{ortho}, we have, for all   $\phi \in \mathcal C^{x,y}([t_1,t_2],M)$, 
$$
S_{t_1,t_2}(\phi)=\frac 12 \int_{t_1}^{t_2}|\dot \phi_s-(\nabla f(\phi_s)-\boldsymbol{\ell}(\phi_s))|^2ds + 2\int_{t_1}^{t_2}\dot \phi_s\cdot \nabla f(\phi_s) ds 
\ge 2(f(\phi(t_2))-f(\phi(t_1))),
$$
which implies the result.
 \end{proof}
 
\begin{remark}
 \label{re.Vd=}
The proof of Lemma~\ref{le.ff} also leads to the following:
for every $x\in \overline D$ and every $-\infty \leq t_{-}\leq t_{+}\leq +\infty$
such that the solution $\psi_{t}(x)$  of $ \dot X= \nabla f(X)-\boldsymbol{\ell}(X)$ 
with initial condition $\psi_{0}(x)=x$
satisfies $\{\psi_{t}(x),t\in[t_{-},t_{+}]\}\subset \overline D$,
where $\psi_{t_{\pm}}(x):=\lim_{t\to \pm\infty}\psi_{t}(x)$ when $t_{\pm}=\pm\infty$, it holds
$$
V_{D}(\psi_{t_{-}}(x),\psi_{t_{+}}(x))=2(\psi_{t_{+}}(x)- \psi_{t_{-}}(x)).
$$
 \end{remark}

\subsubsection{On the structure of the dynamical system}
To prove Theorem~\ref{th:leveling}  we want to use  \cite[Theorem 5.3 in Chapter 6]{FrWe} 
with a suitable domain $D$
such that
\begin{equation}\label{eq.cond-nabla}
\nabla f\neq 0 \text{ on } \partial D.
\end{equation} 
The construction of $D$ is the purpose of the next section. Before,  we have to check that the conditions stated at the beginning of   \cite[Section 2 in Chapter 6]{FrWe} are satisfied. 
More precisely, we have to check that the exists a finite number of compact subsets $K_1,\ldots,K_l$ of $D$ such that:
\begin{enumerate}
\item[\textbf{(a)}] For any $x\in \overline D$ such that $\varphi_t(x)\in \overline D$ for all $t\ge 0$, it holds $\omega (x)\subset K_q$ for some $q\in \{1,\ldots,l\}$.  
\item[\textbf{(b)}] For all $i\in \{1,\ldots,l\}$ and all  $x,y\in K_i$, $V_D(x,y)=0$.
\item[\textbf{(c)}] If $x\in K_i$ and $y\notin K_i$ ($y\in \overline D$), either $V_D(x,y)>0$ or $V_D(y,x)>0$. 
\end{enumerate}
In the following, we write $\{y\in D, \nabla f(y)=0\}=\{y_1,\ldots, y_l\}$ and we define
\begin{equation}\label{eq.Ki}
K_i=\{y_i\}, \  \forall i\in \{1,\ldots,l\}.
\end{equation}

\begin{lemma}\label{le.WF1} Assume    \eqref{ortho} and \eqref{eq.cond-nabla}.
When the compact sets $K_i$, $i=1,\dots,l$, are defined by  \eqref{eq.Ki}, Conditions {\rm \textbf{(a)}}  and {\rm \textbf{(b)}} above are satisfied.  
\end{lemma}
\begin{proof}
By  \eqref{eq.incluw} and \eqref{eq.cond-nabla},  if $\{\varphi_t(z), t\ge 0\}\subset \overline D$, 
$\omega(x)=\{y\}$ for some critical point $y$ of $f$ in $ D$. 
Thus, Condition \textbf{(a)} holds. In addition, according to \eqref{eq.Vd},   Condition \textbf{(b)} holds.
\end{proof}

Condition {\rm \textbf{(c)}} is the purpose of the next proposition. 
 \begin{proposition}\label{pr.Vxx}
 Assume    \eqref{ortho}
 and \eqref{eq.cond-nabla}.
 When the compact sets $K_i$, $i=1,\dots,l$, are defined by  \eqref{eq.Ki},
Condition {\rm \textbf{(c)}} holds. 
  \end{proposition}
   The following lemma will be useful to prove Proposition~\ref{pr.Vxx}. 
  
    \begin{lemma}\label{le.Sug}
    Assume    \eqref{ortho}.
Let  $z\in \overline D$ be such that $\nabla f(z)\neq0$ and, for some  $T>0$, $\{\varphi_t(z), t\in [0,T]\}\subset \overline D$. Then, for all   $y\in     \overline D\setminus \{\varphi_t(z), t\in [0,T]\}$ satisfying $f(y)>f(\varphi_T(z))$, it holds  $V_D(z,y)>0$.  
 \end{lemma}

 \begin{proof}
 Set $\rho_0=\inf\{|y-\varphi_t(z)|, 0\le t\le T\}>0$. 
 Let $T'\in (0,T]$. 
From the last item of Lemma~\ref{le.Vcontinuous}:
 $$d_{T'}:=\inf\Big\{ S_{0,T'}(\phi), \phi  \in \mathcal C([0,T'],M) \text{ s.t. } \phi_0=z \text{ and }\max_{t\in [0,T']}|\phi_t-\varphi_t(z)|\ge \rho_0/2\Big\}>0.$$
 We then have, for $T'\in (0,T]$ and $\phi\in \mathcal C^{z,y}([0,T'],\overline D)$, $S_{0,T'}(\phi)\ge d_{T'}\ge d_T>0$. Consequently, 
\begin{equation}\label{eq.inf1}
\inf\big\{ S_{0,T'}(\phi), \ \, \phi\in \mathcal C^{z,y}([0,T'],\overline D) \text{ and } T'\in (0,T]\big\}>0.
\end{equation}
 Let us now consider the infimum above when $T'\ge T$. Let $0<T_1<T_2<T$ be such that $f(y)>f(\varphi_{T_1}(z))$. Notice that \eqref{eq.flow-e} and $\nabla f(z)\neq0$ imply
\begin{equation*}
\label{eq.fT1T2}
f(z)>f(\varphi_{T_1}(z))>f(\varphi_{T_2}(z)).
\end{equation*}   
It follows that
$$\varphi(z)|_{[0,T_2]}\notin G_{T_2}^z:=\big  \{ \phi\in \mathcal C ([0,T_2],\overline D), \ \phi_0=z \text{ and, for all } t\in [0,T_2], f(\phi_t)\ge f(\varphi_{T_1}(z))\big  \}$$
and the last item of Lemma~\ref{le.Vcontinuous} then implies that
 $$A:=\inf\big \{ S_{0,T_2}(\phi), \phi\in G_{T_2}^z\big \}>0.$$
 Consider  $T'\ge T$ and $\phi\in \mathcal C^{z,y}([0,T'],\overline D)$. Assume that  $\phi\in G_{T'}^z$. Then $\phi|_{[0,T_2]}\in G_{T_2}^z$, and thus $S_{0,T'}(\phi)\ge S_{0,T_2}(\phi) \ge A$. Assume now that $\phi\notin G_{T'}^z$, i.e. that $f(\phi_t)< f(\varphi_{T_1}(z))$ for some $t\in [0,T']$.  Let $t_1\in (0,T')$ be such that $f(\phi_{t_1})= f(\varphi_{T_1}(z))$. Using Lemma~\ref{le.ff}, it  holds 
  $$S_{0,T'}(\phi) \ge S_{t_1,T'}(\phi) \ge 2(f(\phi_{T'})-f(\phi_{t_1}))=2(f(y)-f(\varphi_{T_1}(z)))>0.$$
   In conclusion, for all $T'\ge T$ and $\phi\in \mathcal C^{z,y}([0,T'],\overline D)$, $S_{0,T'}(\phi)\ge \min(f(y)-f(\varphi_{T_1}(z)),  A  )>0$. 
 Together with 
 \eqref{eq.inf1}, this ends the proof of the lemma.
 \end{proof}
  We are now in position to prove Proposition~\ref{pr.Vxx}.
  \begin{proof}[Proof of Proposition~\ref{pr.Vxx}]
 Let $x\in \overline D$ be  such that $\nabla f(x)=0$,
 so that $x\in D$ according to
  \eqref{eq.cond-nabla}. Let us also consider $y\in \overline D$ such that   $y\neq x$.
According to Lemma~\ref{le.ff}, it suffices to consider the case when $f(x)=f(y)$.
Since $x\in D$ and $f$ admits a   finite number of critical points in $M$, there exists  a  sphere  $C(x,r)=\{w\in M,  |w-x|= r\}\subset D$  of radius $0<r<|x-y|$ 
such that  $\vert \nabla f\vert>0$ on $C(x,r)$. Then, using the two first items of Lemma~\ref{le.Vcontinuous}, there exists $z\in C(x,r)$ such that
   $$V_D(x,y)=\inf_{\xi\in C(x,r)}(V_D(x,\xi)+V_D(\xi,y))=V_D(x,z)+V_D(z,y).$$
If $f(z)<f(x)=f(y)$, then  Lemma~\ref{le.ff} implies $V_D(x,y)\geq V_D(  z,y)\ge 2(f(y)-f(z))>0$. 
Similarly,    
   if $f(z)>f(x)$, then $V_D(x,y)\geq V_D( x,z)\ge 2(f(z)-f(x))>0$.
   Let us lastly consider the case when $f(z)=f(x)$.  
Since $z\in D$ and $\nabla f(z) \neq 0$, there exists $T>0$ such that   
   $$\{\varphi_t(z), t\in [0,T]\}\subset \overline D \ \text{ and, according to  \eqref{eq.flow-e},\ }\  f(z)>f(\varphi_t(z)) \text{ for all } t\in (0,T].$$
 Using $f(z)=f(y)$ and $z\neq y$,  it follows that 
   $y\notin \{\varphi_t(z), t\in [0,T]\}$ and $f(y)>f(\varphi_T(z))$. 
Therefore, according to  Lemma~\ref{le.Sug}, $V_D(z,y)>0$ and thus $V_D(x,y)>0$,
which completes the proof of Proposition~\ref{pr.Vxx}. 
  \end{proof}
  
Following the terminology of \cite{FrWe},  we  say  that  a subset $N\subset   M$ is   \textit{stable}     if, for any $x\in N$ and $y\in M\setminus N$, $V_M(x,y)>0$ (see the lines preceding \cite[Lemma 4.2 in Chapter 6]{FrWe}). We then have:

\begin{lemma}\label{le.stable-compact} Assume    \eqref{ortho}.
For any 
 critical point $x$ of $f$ in $M$, the set $\{x\}$ is  stable (in the sense defined above)  if and only if $x$ is a local minimum of $f$ in $M$. 
\end{lemma}

\begin{proof}
Assume that $x$ is a local minimum of the Morse function $f$ in $M$, and take $y\in M\setminus\{x\}$.
Since  $x$ is a strict minimum, there exists $0<r<|x-y|$
  such that $f>f(x)$ on $ C(x,r)=\{w\in M,  |w-x|= r\}$. Thus, according to  Lemma~\ref{le.Vcontinuous}, there exists $z^{*}\in C(x,r)$ such that
  $$V_M(x,y)=\inf_{z\in C(x,r)}(V_M(x,z)+V_M(z,y))=V_M(x,z^*)+V_M(z^*,y).$$
 Using in addition Lemma~\ref{le.ff}, $V_M(x,z^*)\ge f(z^*)-f(x)>0$ and thus $V_M(x,y)>0$, which implies that $\{x\}$ is  stable. 
 
  Let us now assume that $x$ is   not a local minimum of $f$ in $M$.
Then, according to Lemma~\ref{le.lep-michel}, the dimension of the unstable manifold of   $x$ for   the dynamical system $\dot X=\boldsymbol{b}(X)$
is at least one, and thus
 there exists $z^*\in M\setminus\{x\}$  such that
  $\varphi_t(z^{*})\to x$ when $t\to -\infty$.
It thus follows from the third item of Lemma~\ref{le.Vcontinuous} that $V_{M}(x,z^{*})=0$,
showing that $x$ is  not stable.
\end{proof}

\subsubsection{ Freidlin-Wentzell graphs and mean exit time.} 
Let us first introduce some notation. Let $\mathbf L$ be a finite set and $\mathbf W\subset \mathbf L$. A graph consisting of arrows $m\to n$ (for $m\in \mathbf L\setminus \mathbf W$, $n\in \mathbf L$, and $m\neq n$) is called a $\mathbf W$-graph over $\mathbf L$ (see the beginning of  \cite[Section~3 in Chapter 6]{FrWe})   if:
\begin{enumerate}
\item[$\bullet$] every point $m\in \mathbf L\setminus \mathbf W$ is the  initial point of exactly one arrow, 
  \vspace{0.2cm}
\item[$\bullet$] there are no closed cycles in the graph.
\end{enumerate}
The last condition can be replaced by the following one: for every point $m\in \mathbf L\setminus \mathbf  W$, there exists a sequence of arrows leading from $m$ to some $n\in \mathbf W$. 
The set of $\mathbf W$-graphs over $\mathbf L$ is denoted by $G^{\mathbf L}(\mathbf W)$.  

When  Conditions \textbf{(a)}, \textbf{(b)}, and \textbf{(c)} hold, and when at least one
of the compact subsets $K_1,\ldots , K_l$ of $D$
is stable, we label  these sets
so that 
$K_1,\ldots , K_{p_s}$  are the stable compact sets among $K_1,\ldots , K_l$, where $1\leq {p_s}\le l$. 
In this case,  \cite[Theorem 5.3 in Chapter 6]{FrWe} applies, and implies that,
for every $x\in D$ and uniformly in $x$ in the compact subsets of $D$,
\begin{equation}
\label{pr.WF-E}
\lim_{h\to 0} h\, \ln\,  \mathbb E_x[\tau_{D^c}]\le  W_D, \ \text{ where } \ W_D:=\min_{g\in G^{\{K_1,\ldots , K_{p_s}, \partial D\}}(\{ \partial D\}) }      \sum_{(m\to n)\in g} V_D(m,n).
\end{equation}

\begin{corollary}\label{co.WF}
Assume    \eqref{ortho}, \eqref{eq.cond-nabla}, and  that
$f$ admits $n+1$ local minima $x_{0},x_{1},\dots, x_{n}$ in~$D$, with $n\geq 0$. Then,
for all $x\in D$, and uniformly in $x$ in the compact subsets of $D$,
$$\lim_{h\to 0} h\, \ln\,  \mathbb E_x[\tau_{D^{c}}]\le \sum_{k=0}^{n}V_D(x_{k},\partial D).$$
\end{corollary}
\begin{proof}
Let us define the compact  sets $K_i$, $i=1,\dots,l$, by  \eqref{eq.Ki}.
According to Lemma~\ref{le.WF1} and Proposition~\ref{pr.Vxx}, Conditions {\rm \textbf{(a)}}, {\rm \textbf{(b)}}, and {\rm \textbf{(c)}} are satisfied. Moreover, according to Lemma~\ref{le.stable-compact}, the $\{x_k\}$, $0\leq k\leq n$, are the stable compact sets 
among $K_1,\ldots , K_l$, and thus ${p_s}=n+1$ and $\{K_1,\dots,K_{p_{s}}\}=\{\{x_k\},0\leq k\leq n\}$. 
We conclude by applying \eqref{pr.WF-E} with the graph $(\{x_0\}\to\pa\Omega),\dots,(\{x_n\}\to\pa\Omega)$.
\end{proof}

\subsection{Upper bound on the mean exit time when \eqref{ortho} and  \eqref{well} hold}

\begin{proposition}\label{pr.upper-bound}
Assume that  \eqref{ortho} and  \eqref{well} hold. Then, for every $\beta>0$, there exists $h_{0}>0$
such that, for all $h\in (0,h_0]$,
$$\sup_{x\in \overline \Omega} \mathbb E_x[\tau_{\Omega^c}]\le e^{\frac 2h(\min_{\partial \Omega}f-f(x_0))}e^{\frac \beta h}.$$ 
\end{proposition}

\begin{proof} 
Let us assume that    \eqref{ortho} and  \eqref{well} hold. 
We set
 $$D_\alpha:=\{x\in M, \text{dist}(x,\overline{\Omega})<\alpha\}, \ \alpha >0.$$
For every $\alpha>0$, we have $\overline \Omega\subset D_\alpha$ and $\partial D_\alpha=\{x\in M, \text{dist}(x,\overline{\Omega}) =\alpha\}$. 
In addition, there exists~$\alpha_{0}>0$ such that, for every  $\alpha\in(0,\alpha_{0}]$, $D_\alpha$ is a $\mathcal C^\infty$ subdomain of $M$
 and, since the critical points of $f$  are isolated  in $M$, $\{x\in \overline{D}_\alpha,\nabla f(x)=0\}\subset \overline \Omega$.  In particular, 
$\vert \nabla f\vert>0$ on $\pa D_\alpha$ and
the local minima of $f$ in $D_\alpha$ are its local minimum $x_{0}$ in $\Omega$ and its local minima  $x_{1},\dots,x_{n}$  on $\pa\Omega$. 
Because $\overline \Omega$ is a compact subset of $D_\alpha$, it follows from
Corollary~\ref{co.WF} that for every $\alpha\in(0,\alpha_{0}]$ and $ \epsilon>0$, we have for all $h$ small enough:
 $$\sup_{x\in \overline \Omega} \mathbb E_x[\tau_{\Omega^c}]\le \sup_{x\in \overline \Omega} \mathbb E_x[\tau_{D_\alpha^c}]\le e^{\frac 2h \sum_{k=0}^{n}V_{D_\alpha}(x_0,\partial D_\alpha)}e^{\frac \epsilon h}.$$
In order to prove Proposition~\ref{pr.upper-bound}, it then enough to show that
\begin{equation}\label{eq.Da}
V_{D_\alpha}(x_0,\partial D_\alpha)+\sum_{k=1}^{n}V_{D_\alpha}(x_k,\partial D_\alpha)\le  2(\min_{\partial \Omega}f-f(x_0))+ o_\alpha(1).
\end{equation}
Using the second item of Lemma~\ref{le.Vcontinuous}, we have, for every  $y\in  \partial D_\alpha$ and $z\in \partial \Omega$,
$$V_{D_\alpha}(x_0,\partial D_\alpha)\le V_{D_\alpha}(x_0,y) \le V_{D_\alpha}(x_0,z)+V_{D_\alpha}(z,y) .$$ 
Moreover, according to Lemma~\ref{le.ff}
and to
Remark~\ref{re.Vd=}, for every $z\in \partial \mathbf{C}_{{\rm min}} \cap \pa \Omega$,
$$V_{D_\alpha}(x_0,z)= 2(f(z)-f(x_0)) = 2(\min_{\partial \Omega}f-f(x_0)).$$
Consequently,  for every  $z\in \partial \mathbf{C}_{{\rm min}} \cap \pa \Omega$ and $\alpha>0$ small enough,
$$V_{D_\alpha}(x_0,\partial D_\alpha)\le 2(\min_{\partial \Omega}f-f(x_0))+V_{D_\alpha}(z,\partial D_\alpha)\le 2(\min_{\partial \Omega}f-f(x_0))+\frac12(1 + \|\boldsymbol{b}\|_{\infty})^{2}\alpha \,,$$
where we used the fact that
for every $x\neq y\in M$, $\phi:t\in [0,|y-x|]\mapsto x + \frac{y-x}{|y-x|}t$ satisfies
$
S_{0,|y-x|}(\phi)\leq \frac12(1 + \|\boldsymbol{b}\|_{\infty})^{2}|x-y|$.
The same argument shows that
$V_{D_\alpha}(x_k,\partial D_\alpha)\leq \frac 12(1 + \|\boldsymbol{b}\|_{\infty})^{2}\alpha$
for every $1\le k\le n$ (since $x_{k}\in\pa\Omega$).
This implies \eqref{eq.Da} and thus completes the proof of  Proposition~\ref{pr.upper-bound}.
\end{proof}

\subsection{Leveling results  for $x\mapsto \mathbb E_x[\tau_{\Omega^c}]$ and  \textit{commitor} functions}

The following result provides a  local leveling result for $x\mapsto \mathbb E_x[\tau_{\Omega^c}]$.

\begin{lemma}\label{le.levlocal} 
Assume  \eqref{ortho} and \eqref{well}. Let  $\delta_1>0$ and  $r_h=e^{- \delta_1/ h}$. 
Then,   there exist $h_0>0$ and $c>0$ such that, for all $h\in (0,h_0]$, $\sup_{x\in \bar{  B}(x_0,r_h)}\vert \mathbb E_{x}[\tau_{\Omega^c}]-\mathbb E_{x_0}[\tau_{\Omega^c}]\vert\le e^{-\frac ch} \mathbb E_{x_0}[\tau_{\Omega^c}]$. 
\end{lemma}
  \begin{proof}
Since \eqref{ortho} holds, $\boldsymbol{b}(x_0)=0$ (see \eqref{eq.incluw}). In addition, 
according to Lemma~\ref{le.lep-michel},
the eigenvalues of the matrix $\text{Jac } \boldsymbol{b}(x_0)=-(\Hess f(x_0)+ \mathsf L(x_0))$ all belong to $\{\mathsf z\in \mathbb C, \Re \mathsf z<0\}$ (in particular, $x_0$ is an  asymptotically stable equilibrium point of the dynamical system  \eqref{eq.flow}). The proof then follows the same lines as the one of~\cite[Lemma 3]{NectouxCPDE}.  
  \end{proof}

Denote by $\tau_{\bar B(x_0,r_h)}$ the first time the process~\eqref{eq.langevin} hits  the closed ball $\bar B(x_0,r_h)$, where     we recall that $r_h=e^{- \delta_1/ h}$, $\delta_1>0$. The constant $\delta_1>0$ will be fixed in \eqref{eq.Pcle0-} below. 
We assume that  $h$ is small enough so that $\bar B(x_0,r_h)\subset \mathbf{C}_{\text{min}}$.  The function $$x\mapsto \mathbb P_{x}[\tau_{\bar B(x_0,r_h)}< \tau_{\Omega^c}]$$ is called the  {commitor  function} (or the \textit{equilibrium potential})  between  $\Omega$ and $\bar B(x_0,r_h)$.  The following result provides a (global) leveling result for $x\mapsto \mathbb E_x[\tau_{\Omega^c}]$ in $\mathcal A(\{x_0\})$.

\begin{proposition}\label{pr.levP}
Assume \eqref{ortho} and  \eqref{well}. Then, there exists $\delta_1>0$ such that, for all
    compact subset  $ K$ of $\mathcal A( \{x_0\})$ (see \eqref{eq.A} and \eqref{eq.Ax0}),  there exist $h_0>0$ and $c>0$ such that for all $h\in (0,h_0]$,
$$\sup_{x\in K}\big\vert \mathbb P_{x}[\tau_{\bar B(x_0,r_h)}< \tau_{\Omega^c}]-1\big\vert\le  e^{-\frac ch}.$$
\end{proposition}
 
\begin{remark}
 Applying \cite[Theorem 2]{Day4} with $\Omega=\mathcal A( \{x_0\})$ leads to a slightly weaker version  
of Proposition \ref{pr.levP}, where  $\delta_1>0$ depends on $K$. 
 \end{remark}

\begin{proof}
For $\eta \in (0,\min_{\pa \Omega}f-f(x_0))$,  set 
\begin{equation}\label{eq.Ceta}
\mathbf{C}_{\text{min}}(\eta) := \mathbf{C}_{\text{min}}\cap \{f<\min_{\pa \Omega}f-\eta\}=\{x\in \Omega, f(x)<\min_{\pa \Omega}f-\eta\}.
\end{equation} 
The set   $\mathbf{C}_{\text{min}}(\eta)$  is open,  smooth  (since $\nabla f\neq 0$  on $\pa \mathbf{C}_{\text{min}}(\eta)$), and is the connected component of $\{f<\min_{\pa \Omega}f-\eta\}$ containing $x_0$ (see for instance~\cite[Proposition 18]{DLLN-saddle1}).  Recall also that $x_0$ is an  asymptotically stable equilibrium point of the dynamical system  \eqref{eq.flow}.
Moreover, \eqref{eq.flow-e} implies that $\varphi_t(x)\in\overline{\mathbf{C}}_{\text{min}}(\eta) $
for all $x\in \overline{\mathbf{C}}_{\text{min}}(\eta)$ and $t\in\R^{+}$, and thus
that $\lim_{t\to +\infty}\varphi_t(x)= x_0$ since $x_{0}$ is the  unique critical point of $f$  in $\overline{\mathbf{C}}_{\text{min} }(\eta)$
 (see indeed \eqref{eq.convphi}).
    
Fix  now
  \begin{equation}\label{eq.eta0}
  \eta_0\in (0,\min_{\pa \Omega}f-f(x_0)) \text{ and   }   \eta_*\in (\eta_0,\min_{\pa \Omega}f-f(x_0)).
     \end{equation}
  It holds  $\overline{\mathbf{C}}_{\text{min}}(\eta_*) \subset \mathbf{C}_{\text{min}}(\eta_0)$. In the following $h>0$ is small enough so that   $\bar B(x_0,r_h) \subset {\mathbf{C}_{\text{min}}(\eta_*)}$, where we recall that $r_h=e^{-\delta_1/h}$. 
According to \cite[Theorem 2]{Day4},  there exist $\delta_1>0$ (which is now kept fixed), $h_0>0$,  and $c>0$ such that for all $h \in (0,h_0]$:
  \begin{equation}\label{eq.Pcle0-}
  \sup_{y\in \overline{\mathbf{C}}_{\text{min}}(\eta_*)}\mathbb P_y[\tau_{\mathbf{C}^c_{\text{min}}(\eta_0)}\le  \tau_{ \bar B(x_0,r_h) } ]\le   e^{-\frac ch}.
   \end{equation}
    Since the trajectories of the process~\eqref{eq.langevin} are continuous, one has
  $\{\tau_{\Omega^c}<  \tau_{ \bar B(x_0,r_h)}\}\subset \{\tau_{\mathbf{C}^c_{\text{min}}(\eta_0)}< \tau_{ \bar B(x_0,r_h)}\}$  
     for all $y\in   \overline{\mathbf{C}}_{\text{min}}(\eta_*)$ when $X_0=y$, 
  so that (using also $\{\tau_{\Omega^c}=  \tau_{ \bar B(x_0,r_h) }\}=\emptyset$):
  \begin{equation}\label{eq.Pcle}
 \sup_{y\in \overline{\mathbf{C}}_{\text{min}}(\eta_*)}\mathbb P_y[\tau_{\Omega^c}\le  \tau_{ \bar B(x_0,r_h) }]\le   e^{-\frac ch},
 \end{equation}
   which proves the proposition when $K=\overline{\mathbf{C}}_{\text{min} }(\eta_*)$. Let us now consider the case when $K\subset \mathcal A( \{x_0\})$. In view of \eqref{eq.Pcle}, it is enough to treat the case when $K\subset \Omega\setminus \mathbf{C}_{\text{min}}(\eta_*)$.  Pick $K\subset \Omega\setminus \mathbf{C}_{\text{min}}(\eta_*)$ with $K\subset \mathcal A( \{x_0\})$. Recall that this implies that for all $x\in K$, $\varphi_{t}(x)\in \Omega$ for all $t\ge 0$ and $\lim_{t\to+\infty}\varphi_t(x)= x_0$. 
Then, there exists $T_K> 0$ such that $\varphi_{T_K}(x)\in {\mathbf{C}}_{\text{min}}(\eta_*)$ for all $x\in K$. The set $\{\varphi_{T_K}(x), x\in K\}$ is a compact subset of  the open set ${\mathbf{C}}_{\text{min}}(\eta_*)$ and the compact subset $ \{\varphi_{t}(x), (x,t)\in K\times [0,T_K]\}$  of $\Omega$ does not contain 
$x_0 \notin K$.  
We can thus consider $\delta>0$ small enough such that:
   \begin{enumerate}
   \item[C1.]  $\{\varphi_{T_K}(x)+z, x\in K \text{ and } |z|\le \delta\}\subset   {\mathbf{C}}_{\text{min}}(\eta_*)$,
   \item[C2.] $x_0\notin K_{T_{K},\delta}:=\{\varphi_{t}(x)+z, (x,t)\in K\times [0,T_K] \text{ and } |z|\le \delta\}$.  
   \end{enumerate}
By item C2  above, for any   $h$ small enough, 
$\bar B(x_0,r_h) \cap  K_{T_{K},\delta}=\emptyset$.
Then,  for all $x\in K$, if $X_0=x$ and $\sup_{t\in [0,T_K]}|X_t-\varphi_t(x)|\le \delta$:
\begin{equation}
\label{eq.condM}
 T_K<\tau_{ \bar B(x_0,r_h) }.
\end{equation} 
Moreover, according to~\cite[Lemma 1]{day-83} and its note,  since $M$ is compact, there exists $c'>0$ such that for all $h$ small enough:
\begin{equation}
\label{eq.condMVD}
 \sup_{x\in M}\mathbb P_x\Big [\sup_{t\in [0,T_K]}|X_t-\varphi_t(x)|> \delta\Big ]\le   e^{-\frac{c'}h}.
 \end{equation} 
 On the other, by item C1 above, if $X_0=x\in K$ and  $  \sup_{t\in [0,T_K]}|X_t-\varphi_t(x)|\le  \delta$, it holds $X_{T_K}\in    {\mathbf{C}}_{\text{min}}(\eta_*)$. 
Then, for all $x\in K$, using  the   Markov property and \eqref{eq.condM}, we have  
\begin{align*}
 \mathbb P_x\Big [ \tau_{ \bar B(x_0,r_h) }<\tau_{\Omega^c}, \, \sup_{t\in [0,T_K]}|X_t-\varphi_t(x)|\le  \delta\Big ]&=  \mathbb E_x\Big[  \mathbb E_{X_{T_K}} \big [\textbf{1}_{\tau_{ \bar B(x_0,r_h) }<\tau_{\Omega^c}}\big] \mathbf{1}_{ \sup_{t\in [0,T_K]}|X_t-\varphi_t(x)|\le  \delta}\Big ]\\
 &\geq (1-e^{-\frac ch})\mathbb P_x\Big[ \,  { \sup_{t\in [0,T_K]}|X_t-\varphi_t(x)|\le  \delta}\Big ]\\
 &\geq(1-e^{-\frac ch})(1- e^{-\frac{c'}h}),
 \end{align*}
 where we used  respectively  \eqref{eq.Pcle} and \eqref{eq.condMVD}  at the second and third equalities.  In conclusion, we  have proved that for some $c>0$ and every $h$ small enough, 
   $ \sup_{x\in K}|\mathbb P_x[  \tau_{ \bar B(x_0,r_h) }<\tau_{\Omega^c}]-1|\le   e^{-\frac ch}$,
   which completes  the proof of Proposition~\ref{pr.levP}. 
\end{proof}

\begin{proposition}\label{pr.Ewedge}
Assume \eqref{ortho} and \eqref{well}.
Then, for every $\eta_*\in (0,\min_{\pa \Omega}f-f(x_0))$,
there exist $h_0>0$ and $c>0$ such that, for all $h\in (0,h_0]$,
$$
\sup_{x\in \overline{\mathbf{C}}_{ {\rm min}}(\eta_*) } \,  \mathbb E_x[\tau_{\bar B(x_0,r_h)} \wedge \tau_{\Omega^c}] \le  e^{\frac 2h (\min_{\pa \Omega}f-f(x_0))}e^{-\frac ch},
$$
where $\mathbf{C}_{ {\rm min}}(\eta_*)$ is defined in \eqref{eq.Ceta}.
\end{proposition}

\begin{proof}
  The proof of Proposition \ref{pr.Ewedge} is inspired by the one of~\cite[Lemma 6]{day-83}. 
Take $\eta_{0}\in(0,\eta^{*})$.
For ease of notation, we set 
$$K:=\overline{\mathbf{C}}_{\text{min}}(\eta_*)\ \ \text{and}\  \ D':=\mathbf{C}_{\text{min}}(\eta_0).$$
Recall that
$K\subset D'\subset \mathcal A(\{x_0\})$  and
assume that $h>0$ is small enough so that   $\bar B(x_0,r_h) \subset \text{int}\,  K$ (see \eqref{eq.eta0} and the lines below). According 
to~\cite[Theorems 3.1 and 4.1 in Chapter 4]{FrWe} (note that  $n_{ \mathbf{C}_{\text{min}}(\eta)}=\nabla f /|\nabla f |$ and then, using \eqref{ortho},  
$\boldsymbol{b}\cdot n_{ \mathbf{C}_{\text{min}}(\eta)}>0 \text{ on } \pa \mathbf{C}_{\text{min}}(\eta)$),
 we have  uniformly in $y$ in the compacts of $D'$:
\begin{equation}\label{eq.WFE}
\lim_{h\to 0} h \ln   \mathbb  E_{y}[\tau_{D'^{c}}]=   2(\min_{\pa \Omega}f-f(x_0)-\eta_0).
\end{equation}
In particular, for every $\beta >0$ and every $h$ small enough,
\begin{equation}\label{eq.Ah'}
  A_h^{D'}:=\sup_{y\in K} \mathbb E_y[\tau_{D'^{c}}]\le e^{\frac 2h(\min_{\pa \Omega}f-f(x_0)-\eta_0) }e^{\frac \beta h}.
\end{equation}
Similarly,
according to Proposition~\ref{pr.upper-bound}, it holds for every $\beta>0$ and every $h$ small enough, 
\begin{equation}\label{eq.Ah}
A_h^\Omega:=\sup_{x\in \overline{\Omega}} \mathbb E_x[ \tau_{\Omega^c}]   \le e^{\frac 2h(\min_{\pa \Omega}f-f(x_0)) }e^{\frac \beta h}.
\end{equation}
Besides, using the strong Markov property, we have  for all $x\in K$:
\begin{equation}\label{eq.Exw1}
\mathbb E_x[ \tau_{\Omega^c}]= \mathbb E_x[\tau_{\bar B(x_0,r_h)} \wedge \tau_{\Omega^c}] +  \mathbb E_{x}  \big [  \mathbf 1_{\tau_{\Omega^c}>\tau_{\bar B(x_0,r_h)}}\mathbb E_{X_{\tau_{\bar B(x_0,r_h)} }}[\tau_{\Omega^c}]\big].
\end{equation}
In addition, by continuity of the trajectories of the process \eqref{eq.langevin}, we have $\tau_{D'^{c}}< \tau_{\Omega^c}$ when $X_0=y \in \bar B(x_0,r_h)$. Thus, using the strong Markov property,
\begin{equation}\label{eq.Exw2}
\mathbb E_{X_{\tau_{\bar B(x_0,r_h)} }}[\tau_{\Omega^c}]\ge \mathbb E_{X_{\tau_{\bar B(x_0,r_h)} }}[\tau_{\Omega^c} -\tau_{D'^{c}}]= \mathbb E_{X_{\tau_{\bar B(x_0,r_h)} }}\big [ \mathbb E_{X_{\tau_{D'^{c}}}}[\tau_{\Omega^c}] \big ].
\end{equation}
For $x\in  D'$, let $\mu^h_x$ be the hitting distribution on $\pa D'$ for the process \eqref{eq.langevin} when $X_0=x$, i.e.: 
  \begin{equation}\label{eq.Law1}
  \mu^h_x(B)=\mathbb P_x[ X_{\tau_{ D'^c}} \in B], \text{ for every  Borel subset $B$ of $\pa D'$}.
  \end{equation}
The properties of $D'$ listed just after \eqref{eq.Ceta} allow us to use~\cite[Theorem 1]{Day4} (see also Eq.~(5.1) there), leading to  $\Vert \mu^h_x-\mu^h_y\Vert\le e^{-\frac ch}$ uniformly in $x,y\in K$ (where $\Vert\cdot\Vert$ is the  total variation distance).  Using this and \eqref{eq.Exw2} with $y=X_{\tau_{\bar B(x_0,r_h)} }$, we deduce from \eqref{eq.Exw1} that  for all $x\in K$:
\begin{align*}
\mathbb E_x[ \tau_{\Omega^c}]
&\ge
\mathbb E_x[\tau_{\bar B(x_0,r_h)} \wedge \tau_{\Omega^c}] +  \mathbb E_{x}  \Big [  \mathbf 1_{\tau_{\Omega^c}>\tau_{\bar B(x_0,r_h)}}\mathbb E_{X_{\tau_{\bar B(x_0,r_h)} }}\big [ \mathbb E_{X_{\tau_{D'^{c}}}}[\tau_{\Omega^c}] \big ]\Big ]\\
&\ge \mathbb E_x[\tau_{\bar B(x_0,r_h)} \wedge \tau_{\Omega^c}] + \mathbb E_{x}  \Big [  \mathbf 1_{\tau_{\Omega^c}>\tau_{\bar B(x_0,r_h)}}\Big( \mathbb E_x\big [ \mathbb E_{X_{\tau_{D'^{c}}}}[\tau_{\Omega^c}] \big ] 
-A_h^\Omega e^{-\frac ch}\Big)\Big]\\
&\ge \mathbb E_x[\tau_{\bar B(x_0,r_h)} \wedge \tau_{\Omega^c}] + \mathbb P_{x}   [  \tau_{\Omega^c}>\tau_{\bar B(x_0,r_h)}] \, \mathbb E_x\big [ \mathbb E_{X_{\tau_{D'^{c}}}}[\tau_{\Omega^c}] \big ] 
-A_h^\Omega e^{-\frac ch}.
\end{align*}
On the other hand, according to the strong Markov property,
$\mathbb E_x[ \tau_{\Omega^c}]= \mathbb E_x[\tau_{ D'^c}] +  \mathbb E_{x}    [\,    \mathbb E_{X_{\tau_{ D'^c}}}[\tau_{\Omega^c}] \,  ]$
 for all $x\in K$.  
It follows that for all $x\in K$,
\begin{align*}
 \mathbb E_x[\tau_{\bar B(x_0,r_h)} \wedge \tau_{\Omega^c}] &\le   (1-\mathbb P_{x}   [  \tau_{\Omega^c}>\tau_{\bar B(x_0,r_h)}]) \, \mathbb E_{x}    [\,    \mathbb E_{X_{\tau_{ D'^c}}}[\tau_{\Omega^c}] \,  ] 
 +A_h^\Omega e^{-\frac ch}
 + \mathbb E_x[\tau_{ D'^c}]\\
 &\le  (1-\mathbb P_{x}   [  \tau_{\Omega^c}>\tau_{\bar B(x_0,r_h)}]) \, A_h^\Omega 
 +A_h^\Omega e^{-\frac ch}+ A_h^{D'},
 \end{align*}
which implies Proposition~\ref{pr.Ewedge}, using  \eqref{eq.Ah'}, \eqref{eq.Ah}, and Proposition \ref{pr.levP} (with $K=\overline{\mathbf{C}}_{\text{min}}(\eta_*)$). \end{proof}

\begin{theorem}\label{th:leveling}
Assume  \eqref{ortho} and \eqref{well}. Let $ K$ a compact subset of $\mathcal A( \{x_0\})$ (see \eqref{eq.A} and \eqref{eq.Ax0}).
Then, there exist $h_0>0$ and $c>0$ such that, for all $h\in (0,h_0]$ and   uniformly in $x\in K$: 
 $$\mathbb  E_x[\tau_{\Omega^c}]= \mathbb  E_{x_0}[\tau_{\Omega^c}](1+O(e^{-\frac ch})) \text{ and } \lim_{h\to 0} h \ln   \mathbb  E_{x}[\tau_{\Omega^c}]=  2(\min_{\pa \Omega}f-f(x_0)).  $$
\end{theorem}

  \begin{proof}
  First of all, according to Proposition \ref{pr.upper-bound}, \eqref{eq.WFE}, and to the fact that $ \mathbb  E_{y}[\tau_{D'^{c}}]\le  \mathbb  E_{y}[\tau_{\Omega^c}]$ for all $y\in D'$, we have,  uniformly in $y$ in the compacts of $D'$:
  \begin{equation}\label{eq.WFE2}
\lim_{h\to 0} h \ln   \mathbb  E_{y}[\tau_{\Omega^c}]=   2(\min_{\pa \Omega}f-f(x_0)).
\end{equation}
 Let $K$ be a compact subset of $\mathcal A( \{x_0\})$. Assume first that $K=\overline{\mathbf{C}}_{\text{min}}(\eta_*)$  (see \eqref{eq.Ceta} and \eqref{eq.eta0}). 
Using \eqref{eq.Exw1}, Lemma \ref{le.levlocal}, and Propositions~\ref{pr.Ewedge} and~\ref{pr.levP},
we have uniformly in  $x\in  K$:
 $$\mathbb E_x[ \tau_{\Omega^c}]= \mathbb E_{x_0}[ \tau_{\Omega^c}](1+O(e^{-\frac ch}))+ O( e^{\frac 2h (\min_{\pa \Omega}f-f(x_0))}e^{-\frac ch}).$$
Using in addition  \eqref{eq.WFE2} with $y=x_0\in D'$, we deduce that  for some $c>0$ and    uniformly in $x\in K=\overline{\mathbf{C}}_{\text{min}}(\eta_*)$, it holds for every $h$ small enough:
\begin{equation}\label{eq.Lev}
\mathbb E_x[ \tau_{\Omega^c}]= \mathbb E_{x_0}[ \tau_{\Omega^c}](1+O(e^{-\frac ch})).
\end{equation}
 This proves Theorem \ref{th:leveling} when $K=\overline{\mathbf{C}}_{\text{min}}(\eta_*)$. Let us now consider the general case $K\subset \mathcal A( \{x_0\})$. 
Let  $T_K\ge  0$ be such that $\varphi_{T_K}(x)\in {\mathbf{C}}_{\text{min}}(\eta_*)$ for all $x\in K$, and
take $\delta>0$ small enough so that:
   \begin{itemize}
  \item[$\bullet$] $\{\varphi_{t}(x)+z, (x,t)\in K\times [0,T_K] \text{ and } |z|\le \delta\}
  \subset \Omega$,
   \item[$\bullet$]  $\{\varphi_{T_K}(x)+z, x\in K \text{ and } |z|\le \delta\}\subset   {\mathbf{C}}_{\text{min}}(\eta_*)$. 
   \end{itemize}
These two conditions imply that for all $x\in K$, if $X_0=x$ and $\sup_{t\in [0,T_K]}|X_t-\varphi_t(x)|\le \delta$:
\begin{equation}
\label{eq.condOmega}
 T_K<\tau_{\Omega^c} \text{ and } X_{T_K}\in {\mathbf{C}}_{\text{min}}(\eta_*). 
\end{equation}
From the Markov property, \eqref{eq.condOmega},~\eqref{eq.condMVD}, and~\eqref{eq.Lev}, we
have  uniformly in $x\in K$: 
\begin{align*}
\mathbb E_x[ \tau_{\Omega^c}  \, \mathbf 1_{\sup_{t\in [0,T_K]}|X_t-\varphi_t(x)|\le \delta}]&=T_K \mathbb P_x\big [  \sup_{t\in [0,T_K]}|X_t-\varphi_t(x)|\le \delta\big ]\\
&\quad  + \mathbb E_x\big [ \mathbb E_{ X_{T_K}} [\tau_{\Omega^c}]  \, \mathbf 1_{\sup_{t\in [0,T_K]}|X_t-\varphi_t(x)|\le \delta}\big ]\\
&= T_K(1+O(e^{-\frac ch})) + \mathbb E_{x_0}[ \tau_{\Omega^c}](1+O(e^{-\frac ch}))
\end{align*}
and
\begin{align*}
\mathbb E_x[ \tau_{\Omega^c}  \, \mathbf 1_{\sup_{t\in [0,T_K]}|X_t-\varphi_t(x)|> \delta} \, \mathbf 1_{T_K<\tau_{\Omega^c}}]&=T_K \mathbb P_x[  \sup_{t\in [0,T_K]}|X_t-\varphi_t(x)|> \delta,T_K<\tau_{\Omega^c}]\\
&\quad  + \mathbb E_x\big [ \mathbb E_{ X_{T_K}} [\tau_{\Omega^c}]  \, \mathbf 1_{\sup_{t\in [0,T_K]}|X_t-\varphi_t(x)|> \delta} \, \mathbf 1_{T_K<\tau_{\Omega^c}}\big ]\\
&= T_K O(e^{-\frac ch}) + \mathbb E_{x_0}[ \tau_{\Omega^c}] \, O(e^{-\frac ch}).  
\end{align*}
On the other hand,   using \eqref{eq.condMVD}, it holds for every $x\in K$:
\begin{align*}
\mathbb E_x[ \tau_{\Omega^c} \,  \mathbf 1_{\sup_{t\in [0,T_K]}|X_t-\varphi_t(x)|> \delta} \, \mathbf 1_{\tau_{\Omega^c}\le T_K}]\le T_Ke^{-\frac ch}.
 \end{align*}
Combining the three previous estimates leads to $\mathbb E_x[ \tau_{\Omega^c} ]=\mathbb E_{x_0}[ \tau_{\Omega^c}](1+O(e^{-\frac ch}))$ for all $h$ small enough, uniformly in $x\in K$. This ends the proof of Theorem~\ref{th:leveling}. 
  \end{proof}

\section{Spectral analysis of $\Re(P_{h})$ and of $P_{h}$}
\label{se-Re-Ph}

Recall that we assume   \eqref{ortho}  throughout this work.

\subsection{Analysis of the real part of $P_{h}$}
\label{sub.RePh}

This section is devoted to a preliminary spectral analysis of the operator (see Proposition~\ref{pr.spectre})
$$\Re(P_{h}):=\frac12(P_{h}+P^{*}_{h})= \Delta_{f,h} + 2 \,\boldsymbol{\ell}\cdot\nabla f -h \div \boldsymbol{\ell} = \Delta_{f,h} -h \div \boldsymbol{\ell}$$ 
with domain $D(\Re(P_{h}))=H^2(\Omega)\cap H^1_0(\Omega)=D(P_{h})=D(P_{h}^{*})$.
This operator is self-adjoint with a compact resolvent and
 is the Friedrichs extension of  the closed quadratic form
\begin{equation}
\label{eq.fried}
u\in H^1_0(\Omega)\mapsto \int_{\Omega} \vert \nabla_{f,h} u\vert^2 -h \int_{\Omega}(\div \boldsymbol{\ell})\,|u|^{2}\,.
\end{equation}
It is consequently bounded from below
by $-h \|\div \boldsymbol{\ell}\|_{L^{\infty}(\Omega)}$, and hence
$$\sigma(\Re(P_{h}))\subset [-h \|\div \boldsymbol{\ell}\|_{L^{\infty}(\Omega)},+\infty).$$

When $\div\boldsymbol{\ell}=0$, the operator $\Re(P_{h})$ is nothing  but  
the Witten Laplacian  $\Delta_{f,h}$ (see \eqref{eq.Witten}) with domain $D(\Delta_{f,h})=H^2(\Omega)\cap H^1_0(\Omega)$ and
 is in particular  positive.
Let us now  define  
\begin{equation}\label{eq.m0}
\mathsf U_0 =   \{x\in\Omega,\  x \text{ is a  local minimum of $f$ }\} \ \text{ and }\  \mathsf m_0:= {\rm Card }  ( \mathsf U_0  )<+\infty \footnote{We recall that $f$ has a finite number of critical point in $M$ by \eqref{ortho}.}. 
\end{equation} 
Then, according to~\cite[Theorem 1]{DoNe2}, there exist $c_0>0$ and $h_0>0$ such that for all $h\in (0,h_0]$: 
\begin{equation}\label{eq.energy2}
\dim \Ran \, \pi_{[0,c_0h]}\big (\Delta_{f,h}\big ) =\ \mathsf  m_0, 
\end{equation}
where, 
for a Borel set $I\subset \mathbb R$, $\pi_{I}(\Delta_{f,h}\big )$
denotes the spectral projector associated with 
$\Delta_{f,h}$
and~$ I$. For ease of notation, we  set
\begin{equation}\label{eq.Pidelta}
\pi_h^\Delta :=    \pi_{[0,c_0h]}\big (\Delta_{f,h}\big ).
 \end{equation}
Moreover, the $\mathsf m_0$ eigenvalues of $\Delta_{f,h}$
in $[0,c_0h]$  are exponentially small in the limit $h\to 0$,
i.e. there exists $c>0$ such that for every $h>0$ small enough, 
\begin{equation}
\label{eq.Witt-expo}
\sigma(\Delta_{f,h})\cap [0,c_0h]\subset [0,e^{-\frac ch}].
\end{equation}

Additionally, we can apply \cite[Lemma 3.1]{LePMi20} since  \eqref{ortho} holds:
for every critical point $\u\in M$ of $f$, there exists a smooth map
$J$ defined around $\u$ and with values in $\mathcal M_{d}(\mathbb R)$
such that $J(\u)$ is antisymmetric and $\boldsymbol{\ell}(x)=J(x)\nabla f(x)$
around $\u$. It follows that 
\begin{align*}\div\boldsymbol{\ell}(\u)=\tr\big(J(\u)\Hess f(\u)\big)&=\tr\big(\Hess f(\u)J(\u)\big)\\
&= \tr\big({}^t\big(\Hess f(\u)J(\u)\big)\big)=-\tr\big(J(\u)\Hess f(\u)\big),
\end{align*}
and hence:
\begin{equation}\label{eq.div0}
\text{for every critical point $\u\in M$ of $f$\,,}\ \ \div\boldsymbol{\ell}(\u)\ =\ 0.
\end{equation}

The above analysis together with standard tools of 
spectral theory and
semiclassical analysis for Schr\"odinger operators (see e.g. \cite{CFKS,DiSj})
lead  to the following proposition.
The proof basically relies on the fact that \eqref{eq.div0}
implies that $\Re(P_{h})$
is a perturbation of $\Delta_{f,h}$ of order $O(h^{\frac32})$.

\begin{proposition}\label{pr.Real} 
Let us assume that  \eqref{ortho} holds.
Then, there exist $C,c>0$ and $h_{0}>0$
 such that, for all $h\in (0,h_0]$, one has, counting the eigenvalues with multiplicity,
 $$\sigma(\Re(P_{h}))\cap (-\infty, ch ]\subset [-Ch^{\frac32},e^{-\frac{c}{h}}]
 \quad\text{and}\quad{\rm Card }\big(\sigma(\Re(P_{h}))\cap (-\infty, ch ]\big)=\mathsf m_0, $$
 where $\mathsf m_0$ is defined in~\eqref{eq.m0}.
 
 Moreover, there exists  $c_1>0 $  and
 $h_0>0$ such that, for all $h\in (0,h_0]$:
$$
\forall u\in H^2(\Omega)\cap H^1_0(\Omega)\,,\ \ \  \langle  \Re (P_h)(1-\pi_h^\Delta )u,  (1-\pi_h^\Delta )u\rangle_{L^2}
\geq    c_1h \Vert (1-\pi_h^\Delta )u\Vert_{L^2}^2,
$$
where $\pi_h^\Delta$ is the spectral projector associated with $\Delta_{f,h}$
and the interval $[0,c_0h]$ (see \eqref{eq.Pidelta}).
\end{proposition}
Note that the spectrum of the operator  $\Re(P_{h})$ is a priori not included in $[0,+\infty)$.

\begin{proof}
Let us define $\mathsf m:={\rm Card }(\{x\in\overline\Omega\,,\ \nabla f(x)=0\})$
and, when $\mathsf m_0>0$,  let us order the elements $x_{1},\dots,x_{\mathsf m}$
of $\{x\in\overline\Omega\,,\ \nabla f(x)=0\}$ so that (see \eqref{eq.m0})
$$
 \big \{x_1,\ldots,x_{\mathsf m_{0}} \big \}\ =\ \mathsf U_0.
$$
We consider, for every $x_{j}\in\Omega$,
a smooth open connected neighborhood $O_{j}$ of $x_{j}$
in $\Omega$
such that  $\overline{  O_{j}}\subset \Omega$.
When moreover $j\in\{1,\dots,\mathsf m_{0}\}$,
we also assume that $x_{j}$
is the only point where 
$f$ attains its minimal value  in $\overline{  O_{j}}$.
Similarly,  when $x_{j}\in \pa\Omega$,
we consider
a smooth open set $  O_{j}\subset \Omega$ 
such that $ \overline{O_{j}}$ is a neighborhood of $x_{j}$ in $\overline\Omega$.
In addition, we assume that $ \overline{O_{i}}\cap  \overline{O_{j}}=\emptyset$
when $i\neq j$, so that each $ \overline{O_{i}}$ contains precisely
one critical point of $f$, $x_{i}$, which is in its interior.\medskip

\noindent
\textbf{Step 1.} Let us first prove that there exists $c>0$ such that, for every $h>0$ small enough,
\begin{equation}
\label{eq.geq}
\dim \Ran \, \pi_{(-\infty,e^{-\frac ch}]}\big (\Re(P_{h})\big )\  \geq\ \mathsf  m_0.
\end{equation}
This  is obvious when $\mathsf  m_0=0$.
When $\mathsf m_{0} >0$, let us introduce, for every $j\in\{1,\dots,\mathsf m_{0}\}$,
a cut-off function $\chi_{j}\in \mathcal C_{c}^\infty(O_{j}) $
such that $\chi_{j}=1$ in a neighborhood of $x_{j}$
and
\begin{equation}
\label{de.psi-j}
\psi_{j}:= \frac{\chi_{j} e^{-\frac fh}}{\|\chi_{j} e^{-\frac fh}\|_{L^{2}}}.
\end{equation}
Since $x_{j}$
is the only point where 
$f$ attains its minimal value  on $\supp \chi_{j}\subset O_{j}$,
standard Laplace asymptotics give, in the limit $h\to0$,
$$
\Vert\chi_{j} e^{-\frac fh}\Vert^{2}_{L^{2}}= \frac{(\pi h)^{\frac d2}}{(\det\Hess f(x_{j}))^{\frac12}}e^{-2\frac{f(x_{j})}h}\big(1+O(h)\big).
$$
Using in addition the fact that $\chi_{j}=1$ near $x_{j}$ and thus that $f-f(x_{j})>2 c_{j}$ on $\supp \nabla \chi_{j}$
for some $c_{j}>0$, we have when $h\to0$:
\begin{equation}
\label{eq.Delta-expo}
\|\Delta_{f,h}\psi_{j}\|_{L^{2}}
=\frac{h}{\|\chi_{j} e^{-\frac fh}\|_{L^{2}}}\Vert (-h\div +\nabla f \cdot) (e^{-\frac fh}\nabla\chi_{j})\Vert_{L^{2}}
\leq e^{-\frac{c_{j}}h}.
\end{equation}
Since moreover 
 $\div \boldsymbol{\ell}(x_{j})=0$ according to
\eqref{eq.div0}, Laplace asymptotics give, when $h\to 0$,
\begin{equation}
\label{eq.div-h}
\| (\div \boldsymbol{\ell})\psi_{j}\|_{L^{2}}=O(h^{\frac12}).
\end{equation}
The two above relations imply the following one which will be useful in the sequel: 
\begin{equation}
\label{eq.RePh<}
\|\Re(P_{h})\psi_{j}\|_{L^{2}}
= 
\|(\Delta_{f,h} -h \div \boldsymbol{\ell})\psi_{j}\|_{L^{2}}
= O(h^{\frac32}).
\end{equation}

Besides, using \eqref{eq.fried}, an integration by parts, \eqref{ortho},
and  $f-f(x_{j})> 2c_{j}$ on $\supp \nabla \chi_{j}$,
it holds when $h\to 0$:
\begin{align*}
\langle \Re(P_{h})\psi_{j}, \psi_{j}   \rangle_{L^{2}} &= 
\frac{h^{2}}{\|\chi_{j} e^{-\frac fh}\|^{2}_{L^{2}}}\int_{O_{j}} |\nabla \chi_{j}|^{2}e^{-2\frac fh}
-\frac{h}{\|\chi_{j} e^{-\frac fh}\|^{2}_{L^{2}}} \int_{O_{j}}\div\boldsymbol{\ell } \,\chi_{j}^{2}\,e^{-2\frac fh}\\
&= \frac{h^{2}}{\|\chi_{j} e^{-\frac fh}\|^{2}_{L^{2}}}\int_{O_{j}} |\nabla \chi_{j}|^{2}e^{-2\frac fh}  +2\frac{h}{\|\chi_{j} e^{-\frac fh}\|^{2}_{L^{2}}} \int_{O_{j}} \chi_{j }\ \boldsymbol{\ell}\cdot \nabla\chi_{j}\ e^{-2\frac fh}
 \ \leq\  e^{-\frac{c_{j}}h}\,.
\end{align*}
Since the $\psi_{j}$, $j\in\{1,\dots,\mathsf m_{0}\}$, 
are normalized in $L^{2}(\Omega)$ with disjoint supports, 
it follows from the Min-Max principle that $\Re(P_{h})$ admits, for $c:=\min(c_{1},\dots,c_{\mathsf m_{0}})$,  at least $\mathsf m_0$ eigenvalues
less that $e^{-\frac ch}$ when $h\to 0$, which proves \eqref{eq.geq}.
\medskip

\noindent
\textbf{Step 2.} Let us now prove
that there exists $c_1>0 $  such that, for every $h>0$ small enough:
\begin{equation}
\label{eq.real-min}
\forall u\in H^2(\Omega)\cap H^1_0(\Omega)\,,\ \ \  \langle  \Re (P_h)(1-\pi_h^\Delta )u,  (1-\pi_h^\Delta )u\rangle_{L^2}
\geq    c_1h \Vert (1-\pi_h^\Delta )u\Vert_{L^2}^2.
\end{equation}

To this end,
we first define
a cut-off function $\chi\in \mathcal C_{c}^{\infty}(\mathbb R^{d},[0,1]) $
such that $\chi=1$ in $\{|x|\leq 1\}$, $\chi=0$ in $\{|x|\geq 2\}$,
and $\sqrt{1-\chi^{2}}\in \mathcal C^{\infty}(\mathbb R^{d})$.
Then, for every $j\in\{1,\dots,\mathsf m\} $,
we define the following smooth function on $\overline\Omega$:
$$\chi_{j,h}:x\in\overline\Omega\longmapsto\chi(h^{-\varepsilon}(x-x_{j}))\in\mathbb R^{+},$$
where $\varepsilon\in(0,\frac12)$ is arbitrary but fixed.
In particular,  for every $h>0$ small enough,
$\supp\chi_{j,h}\subset O_{j}$ 
when $x_{j}\in\Omega $
and, when $x_{j}\in\pa\Omega $,  $\overline{O_{j}}$ is a neighborhood of  
$\supp\chi_{j,h}$ in $\overline\Omega$.
Lastly, we define the  smooth function
$$\chi_{0,h}:x\in\overline\Omega\longmapsto\big(1-\sum_{j=1}^{\mathsf m}\chi^{2}_{j,h}\big)^{\frac12}\,, $$
so that $\sum_{j=0}^{\mathsf m}\chi^{2}_{j,h}=1$ on $\overline\Omega$.\medskip

\noindent
\textbf{Step 2a. Analysis on $\supp \chi_{0,h}$.}
Since $\supp \chi_{0,h}$ is at a distance greater than $h^{\varepsilon}$
from the set of the critical points of the Morse function $f$ in $\overline\Omega$,
there exists $c>0$ such that, for every $h>0$ small enough, $|\nabla f(x)|^{2}\geq 3ch^{2\varepsilon}$
on $\supp \chi_{0,h}$. Since $2\varepsilon<1$,
it follows that for every $h>0$ small enough and every $u\in H^2(\Omega)\cap H^1_0(\Omega)$:
\begin{align}
\nonumber
 \langle  \Re (P_h) \chi_{0,h}u,  \chi_{0,h}u\rangle_{L^2}
 &\ =\ \langle  (-h^{2}\Delta + |\nabla f|^{2}-h\Delta f -h \div \boldsymbol{\ell} ) \chi_{0,h}u, \chi_{0,h}u\rangle_{L^2}\\
 \nonumber
& \ \geq \ \langle  (|\nabla f|^{2}-h\Delta f -h \div \boldsymbol{\ell} ) \chi_{0,h}u, \chi_{0,h}u\rangle_{L^2}\\
\label{eq.low1}
&\ \geq \ 2ch^{2\varepsilon}\|\chi_{0,h}u\|^{2}_{L^{2}}\,.
\end{align}

\noindent
\textbf{Step 2b. Analysis on $\supp \chi_{j,h}$ when $x_{j}\notin \mathsf U_{0}$.}
In this case, it holds 
$O_{j}\cap \mathsf U_{0}=\emptyset$.
Applying 
\cite[Theorem 1]{DoNe2} to the Witten Laplacian
$\Delta_{f,h}^{O_{j}}$
with domain $D(\Delta_{f,h}^{O_{j}})=H^2(O_{j})\cap H^1_0(O_{j})$
then implies the existence of $c>0$ such that, for every $h>0$ small enough,
$$
\dim \Ran \, \pi_{[0,3c h]}\big (\Delta_{f,h}^{O_{j}}\big )=0.
$$
It follows that for every $h>0$ small enough and every $u\in H^2(\Omega)\cap H^1_0(\Omega)$:
\begin{align}
\nonumber
 \langle  \Re (P_h) \chi_{j,h}u,  \chi_{j,h}u\rangle_{L^2}
 &\ =\ \langle  (\Delta^{O_{j}}_{f,h} -h \div \boldsymbol{\ell} ) \chi_{j,h}u, \chi_{j,h}u\rangle_{L^2}\\
 \nonumber
& \ \geq \ \langle  (3ch -h \div \boldsymbol{\ell} ) \chi_{j,h}u, \chi_{j,h}u\rangle_{L^2}\\
\label{eq.low2}
&\ =\ \big(3ch + O(h^{1+\varepsilon})\big) \|\chi_{j,h}u\|^{2}_{L^{2}}
\ \geq \ 2ch\|\chi_{j,h}u\|^{2}_{L^{2}}\,,
\end{align}
where, to obtain the last inequality, we have used that $\div \boldsymbol{\ell}(x_{j})=0$ (see \eqref{eq.div0})
and  $\supp \chi_{j,h}\subset \{|x-x_{j}|\leq 2h^{\varepsilon}\}$
imply that, for every $h>0$ small enough, $\|\div \boldsymbol{\ell}\|_{L^{\infty}}=O(h^{\varepsilon})$ on $\supp \chi_{j,h}$.\medskip

\noindent
\textbf{Step 2c. Analysis on $\supp \chi_{j,h}$ when $x_{j}\in \mathsf U_{0}$.}
In this case, it holds 
$O_{j}\cap \mathsf U_{0}=\{x_{j}\}$ and,
applying again
\cite[Theorem 1]{DoNe2} to the Witten Laplacian
$\Delta_{f,h}^{O_{j}}$
with domain $D(\Delta_{f,h}^{O_{j}})=H^2(O_{j})\cap H^1_0(O_{j})$
then implies the existence of $c>0$ such that, for every $h>0$ small enough,
\begin{equation}
\label{eq.dim1}
\dim \Ran \, \pi_{[0,3c h]}\big (\Delta_{f,h}^{O_{j}}\big )=1.
\end{equation}

Let us  define
$$\psi_{j,h}:= \frac{\chi_{j,h} e^{-\frac fh}}{\|\chi_{j,h} e^{-\frac fh}\|_{L^{2}}} \in\mathcal C_{c}^{\infty}(O_{j},\mathbb R^{+})$$
and note that $\psi_{j,h}$ both  belongs to $D(\Delta_{f,h}^{O_{j}})$
and to $D(\Delta_{f,h})$.
Moreover, using
$\supp \chi_{j,h}\subset \{|x-x_{j}|\leq 2h^{ \varepsilon}\}$,
tail estimates and Laplace aymptotics, 
there exists $c'>0$ such that, for every $h>0$ small enough,
$$
\langle \Delta_{f,h}\psi_{j,h}, \psi_{j,h}   \rangle_{L^{2}} = 
\frac{h^{2}}{\|\chi_{j,h} e^{-\frac fh}\|^{2}_{L^{2}}}\int_{O_{j}} |\nabla \chi_{j,h}|^{2}e^{-2\frac fh}
\leq e^{-2c' \frac{h^{2\varepsilon}}{h}}\,.
$$ 
Hence,
 using
the spectral estimate
\begin{equation}
\label{eq.quadra}
\forall b >0\,, \ \forall u\in Q\left (T\right )\,,\ 
\left\Vert \pi_{[b,+\infty)} (T) \, u\right\Vert^2 \leq \frac{q_T(u)}{b}\,,
\end{equation}
with $b=3ch$ and $T=\Delta_{f,h}^{O_{j}}$,
valid 
for any 
 nonnegative self-adjoint operator $(T,D(T))$  on a Hilbert space $\left(\mc H, \Vert\cdot\Vert\right)$
 with  associated quadratic form  $(q_T,Q(T))$,
we obtain (since $2\varepsilon<1$)
\begin{equation}
\label{eq.compar-Oj}
\pi_h^{\Delta,O_{j}}
\psi_{j,h} = \psi_{j,h} + O(e^{-c' \frac{h^{2\varepsilon}}{h}})\ \ \text{in $L^{2}(O_{j})$},
\end{equation}
where for conciseness we have set $\pi_h^{\Delta,O_{j}}:=\pi_{[0,3ch]}\big (\Delta_{f,h}^{O_{j}}\big )$.
In particular, according to \eqref{eq.dim1},
$\pi_h^{\Delta,O_{j}}$ is the orthogonal projector on ${\rm Span }(\Psi_{j})$, where, using also 
\eqref{eq.compar-Oj},
\begin{equation}
\label{eq.compar-Oj'}
\Psi_{j}\ :=\ \frac{\pi_h^{\Delta,O_{j}}\psi_{j,h}}{\|\pi_h^{\Delta,O_{j}}\psi_{j,h}\|}=\psi_{j,h} + O(e^{-c' \frac{h^{2\varepsilon}}{h}})
\ \ \text{in $L^{2}(O_{j})$}.
\end{equation}
Note lastly that the same  analysis with $\chi_{j,h}\psi_{j,h}$, $b=c_{0}h$, and $T=\Delta_{f,h}$,  shows that
\begin{equation}
\label{eq.compar-Oj''}
\pi_h^\Delta (\chi_{j,h}\psi_{j,h}) = \chi_{j,h}\psi_{j,h} + O(e^{-c' \frac{h^{2\varepsilon}}{h}})\ \ \text{in $L^{2}(\Omega)$}.
\end{equation}

We can now finish this step.
Let us recall that
$\div \boldsymbol{\ell}(x_{j})=0$
and  $\supp \chi_{j,h}\subset  \{|x-x_{j}|\leq 2h^{ \varepsilon}\}$
imply that, for every $h>0$ small enough, $\|\div \boldsymbol{\ell}\|_{L^{\infty}}=O(h^{ \varepsilon})$ on $\supp \chi_{j,h}$.
Thus, for every $h>0$ small enough and every $u\in H^2(\Omega)\cap H^1_0(\Omega)$,
setting $w:=(1-\pi_h^\Delta)u\in H^2(\Omega)\cap H^1_0(\Omega)$, we have
$$
\langle  \Re (P_h) \chi_{j,h}w,  \chi_{j,h}w\rangle_{L^2}
  = \langle  \Delta^{O_{j}}_{f,h} \chi_{j,h}w, \chi_{j,h}w\rangle_{L^2}
  + O(h^{1+\varepsilon})\|\chi_{j,h}w\|_{L^{2}}^{2}.
$$
Therefore, using in addition \eqref{eq.dim1},
\begin{align}
\label{eq.lowDelta}
 \langle  \Re (P_h) \chi_{j,h}w,  \chi_{j,h}w\rangle_{L^2}
  \ \geq\ 3ch \|(1-\pi_h^{\Delta,O_{j}})\chi_{j,h}w\|^{2}
 +O(h^{1+\varepsilon})\|\chi_{j,h}w\|^{2}.
\end{align}
Besides, using \eqref{eq.compar-Oj'} and then \eqref{eq.compar-Oj''}
together with $w\in \big(\Ran \pi_h^\Delta\big)^{\perp}$,
\begin{align}
 \nonumber
\chi_{j,h}w
&=(1-\pi_h^{\Delta,O_{j}})\chi_{j,h}w+ 
\langle\chi_{j,h}w,\Psi_{j}\rangle_{L^{2}(O_{j})} \Psi_{j}\\
 \nonumber
&=(1-\pi_h^{\Delta,O_{j}})\chi_{j,h}w+ 
\langle\chi_{j,h}w,\psi_{j,h}\rangle_{L^{2}(O_{j})} \Psi_{j}
+ O(e^{-c' \frac{h^{2\varepsilon}}{h}})\|\chi_{j,h}w\|\\
 \nonumber
&=(1-\pi_h^{\Delta,O_{j}})\chi_{j,h}w+ 
 O(e^{-c' \frac{h^{2\varepsilon}}{h}})\|w\|_{L^{2}}.
\end{align}
Injecting this estimate in \eqref{eq.lowDelta}, we obtain that for
every $h>0$ small enough and every 
$u\in H^2(\Omega)\cap H^1_0(\Omega)$,
setting $w:=(1-\pi_h^\Delta)u$,
\begin{align}
\label{eq.low3}
 \langle  \Re (P_h) \chi_{j,h}w,  \chi_{j,h}w\rangle_{L^2}
 \geq  2ch\|\chi_{j,h}w\|^{2}_{L^{2}}+O(e^{-c' \frac{h^{2\varepsilon}}{h}})\|w\|_{L^{2}}^{2}.
\end{align}

\noindent
\textbf{Step 2d. Proof of \eqref{eq.real-min}.}
Let us recall  the so-called IMS localization formula (see for example~\cite{CFKS}):
  \begin{align*}
\forall w\in H^{2}(\Omega)\cap H^{1}_{0}(\Omega)\,,\ \ 
\langle  \Re (P_h) w,w \rangle &=\sum_{j=0}^{\mathsf m} \langle  \Re (P_h) \chi_{j,h}w ,
\chi_{j,h}w\rangle    -\sum_{j=0}^{\mathsf m}  h^2\,\big \Vert  | \nabla \chi_{j,h}|\,w   \big \Vert_{L^2(\Omega)}^2\\
& =
\sum_{j=0}^{n} \langle  \Re (P_h) \chi_{j,h}w ,
\chi_{j,h}w\rangle
+ O(h^{2-2\varepsilon})\|w\|^{2}_{L^2(\Omega)}.
\end{align*}
Using in addition the estimates \eqref{eq.low1}, \eqref{eq.low2},  and \eqref{eq.low3},
we obtain the existence of $c>0$ such that, for every $h>0$ small enough and
every $u\in H^2(\Omega)\cap H^1_0(\Omega)$, setting 
$w:=(1-\pi_h^\Delta)u \in H^2(\Omega)\cap H^1_0(\Omega)$,
\begin{align*}
\langle  \Re (P_h) 
w,w \rangle
&\geq   2ch\sum_{j=0}^{\mathsf m}\|\chi_{j,h}w\|^{2}_{L^{2}}
+
O(h^{2-2\varepsilon}+ e^{-c \frac{h^{2\varepsilon}}{h}})\|w\|^{2}_{L^2(\Omega)}
\geq c h\|w\|^{2}_{L^2(\Omega)}\,.
\end{align*}
This proves \eqref{eq.real-min}.\medskip

\noindent
\textbf{Step 3. End of the proof of Proposition~\ref{pr.Real}.}
Let us first recall from  \eqref{eq.geq}
the existence of $c>0$ such that, for every $h>0$ small enough, the dimension of 
$\Ran \, \pi_{(-\infty,e^{-\frac ch}]}\big (\Re(P_{h})\big )$
is at least $ \mathsf  m_0$.
Moreover, since $\dim \Ran \pi_h^\Delta =\mathsf m_0$ (see \eqref{eq.energy2}), it follows from \eqref{eq.real-min} and from the 
Min-Max principle that the
$(\mathsf m_0+1)$-th eigenvalue of 
$\Re(P_{h})$ is bounded from below by $c_{1}h$ when $h\to0$. 
The dimension of $\Ran \, \pi_{(-\infty,e^{-\frac ch}]}\big (\Re(P_{h})\big )$
is thus precisely $ \mathsf  m_0$ for every $h>0$ small enough.
To conclude,
it just remains to show that the~$\mathsf m_{0}$ eigenvalues of 
$\Re(P_{h})$ in $(-\infty,e^{-\frac ch}]$  are of the order $O(h^{\frac32})$ in the limit $h\to 0$.

To this end, note that it is possible to construct, for every $h>0$ sufficiently small, 
a
simple closed loop
$\gamma\subset\{\mathsf z\in\mathbb C, \Re \mathsf z\leq \frac{c_{1}}2h \}$ such that:
\begin{itemize}
\item $\gamma$ contains
$[-h\|\div \boldsymbol{\ell}\|_{L^{\infty}}, \frac{c_{1}}2h]$, and thus
 $\sigma(\Re(P_{h}))\cap(-\infty, \frac{c_{1}}2h]$, 
in its interior,
\item for some $c,c'>0$ independent of $h$,
$|\gamma|\leq ch$
and $\dist(\gamma,\sigma(\Re(P_{h})))\geq c'h$.
\end{itemize}
The rank-$\mathsf m_{0}$ orthogonal spectral projector $\pi_{h}$ associated with $\Re(P_{h})$
and $\sigma(\Re(P_{h}))\cap]-\infty, e^{-\frac ch}]$ then satisfies, for every $h>0$ small enough,
$$
\pi_{h}= \frac{1}{2i\pi}\int_{\gamma} (\mathsf z - \Re(P_h))^{-1} d\mathsf z.
$$
For $j\in\{1,\dots,\mathsf m_{0}\}$, let $\psi_{j}$
be the function defined in \eqref{de.psi-j}
and  recall the relation \eqref{eq.RePh<} which has not yet been used in this proof:
\begin{equation*}
\|\Re(P_{h})\psi_{j}\|_{L^{2}} = O(h^{\frac32}).
\end{equation*}
Using $\|( \mathsf z-\Re(P_h))^{-1}\|
\leq \frac{1}{c'h}$ for every $\mathsf z\in\gamma$,
it follows that
for every $h>0$ small enough, 
\begin{align}
\nonumber
(1-\pi_h)\psi_{j}
&=\frac{1}{2\pi i}\int_{ \gamma } \big ( \mathsf z^{-1}-  ( \mathsf z-\Re(P_h))^{-1}\big)\psi_{j} \, d \mathsf z \\
\label{res<=}
 &=\frac{-1}{2\pi i}\int_{ \gamma} \mathsf z^{-1}\big    (\mathsf z-\Re(P_h))^{-1} \Re(P_h)\psi_{j}\, d \mathsf z
 =O(h^{\frac12}).
\end{align}
Since the family $\big(\psi_{j}\big)_{j\in\{1,\dots,\mathsf m_{0}\}}$ is orthonormal,
the family $\big(\pi_h\psi_{j}=\psi_{j}+O(h^{\frac12})\big)_{j\in\{1,\dots,\mathsf m_{0}\}}$
is linearly independent,  and hence a basis of $\Ran \pi_{h}$, when $h\to 0$.
In addition, any normalized vector $\Psi\in \Ran \pi_{h}$
writes $\Psi\ =\ \sum_{k=1}^{\mathsf m_{0}}\mu_{k} \pi_h\psi_{j}$,
where the complex numbers $\mu_{1},\dots,\mu_{k} $ satisfy $\sum_{k=1}^{\mathsf m_{0}}|\mu_{k}|^{2}=1+O(h^{\frac12})$.
It thus follows from 
\eqref{eq.RePh<} that, when $h\to 0$:
$$
\|\Re(P_{h})\Psi\|_{L^{2}(\Omega)}
=\|\sum_{k=1}^{\mathsf m_{0}}\mu_{k} \pi_h\Re(P_{h})\psi_{j}\|_{L^{2}(\Omega)}
\leq\sum_{k=1}^{\mathsf m_{0}}|\mu_{k}| \|\Re(P_{h})\psi_{j}\|_{L^{2}(\Omega)}=O(h^{\frac32}),
$$
which implies that the $\mathsf m_{0}$ eigenvalues of 
$\Re(P_{h})$ in $(-\infty,e^{-\frac ch}]$  are of the order $O(h^{\frac32})$.

\end{proof}

\subsection{Small eigenvalues of $P_h$ and resolvent estimates.}

The aim of this section is to prove Theorem~\ref{th:2} on the number of small eigenvalues of $P_h$ (or equivalently of $L_h$, see \eqref{eq.unitary}).

\begin{theorem}\label{th:2} 
Let us assume that   \eqref{ortho} holds.
Then, there exists  $c_2>0 $  such that, for all $c_3\in (0,c_2)$, there exist $h_0>0$ and $C>0$ such that, for all 
$\mathsf  z\in \{\mathsf z\in \mathbb C, \Re \mathsf z \le c_2h , \vert \mathsf z\vert \ge c_3h\}$ and $h\in (0,h_0]$, 
 $$ P_{h}-\mathsf  z     \text{  is invertible and } \ \  \Vert (  P_{h}-\mathsf z)^{-1}\Vert \le Ch^{-1}.$$
 In addition,   there exists $h_0>0$ such that for all $h\in (0,h_0]$, $\sigma(P_{h})\cap \{\mathsf z\in \mathbb C, \Re \mathsf z\le c_2h\}$ is composed of exactly  $\mathsf m_0$ eigenvalues $ \lambda_{1,h}, \lambda_{2,h},\ldots,\lambda_{\mathsf m_0,h}$ (counted with algebraic multiplicity),
 where $\mathsf m_{0}$ is defined in \eqref{eq.m0}. Finally, there exists $c>0$ such that  for all $j\in \{1,\ldots,\mathsf m_0\}$ and $h$ small enough, $\vert \lambda_{j,h}\vert\le e^{-\frac ch}$. All these  results also hold for $P_h^*$. 
\end{theorem}
\begin{proof}
 Note first that the last sentence in the statement of Theorem~\ref{th:2} 
 concerning $P_h^*$ is an immediate consequence of the part 
 concerning $P_{h}$ since $\sigma(P^{*}_{h})=\overline{\sigma(P_{h})}$ (with multiplicity)
 and, for all $\mathsf  z\in\C\setminus\sigma(P_{h})$,
 $\|(P_{h}-\mathsf  z)^{-1}\|=\|( P^{*}_{h}-\bar{\mathsf  z })^{-1}\|$
 (see indeed \cite[Section 6.6 in Chapter 3]{Kato}).

 Let us also recall the relations \eqref{eq.energy2},
 \eqref{eq.Pidelta}, and \eqref{eq.Witt-expo} stated in
 the beginning of Section~\ref{sub.RePh}.
Let us consider, for $j\in \{1,\ldots,\mathsf m_0\}$, a 
$L^2(\Omega)$-normalized eigenfunction
$u^\Delta_{j,h}$ of $\Delta_{f,h}$
associated with its $j$-th eigenvalue.
Since $P_{h}=\Delta_{f,h}+  2\boldsymbol{\ell}\cdot \nabla_{f,h}$
has domain $D(P_{h}) = D(\Delta_{f,h})$
and the quadratic form associated with $\Delta_{f,h}$
is given by \eqref{eq.fried} with $\boldsymbol{\ell}=0$:
\begin{equation}\label{eq.PP0}
\text{$\exists c>0$ such that, when $h\to0$,}\ \ \ 
 \Vert P_hu^\Delta_{j,h} \Vert_{L^2(\Omega)}\le e^{-\frac ch}.
 \end{equation}
Similarly, since $P^{*}_{h}=\Delta_{f,h}-  2\boldsymbol{\ell}\cdot \nabla_{f,h}-2h\div\boldsymbol{\ell} $ has domain $D(P^{*}_{h}) = D(\Delta_{f,h})$,
there exists $c>0 $ such that, for every $h>0$ small enough,
$$
\Vert P^{*}_hu^\Delta_{j,h} \Vert_{L^2(\Omega)}\le e^{-\frac ch} + 2h\Vert (\div\boldsymbol{\ell}) u^\Delta_{j,h}\Vert_{L^2(\Omega)}.
$$
Considering now the orthonormal family $(\psi_{j})_{j\in\{1,\dots,\mathsf m_{0}\}}$
defined in the previous section in \eqref{de.psi-j} 
and using  the spectral estimate \eqref{eq.quadra} with $b=c_{0}h$, $T=\Delta_{f,h}$, and \eqref{eq.Delta-expo},
there exists $c'>0$ such that, for every $j\in\{1,\dots,\mathsf m_{0}\}$ and $h>0$ small enough,
\begin{equation}
\label{eq.quasi-ortho}
\pi_h^{\Delta}\psi_{j} = \psi_{j} + O(e^{- \frac{c'}{h}})\ \ \text{in $L^{2}(\Omega)$}.
\end{equation}
Using in addition \eqref{eq.div-h}, it thus follows that, for every $h>0$ small enough, 
$$
\Vert (\div\boldsymbol{\ell}) \pi_h^{\Delta}\psi_{j}\Vert_{L^2(\Omega)}
= O(h^{\frac12}).
$$
Hence, since \eqref{eq.quasi-ortho} implies that
each $u^\Delta_{j,h}$
writes
$u^\Delta_{j,h} = \sum_{k=1}^{\mathsf m_{0}}\mu_{k} \pi_h^{\Delta}\psi_{j}$
for some complex numbers $\mu_{1},\dots,\mu_{k} $ satisfying $\sum_{k=1}^{\mathsf m_{0}}|\mu_{k}|^{2}=1+O(e^{- \frac{c'}{h}})$,
we obtain that for every $h>0$ small enough,
\begin{equation}\label{eq.PP1}
 \Vert P^{*}_hu^\Delta_{j,h} \Vert_{L^2(\Omega)} =  O(h^{\frac32}).
 \end{equation}

Let us now define the operator $\hat{ P}_{h}$ by  
$$\hat{ P}_{h}:=(1-\pi_h^\Delta ) P_h (1-\pi_h^\Delta ) \text{ with domain } (1-\pi_h^\Delta )D(P_h) \, \text{ on } \, \hat{\mathsf E}:=(1-\pi_h^\Delta )L^2(\Omega),$$
where we recall that $D(P_h)= D(\Delta_{f,h}) =H^2(\Omega)\cap H^1_0(\Omega)$. 
Note that the  space $\hat{\mathsf E}$
(equipped with the restricted  $L^2(\Omega)$-Hermitian inner product)
 is a Hilbert space and that the operator~$\hat{ P}_{h}:D(\hat{ P}_{h})\to\hat{\mathsf E}$ is well defined, since $(1-\pi_h^\Delta )D(P_h)=\hat{\mathsf E}\cap D(P_h) \subset D(P_h)$, with dense domain in~$\hat{\mathsf E}$.

The rest of the proof is reminiscent of the analysis led in \cite[Section 2B.]{LePMi20}
and
 is divided into two steps. 
 \medskip
 
 \noindent
 \textbf{Step 1. Resolvent estimates for $\hat{ P}_{h}:D(\hat{ P}_{h})\to\hat{\mathsf E}$.}  
First, the operator $\hat{ P}_{h}$   is closed. 
This  follows from  the fact that $P_h:D(P_{h})\to L^2(\Omega)$ is closed
and from the relation $\hat{ P}_{h}= P_h + \pi_h^\Delta P_h\pi_h^\Delta - \pi_h^\Delta  P_h-P_h\pi_h^\Delta$
on $D(\hat{ P}_{h})$, since $\pi_h^\Delta P_h\pi_h^\Delta - \pi_h^\Delta  P_h-P_h\pi_h^\Delta $
extends into a bounded operator 
$T_{h}$ on~$L^2(\Omega)$.
Indeed, $P_h\pi_h^\Delta$ and then $\pi_h^\Delta P_h\pi_h^\Delta$
extend into bounded operators on~$L^2(\Omega)$ since 
$\pi_h^\Delta$ is continuous with finite rank,
and it is also the case
for $\pi_h^\Delta  P_h$ 
since for all $u \in D(P_h)=D(P^*_{h})$, 
$$\pi_h^\Delta  P_hu=\sum_{j=1}^{\mathsf m_0} \langle u^\Delta_{j,h}, P_hu  \rangle_{L^2(\Omega)}u^\Delta_{j,h}= \sum_{j=1}^{\mathsf m_0} \langle P_h^*u^\Delta_{j,h}, u  \rangle_{L^2(\Omega)}u^\Delta_{j,h}.$$
 The above considerations also imply that
the adjoint of $\hat{ P}_{h}$ is the operator   
$$\hat{ P}_{h}^*= (1-\pi_h^\Delta ) P_h^* (1-\pi_h^\Delta ) \text{ with domain } (1-\pi_h^\Delta )D(P_h).
$$

Let us now prove the following  resolvent estimates for $\hat{ P}_{h}$: there exist $C>0$ and $c_2 >0$ such that, for all $h>0$ small enough and  $\mathsf  z\in \mathbb C$ such that $\Re \mathsf  z\le c_2 h$, 
\begin{equation}\label{eq.hat-reso}
 \hat{ P}_{h}-\mathsf  z     \text{  is invertible and } \Vert (\hat{ P}_{h}-\mathsf  z)^{-1}\Vert \le Ch^{-1}.
 \end{equation}
To prove this claim, let us consider $w\in D(\hat{ P}_{h})=(1-\pi_h^\Delta )D(P_h)$ and $\mathsf  z\in \mathbb C$. Then,
according to Proposition~\ref{pr.Real},
 it holds, for every $h>0$ small enough,
 \begin{align*}
\Re  \langle (\hat{ P}_{h}-\mathsf  z)w,  w\rangle_{L^2(\Omega)}&= \Re \langle  P_h(1-\pi_h^\Delta )w,  (1-\pi_h^\Delta )w\rangle_{L^2(\Omega)}-(\Re\mathsf   z)  \,  \Vert (1-\pi_h^\Delta )w\Vert_{L^2(\Omega)}^2\\
 &= \langle  \Re (P_{h})(1-\pi_h^\Delta )w,  (1-\pi_h^\Delta )w\rangle_{L^2(\Omega)}-(\Re \mathsf  z)  \Vert (1-\pi_h^\Delta )w\Vert_{L^2(\Omega)}^2\\
 &\ge  [c_1h-\Re\mathsf   z ] \Vert (1-\pi_h^\Delta )w\Vert_{L^2(\Omega)}^2=   [c_1h-\Re\mathsf   z ]  \Vert  w\Vert_{L^2(\Omega)}^2.
 \end{align*}
 The same inequality also holds for  $\hat{ P}_{h}^*-\mathsf  z$ since $\Re (P_{h})=\Re (P^{*}_{h})$. 
 Let us now fix $c_2 \in (0, c_1)$. When   $\Re \mathsf  z\le c_2  h$ and $h>0$ is small enough, the previous inequality implies
\begin{equation}\label{eq.hat-coer}
\Vert (\hat{ P}_{h}-\mathsf  z)w\Vert_{L^2(\Omega)}\ge (c_1-c_2 ) h  \Vert  w\Vert_{L^2(\Omega)}.
 \end{equation}
Consequently, when   $\Re \mathsf  z\le c_2  h$ and $h>0$ is small enough, $\hat{ P}_{h}-\mathsf  z$ is injective and its range is closed. Since the same inequality also holds for its adjoint $\hat{ P}_{h}^*-\bar{\mathsf  z}$, the range of $\hat{ P}_{h}-\mathsf z$ is dense in $\hat{\mathsf E}$. Thus, $\hat{ P}_{h}-\mathsf  z:D(\hat{ P}_{h})\to\hat{\mathsf E}$ is invertible  and the relation~\eqref{eq.hat-reso} follows from~\eqref{eq.hat-coer}.  
  \medskip
 
 \noindent
 \textbf{Step 2. Grushin problem and end of the proof of Theorem~\ref{th:2}.}  Define the operators:
 $$R_-: \mathbb C^{\mathsf m_0} \to L^2(\Omega) , \ (\mu_k)_{j=1}^{\mathsf m_0}\mapsto  
  \sum_{j=1}^{\mathsf m_0} \mu_j \, u^\Delta_{j,h},  \text{ and } R_+:  L^2(\Omega)\to   \mathbb C^{\mathsf m_0}  , \, \ u\mapsto (\langle  u, u^\Delta_{j,h} \rangle_{L^2(\Omega)})_{j=1}^{\mathsf m_0}.$$
   We equip $\mathbb C^{\mathsf m_0}$ with the $\ell^2$ norm. 
   Note the relations
   \begin{equation}\label{eq.r--}
  R_+^*=R_-\,,\ \   R_-R_+= \pi_h^\Delta\,,   \text{ and }  R_+R_-=\mathsf I_{ \mathbb C^{\mathsf m_0} },
   \end{equation} 
and that,  for all $h>0$,
      \begin{equation}\label{eq.phr}
       \Vert R_+ \Vert \le1  \ \text{ and }\ \Vert R_- \Vert \le 1.      \end{equation} 
Moreover, according to~\eqref{eq.PP0} and \eqref{eq.PP1}, there exists $c>0$ such that for every $h>0$ small enough, it holds: 
        \begin{equation}\label{eq.phr2}
       \Vert R_+ P_h\Vert  = O(h^{\frac32})  \ \text{ and }\  \Vert   P_h R_-\Vert  \le e^{-\frac ch}.
      \end{equation} 
For $\mathsf z\in \mathbb C$, let us denote by $\mathcal  P_{h} (\mathsf z)$ the linear operator defined by  
 $$(u,u_-)\in D(P_h)\times \mathbb C^{\mathsf m_0} \mapsto  \begin{pmatrix} (P_{h}-\mathsf z)u+  R_- u_-    \\  R_+ u \end{pmatrix}\in L^{2}(\Omega)\times \mathbb C^{\mathsf m_0}.$$
Using \eqref{eq.hat-reso} and the same analysis as the one made to prove~\cite[Lemma  2.2]{LePMi20}, we deduce that, when   $\Re \mathsf z\le c_2  h$ and $h>0$ is small enough, $\mathcal  P_{h} (\mathsf z)$ is invertible (i.e. the Grushin problem $\mathcal  P_{h} (\mathsf z)$ is well posed) and its inverse writes
 $$   (f,g)\in L^2(\Omega) \times \mathbb C^{\mathsf m_0} \mapsto  \begin{pmatrix} \mathcal  E(\mathsf z) & \mathcal E_+(\mathsf z)  \\ \mathcal E_-(\mathsf z) & \mathcal E_{-+}(\mathsf z)      \end{pmatrix} \begin{pmatrix}  f  \\ g   \end{pmatrix} \in D(P_h)\times \mathbb C^{\mathsf m_0},$$
where the operators $\mathcal E$, $\mathcal E_+  $, $\mathcal E_-$, and $ \mathcal E_{-+}  $ 
are holomorphic on $\{\Re \mathsf z\le c_2 h\}$
and satisfy:
\begin{enumerate}
\item $\mathcal E(\mathsf z)=(\hat{ P}_{h}-\mathsf z)^{-1}(1-\pi_h^\Delta )$ and thus, according to \eqref{eq.hat-reso}:   
 \begin{equation}\label{eq.EEz}
\text{for every $\mathsf z\in \{\Re \mathsf z\le c_2 h\}$},\ \  \Vert \mathcal E(\mathsf z)\Vert \le  Ch^{-1},
 \end{equation}
\item$\mathcal E_{-+}(\mathsf z) = -R_+(P_h-\mathsf z)R_-+R_+P_h (\hat{ P}_{h}-\mathsf z)^{-1}(1-\pi_h^\Delta )P_hR_-$,
\item$\mathcal E_{+}(\mathsf z) = R_-  -(\hat{ P}_{h}-\mathsf z)^{-1}(1-\pi_h^\Delta )P_hR_-$,
\item $\mathcal E_-(\mathsf z)=R_+ -R_+P_h(\hat{ P}_{h}-\mathsf z)^{-1}(1-\pi_h^\Delta )$.
\end{enumerate}
Moreover, $P_h-\mathsf z$ is invertible if and only if $\mathcal  E_{-+}(\mathsf z)$ is invertible, and in this case, 
\begin{equation}\label{eq.p-z}
(P_h-\mathsf z)^{-1}=\mathcal E(\mathsf z)- \mathcal E_+(\mathsf z)\mathcal E_{-+}(\mathsf z)^{-1}\mathcal E_-(\mathsf z).
\end{equation}
We refer to  \cite{sjostrand2007elementary} for more details on so-called Grushin problems.

Using~\eqref{eq.hat-reso},~\eqref{eq.r--},~\eqref{eq.phr}, and~\eqref{eq.phr2}, 
one deduces that there exists $c>0$ such that, for every 
$h>0$ small enough and uniformly with respect to  $\mathsf z\in\{\Re \mathsf z\le c_2 h\}$,
$$  \mathcal E_-(\mathsf z)=R_+ +O(h^{\frac12}), \  \mathcal E_+(\mathsf z)=R_-+O( e^{-\frac ch}), \text{ and } \mathcal E_{-+}(\mathsf z)= \mathsf z \mathsf I_{\mathbb C^{\mathsf m_0}} + O( e^{-\frac ch}).$$
In particular, when in addition $\vert \mathsf z\vert \ge e^{-\frac c{2h}}$,  $\mathcal E_{-+}(\mathsf z)$ is invertible and thus so is $P_h-\mathsf z$ (see the line above~\eqref{eq.p-z}). Therefore, for every $h>0$ small enough:
\begin{equation}
\label{eq.spec-expo}
\sigma(P_h)\cap \big \{\mathsf z\in \mathbb C, \Re  \mathsf z\le  c_2h\big \}\subset\big  \{\vert\mathsf z\vert \le  e^{-\frac c {2h}} \big \}.
\end{equation}

Let us now fix $c_3\in (0,c_2)$. The operator $\mathcal E_{-+}(\mathsf z)$ is then invertible for every  $h>0$ small enough and 
every $\mathsf  z\in \{\Re \mathsf z \le c_2h , \vert \mathsf z\vert \ge c_3h\}$, and satisfies  $\mathcal E_{-+}(\mathsf z)^{-1}= \mathsf z^{-1}(\mathsf I_{\mathbb C^{\mathsf m_0}} + O( e^{-\frac c{2h}}))$. Hence, according to~\eqref{eq.p-z}, since $R_-R_+= \pi_h^\Delta $,  $\Vert \pi_h^\Delta \Vert  \le1$, $ \Vert R_+ \Vert \le 1$,  and $\Vert R_- \Vert \le1$, 
 the previous estimates on $\mathcal E_+(\mathsf z)$, $\mathcal E_-(\mathsf z)$, and $\mathcal E_{-+}(\mathsf z)$ imply  
 that for all $h$ small enough 
 and  uniformly with respect to $\mathsf  z\in \{ \Re \mathsf z \le c_2h , \vert \mathsf z\vert \ge c_3h\}$:
 \begin{equation}\label{eq.PH-1}
 (P_h-\mathsf z)^{-1}= \mathcal E(\mathsf z)- \mathsf  z^{-1} (\pi_h^\Delta + O( h^{\frac12}))
  = \mathcal E(\mathsf z)- \mathsf  z^{-1} \pi_h^\Delta + O( h^{-\frac12}).
  \end{equation}
Using in addition \eqref{eq.EEz}, there exists $K>0$ such that  for all  for $h$ small enough and 
 $\mathsf  z\in \{ \Re \mathsf z \le c_2h , \vert \mathsf z\vert \ge c_3h\}$:
$$\Vert (P_h-\mathsf z)^{-1}\Vert \le  Ch^{-1} + |\mathsf z|^{-1} + O( h^{-\frac12})  \le Kh^{-1}.$$

Lastly, take  $\beta \in (c_3,c_2)$. According to \eqref{eq.spec-expo},    the spectral Riesz projector
 \begin{equation}\label{pih}
 \pi_h^P:= \frac{1}{2i\pi}\int_{\{\vert \mathsf z\vert =\beta h\}} ( \mathsf z -P_h)^{-1} d\mathsf z
 \end{equation}
  is well defined for every $h>0$ small enough and its rank is the number of eigenvalues
  of $P_{h}$ in $\{\Re \mathsf z\le  c_2h\big \}$, counted with algebraic multiplicity.
 Moreover,  
   Equation~\eqref{eq.PH-1} implies that for every $h>0$ small enough, 
   \begin{equation}\label{eq.pih==}
   \pi_h^P= \pi_h^\Delta+  O( h^{\frac12})
   \end{equation}
   and thus,  
  $\dim \Ran (\pi^{P}_h)=\dim \Ran (\pi_h^\Delta)=\mathsf m_0$ (see~\eqref{eq.energy2}).
  Therefore, for every $h$ small enough, $\sigma(P_h)\cap\big \{\Re \mathsf z\le  c_2h\big \}$ is composed of $\mathsf m_0$ eigenvalues, counted with algebraic multiplicity, which are exponentially small. 
   This concludes the proof of Theorem~\ref{th:2}.
\end{proof}

\section{Proof  of Theorem~\ref{th:main}}

\subsection{Rough asymptotic estimates on $u_{1,h}^P$ and on $u_{1,h}^{P^{*}}$}
\label{sec.SE-P}

We assume from now on, without loss of generality, that
 the principal eigenmodes $u_{1,h}^P$ of $P_{h}$ and $u_{1,h}^{P^{*}}$ of $P^{*}_{h}$
 defined in Proposition~\ref{pr.spectre}
 are normalized in $L^2(\Omega)$. 
We derive 
 in the following proposition a priori estimates on these eigenmodes which will be used in Section~\ref{eq.secEqx} to prove Theorem~\ref{th:main}. 
  
\begin{proposition}\label{pr.approximation0} 
Assume   \eqref{ortho} and   \eqref{well}. 
 For any $\eta\in (0,\min_{\pa \Omega}f-f(x_0))$, let  $\chi_\eta:\overline \Omega\to [0,1]$ be a  smooth function   such that $\chi_\eta =1$  on $\mathbf{C}_{\text{min}}(\eta)  $ (see \eqref{eq.Ceta}) and  $\chi_\eta =0$  on $\overline \Omega\setminus \mathbf{C}_{\text{min}}(\eta/2)$. Set 
$$\mathsf u_\eta=\frac{\chi_\eta e^{-\frac fh} }{\Vert \chi_\eta  e^{-\frac fh}\Vert_{L^2(\Omega)}}.$$ 
Then,     there exists $c>0$ such that  for all $h$ small enough, $u_{1,h}^P$ and $u_{1,h}^{P^{*}}$
satisfy
\begin{equation}\label{eq.appro}
u_{1,h}^P=\mathsf u_\eta +O(  e^{-\frac ch})\  \text{ and } \ u_{1,h}^{P^*}=\mathsf u_\eta +O(  h^{\frac12})
\ \ \ \text{in $L^2(\Omega)$},
\end{equation} 
as well as
\begin{equation}\label{eq.exLL}
\frac{\int_{\Omega\setminus  \overline{\mathbf{C}}_{\text{min}}(\eta)}    u_{1,h}^{P}  \, e^{-\frac fh }}{\int_{\Omega} u_{1,h}^{P}e^{-\frac fh }}  =O(  e^{-\frac ch })
\  \text{ and } \ 
 \frac{\int_{\Omega\setminus  \overline{\mathbf{C}}_{\text{min}}(\eta)}    u_{1,h}^{P^*}  \, e^{-\frac fh }}{\int_{\Omega} u_{1,h}^{P^*}e^{-\frac fh }}  =O(  e^{-\frac ch }).
 \end{equation}
\end{proposition}

\begin{proof}
 Assume  \eqref{ortho}. According to Theorem~\ref{th:2}, $P_h$ admits precisely $\mathsf m_0$
 eigenvalues    in $\{\Re \mathsf z\le c_2h\}$, where we recall that $\mathsf m_0$ is the number of local minima of $f$ in $\Omega$ (see \eqref{eq.m0}), and these $\mathsf m_0$ eigenvalues  are exponentially small. When in addition \eqref{well} holds,
$\mathsf U_0=\{x_0\}$ and then
  $\mathsf m_0=1$. Thus, $\lambda_{1,h}^P$ is the unique eigenvalue of $P_h$  in $\{\Re \mathsf z\le c_2h\}$ and $\pi_h^P$ (see \eqref{pih}) has rank~$1$. Notice that the same holds for $\pi_h^{P^*}$.

 In what follows, we assume \eqref{ortho} and   \eqref{well}.  
 \medskip
 
 \noindent
 \textbf{Step 1.  Proof of  \eqref{eq.appro}.} Laplace's method provides  (since $\chi_\eta=1$ in a neighborhood of $x_0$ which is, according to \eqref{well}, the unique   global minimum point of $f$ in $\overline \Omega$):
 \begin{equation}\label{eq.Nchi}
 \Vert \chi_\eta  e^{-\frac fh}\Vert_{L^2(\Omega)}=(\pi\, h)^{\frac d4} \, \big(  \det\Hess  f(x_0)\big )^{-\frac 14}          \ e^{-\frac {f(x_{0})}h }  (1+O(h) ).
  \end{equation}
 Since $P_{h}=\Delta_{f,h}+2\boldsymbol{\ell}\cdot \nabla_{f,h}=(-h\div +\nabla f \cdot)\nabla_{f,h}+2\boldsymbol{\ell}\cdot \nabla_{f,h}$ with $\nabla_{f,h}= h\, e^{-\frac f h}\nabla e^{\frac f h}$, the function $P_{h}\mathsf u_\eta$ is supported in $\supp\nabla \mathsf \chi_\eta$, where 
 $f-f(x_{0})$ is larger than $\min_{\pa \Omega }f-f(x_{0})-\eta>0$. Hence, following the reasoning used to prove~\eqref{eq.Delta-expo},
there exists $c>0$ such that, for every $h>0$ small enough:
 \begin{equation}\label{eq.Est-Ph}  
 \Vert P_h \mathsf u_\eta  \Vert_{L^2(\Omega)}\le e^{-\frac ch}.
\end{equation} 
Since moreover  $P^{*}_{h}= 2 \Re(P_{h}) -P_h$, 
\eqref{eq.Est-Ph} and \eqref{eq.RePh<}
imply that, in the limit $h\to0$:
\begin{equation}\label{eq.Est-Ph'}  
 \Vert P^{*}_h \mathsf u_\eta  \Vert_{L^2(\Omega)}= O(h^{\frac32}).
\end{equation}

On the other hand, since $\mathsf u_\eta \in D(P_h)$,
following the argument leading to \eqref{res<=},
 the relation
 \eqref{pih} and the resolvent estimate of Theorem~\ref{th:2} imply
 the existence of $C>0$ 
such that, when $h\to0$,
\begin{equation}
\label{eq.1-Pi}
 \Vert  (1-\pi_h^P) \mathsf u_\eta \Vert_{L^2(\Omega)}\le  C h^{-1}  \Vert P_h\mathsf u_\eta   \Vert_{L^2(\Omega)}.
 \end{equation}
Consequently,  using also \eqref{eq.Est-Ph}, 
there exists $c>0$ such that, for every $h>0$ small enough:
$$
 \pi_h^P\mathsf u_\eta= \mathsf u_\eta +O(  e^{-\frac ch}) \text{ in }  L^2(\Omega). 
$$
In particular,  
 $\Vert \pi_h^P\mathsf u_\eta\Vert _{L^2(\Omega)}= 1+O(e^{-\frac ch})$ for all $h$ small enough and, since  $\pi_h^P$  has rank $1$,  $\mathsf u_\eta\geq 0$, and $u_{1,h}^P,u_{1,h}^{P^*}>0$ in $\Omega$, it  holds:
\begin{equation}\label{eq.Pr1} 
u_{1,h}^P=+\frac{  \pi_h^P\mathsf u_\eta}{\Vert \pi_h^P\mathsf u_\eta\Vert _{L^2(\Omega)}}=\frac{  \pi_h^P\mathsf u_\eta}{1+O(e^{-\frac ch})}=\mathsf u_\eta +O(  e^{-\frac ch}) \ \ \text{ in }  L^2(\Omega).
\end{equation}
Similarly, using the resolvent estimate of Theorem~\ref{th:2} for $P_h^*$ 
together with \eqref{eq.Est-Ph'}, we deduce that,
when $h\to0$, 
 $\pi_h^{P^*}\mathsf u_\eta= \mathsf u_\eta +O(  h^{\frac12})$
 and
\begin{equation}\label{eq.Pr2} 
u_{1,h}^{P^*}=+\frac{  \pi_h^{P^*}\mathsf u_\eta}{\Vert \pi_h^{P^*}\mathsf u_\eta\Vert _{L^2(\Omega)}}=\frac{  \pi_h^{P^*}\mathsf u_\eta}{1+O(h^{\frac12})}= \mathsf u_\eta +O(h^{\frac12}) \ \ \text{ in }  L^2(\Omega). 
\end{equation}
 This ends the proof of \eqref{eq.appro}.  
\medskip

 \noindent
 \textbf{Step 2. Proof of \eqref{eq.exLL}.}  
According to \eqref{eq.appro},  we have:
\begin{align*}
\int_{\Omega} u_{1,h}^{P^*}e^{-\frac f h}= \int_{\Omega} \mathsf u_\eta \, e^{- \frac  fh } + O(h^{\frac12})\|e^{-\frac fh}\|_{L^{2}(\Omega)}  &=  \frac{\int_{\Omega}\chi_\eta\, e^{- \frac 2h f }}{\Vert \chi_\eta  e^{-\frac fh}\Vert_{L^2(\Omega)} } +  O(h^{\frac12})\|e^{-\frac fh}\|_{L^{2}(\Omega)}.
\end{align*}
Hence, using  Laplace's method as we did to get \eqref{eq.Nchi}, we have when $h\to 0$:
$$
\int_{\Omega} u_{1,h}^{P^*}e^{-\frac f h}
 = (\pi\, h)^{\frac d4}    \big(  \det\Hess  f(x_0)\big )^{-\frac 14}          \ e^{-\frac {f(x_0)} h} \big(1+ O(h^{\frac12})\big).
$$
Thus, since $f-f(x_{0})\ge \min_{  \pa \Omega}f -f(x_{0})   -\eta>0$ on $ \Omega\setminus  \overline{\mathbf{C}}_{\text{min}}(\eta)$, there exists $c>0$ such that, for every $h$ small enough: 
\begin{align*}
  \frac{\int_{\Omega\setminus  \overline{\mathbf{C}}_{\text{min}}(\eta)}    u_{1,h}^{P^*}  \, e^{-\frac f h}}{\int_{\Omega} u_{1,h}^{P^*}e^{-\frac f h}}  &=  \frac{\int_{\Omega\setminus  \overline{\mathbf{C}}_{\text{min}}(\eta)}    \chi_{\eta} \, e^{-\frac 2h f}}{\Vert \chi_{\eta}  e^{-\frac fh}\Vert_{L^2(\Omega)}\int_{\Omega} u_{1,h}^{P^*}e^{-\frac f h}} +O(  h^{\frac12}) \frac{\Vert e^{-\frac  fh}\Vert_{L^{2}(\Omega\setminus  \overline{\mathbf{C}}_{\text{min}}(\eta))}}{\int_{\Omega} u_{1,h}^{P^*}e^{-\frac f h}}   
  =O  ( e^{-\frac ch} ),
\end{align*}
which proves  \eqref{eq.exLL} for $u_{1,h}^{P^*}$. The proof for  $u_{1,h}^{P}$ is analogous.
\end{proof}

\subsection{Proof of Theorem~\ref{th:main}}
 \label{eq.secEqx}

 Assume \eqref{ortho} and   \eqref{well}.    
We recall that a quasi-stationary distribution   
for the  process~\eqref{eq.langevin} 
 in $\Omega$ is a probability measure  $\mu_h$ on $\Omega$ such that,  for any time $t\ge0$ and any Borel set $A\subset \Omega$,   $\mathbb P_{\mu_h}(X_t\in A\,  | \,   t<\tau_{\Omega^c}) =\mu_h(A)$. 
Let us now introduce the following probability distribution on $\Omega$ (see Proposition~\ref{pr.spectre}):
$$
 \nu_h(dx)=\frac{u_{1,h}^{P^*}e^{-\frac f h}}{\int_{\Omega}u_{1,h}^{P^*}e^{-\frac f h}}\ dx.
$$
Using the   smoothness of the killed semigroup $P_tf(x)=\mathbb E_x[f(X_t)1_{t<\tau_{\Omega^c}}]$ (summarized e.g. in~\cite[Section 2.1]{ramilarxiv1}) and similar computations as those used in the proof of \cite[Proposition 2.2]{le2012mathematical}, one deduces that   $\nu_h$ is a quasi-stationary distribution\footnote{Even if the uniqueness of $\nu_h$ is not required here, we mention that for elliptic processes with smooth coefficients and when   $\Omega$ is a smooth bounded domain, it is well-known  that the quasi-stationary distribution in $\Omega$ is unique, see e.g.~\cite{gong1988killed,champagnat2018criteria,pinsky-85,guillinqsd}. } 
 for the process \eqref{eq.langevin} in $\Omega$ and that, when $X_0$ is initially distributed according to the  measure $\nu_h$, it holds:
 \begin{equation}\label{eq.expo-law}
 \tau_{\Omega^c}\sim \mathcal E(\lambda_{1,h}^L), \text{ where we recall that } \lambda_{1,h}^L=\frac{ \lambda_{1,h}^{P}}{2h},
  \end{equation}
  and where $\mathcal E(\lambda_{1,h}^L)$ stands for the exponential law of parameter $\lambda_{1,h}^L$. 
  \medskip
  
  \noindent
  \textbf{Step 1.  Proof of  \eqref{eq.info}.}
  Note that the first statement of \eqref{eq.info} has already been proved at the very beginning of the proof of 
  Proposition~\ref{pr.approximation0}.
Moreover, according to~\eqref{eq.expo-law}, it holds,
 \begin{equation}\label{eq.Starting-Point}
 \frac{1}{\lambda_{1,h}^L}=\mathbb E_{\nu_h}[\tau_{\Omega^c}]=\int_{\Omega} \mathbb E_{x}[\tau_{\Omega^c}]\, \nu_h(dx)=\frac{\int_\Omega  \mathbb E_{x}[\tau_{\Omega^c}] u_{1,h}^{P^*}  \, e^{-\frac f h}}{\int_{\Omega} u_{1,h}^{P^*} e^{-\frac f h}}.
 \end{equation}  
Take now $\eta_0\in (0,\min_{\pa \Omega}f-f(x_0))$ and recall that
$\mathbf{C}_{\text{min}}(\eta_0) = \mathbf{C}_{\text{min}}\cap \{f<\min_{\pa \Omega}f-\eta_0\}$
(see~\eqref{eq.Ceta}).
 One then has: 
 \begin{align}\label{eq.Starting-Point2}
 \frac{1}{\lambda_{1,h}^L}&=\frac{\int_{\Omega\setminus \overline{\mathbf{C}}_{\text{min}}(\eta_0)}   \mathbb E_{x}[\tau_{\Omega^c}]  u_{1,h}^{P^*}(x)  \, e^{-\frac {f(x)}h}dx}{\int_{\Omega} u_{1,h}^{P^*}e^{-\frac f h }}+\frac{\int_{ \overline{\mathbf{C}}_{\text{min}}(\eta_0)}  \mathbb E_{x}[\tau_{\Omega^c}]  u_{1,h}^{P^*}(x) \, e^{-\frac {f(x)} h}dx}{\int_{\Omega}u_{1,h}^{P^*}e^{-\frac f h }}\\
 \nonumber
 &\geq \frac{\int_{ \overline{\mathbf{C}}_{\text{min}}(\eta_0)}   \mathbb E_{x}[\tau_{\Omega^c}]  u_{1,h}^{P^*}  \, e^{-\frac f h}}{\int_{\Omega} u_{1,h}^{P^*}e^{-\frac f h }}.
 \end{align}
Moreover, Theorem~\ref{th:leveling} with $K=\overline{\mathbf{C}}_{\text{min}}(\eta_0)$ ($\subset \mathcal A(\{x_0\})$) implies that
for some $c>0$ and every $h>0$ small enough: 
  $$\frac{\int_{ \overline{\mathbf{C}}_{\text{min}}(\eta_0)}   \mathbb E_{x}[\tau_{\Omega^c}]  u_{1,h}^{P^*}  \, e^{-\frac f h}}{\int_{\Omega} u_{1,h}^{P^*}e^{-\frac f h }}= \frac{\int_{ \overline{\mathbf{C}}_{\text{min}}(\eta_0)}      u_{1,h}^{P^*}  \, e^{-\frac f h}}{\int_{\Omega} u_{1,h}^{P^*}e^{-\frac f h }} \times  \mathbb E_{x_0}[\tau_{\Omega^c}](1+O(e^{-\frac ch})).$$ 
 Then, using in addition \eqref{eq.exLL} and taking $c>0$ smaller if necessary, we have when $h\to0$:
\begin{equation}
\label{eq.Starting-Point3}
 \frac{1}{\lambda_{1,h}^L}
\geq \frac{\int_{ \overline{\mathbf{C}}_{\text{min}}(\eta_0)}   \mathbb E_{x}[\tau_{\Omega^c}]  u_{1,h}^{P^*}  \, e^{-\frac f h}}{\int_{\Omega} u_{1,h}^{P^*}e^{-\frac f h }}
=
\mathbb E_{x_0}[\tau_{\Omega^c}](1+O(e^{-\frac ch})),
\end{equation}
which leads, applying  again Theorem~\ref{th:leveling}, to
$$\limsup_{h\to0} \,h \ln \lambda_{1,h}^L\  \leq\   -\lim_{h\to0} \,h \ln \mathbb E_{x_0}[\tau_{\Omega^c}] \ =\  -2(\min_{\pa \Omega}f-f(x_0)).$$

Finally, the fact that 
$$\liminf_{h\to0} \,h \ln \lambda_{1,h}^L\  \geq\  -2(\min_{\pa \Omega}f-f(x_0))$$
is a direct consequence
Proposition~\ref{pr.upper-bound} together with 
 the inequality $\lambda^L_{1,h}\sup_{x\in \overline \Omega} \mathbb E_{x}[\tau_{\Omega^c}]\ge 1$. This standard inequality can be derived  as follows. Define the smooth function $g:x\in \overline \Omega \mapsto v_h-\lambda^L_{1,h}  \mathbb E_{x}[\tau_{\Omega^c}]$, where $v_h$ is the principal eigenvalue of $L_h$ satisfying $v_h>0$ in $\Omega$ and  $\sup_{\overline \Omega} v_h=1$. It then holds $L_hg= \lambda^L_{1,h}(v_h-1)\le 0$. Hence, according to the weak maximum principle~\cite[Theorem 1 in Section 6.4.1]{Eva}, we have $g\le 0$ on $\overline \Omega$ and thus the announced inequality.
\medskip
  
\noindent
\textbf{Step 2. Proof of  \eqref{eq.ex}.}  
Injecting the equality in \eqref{eq.Starting-Point3} into the relation \eqref{eq.Starting-Point2} leads to
the existence of $c>0$ such that, for every $h>0$ small enough,
\begin{equation}
\label{eq.Starting-Point4}
\frac{1}{\lambda_{1,h}^L}=\frac{\int_{\Omega\setminus \overline{\mathbf{C}}_{\text{min}}(\eta_0)}   \mathbb E_{x}[\tau_{\Omega^c}]  u_{1,h}^{P^*}(x)  \, e^{-\frac {f(x)}h}dx}{\int_{\Omega} u_{1,h}^{P^*}e^{-\frac f h }}+\mathbb E_{x_0}[\tau_{\Omega^c}](1+O(e^{-\frac ch})).
\end{equation}
Moreover, it follows from \eqref{eq.info}, Proposition~\ref{pr.upper-bound}, and \eqref{eq.exLL}    that for some $c>0$
and every $h>0$ small enough,
\begin{equation*}
\label{eq.LL-1}
\lambda_{1,h}^L \, \frac{\int_{\Omega\setminus  \overline{\mathbf{C}}_{\text{min}}(\eta_0)}    \mathbb E_{x}[\tau_{\Omega^c}] \,  u_{1,h}^{P^*}(x)  \, e^{-\frac {f(x)} h}dx}{\int_{\Omega} u_{1,h}^{P^*}e^{-\frac f h}} \le   e^{-\frac ch},
\end{equation*}
  Plugging this estimate  into \eqref{eq.Starting-Point4} leads to 
  $ 1+O(e^{-\frac ch})=\lambda_{1,h}^L\mathbb E_{x_0}[\tau_{\Omega^c}](1+O(e^{-\frac ch}))$ when $h\to 0$.
   Together with Theorem~\ref{th:leveling}, this proves  \eqref{eq.ex}.  
  \medskip

  \noindent
  \textbf{Step 3. Proof of  \eqref{eq.loiP}.}   Set  $m_h= e^{\frac 2h(  \min_{\partial \Omega}f -f(x_0) - \frac{\eta_0} 2)}$.    Consider    a compact subset
 $ K$  
   of $\mathcal A(\{x_0\})$
  and  $\eta_*\in (\eta_0,\min_{\pa \Omega}f-f(x_0))$, so that 
  $\overline{\mathbf{C}}_{\text{min}}(\eta_*) \subset \mathbf{C}_{\text{min}}(\eta_0)$ and $ \overline{\mathbf{C}}_{\text{min}}(\eta_*)\subset \mathcal A(\{x_0\})$ (see \eqref{eq.eta0}).
 We claim  that,  for all $x\in  K$, $y\in \overline{\mathbf{C}}_{\text{min}}(\eta_*)$, and  all $u>0$:
  \begin{equation}
 \label{eq.L3-galves}
 \mathbb P_x[\tau_{\Omega^c}>u]\le \mathbb P_{y}[\tau_{\Omega^c}>u -2m_h] +  R_1  \text{ and } \mathbb P_x[\tau_{\Omega^c}>u]\ge \mathbb P_{y}[\tau_{\Omega^c}>u +m_h] + R_2,  
 \end{equation}
 where,  for $j\in \{1,2\}$, $R_j$ is independent of  $u> 0$ and of $x,y$,  and  satisfies, for some $c>0$ and  every $h$ small enough: $\vert   R_j\vert \le   e^{-\frac ch}$.

  To prove~\eqref{eq.L3-galves}, we first consider the case  when $ K=  \overline{\mathbf{C}}_{\text{min}}(\eta_*)$. 
  Using \eqref{eq.WFE} and the Markov inequality, there exists $c>0$ such that for every $h$ small enough:
 \begin{equation}\label{eq.CHE}
  \sup_{x\in   \overline{\mathbf{C}}_{\text{min}}(\eta_*)  } \,    \mathbb P_x[\tau_{\mathbf{C}^c_{\text{min}}(\eta_0)}> m_h ]    \le e^{-\frac{c}{  h}}.
  \end{equation}
  Recall   that for $x'\in   {\mathbf{C}}_{\text{min}}(\eta_0)$,  $\mu^h_{x'}$ denotes the hitting distribution on $\pa     {\mathbf{C}}_{\text{min}}(\eta_0)$ for the process \eqref{eq.langevin} when $X_0=x'$  (see \eqref{eq.Law1}) and   $\Vert \mu^h_{x'}-\mu^h_{y'}\Vert\le e^{-\frac ch}$ uniformly in $x',y'\in  \overline{\mathbf{C}}_{\text{min}}(\eta_*)$.  
We then have for all $u'>0$, $v>0$, and $x',y'\in  \overline{\mathbf{C}}_{\text{min}}(\eta_*)$, using the strong Markov property, 
\begin{align}
\nonumber
\mathbb P_{x'}[\tau_{\Omega^c}>u']\ge  \mathbb P_{x'}[\tau_{\Omega^c}>u'+ \tau_{  {\mathbf{C}}^c_{\text{min}}(\eta_0)} ] &= \int \mathbb P_z[\tau_{\Omega^c}>u'] \mu^h_{x'}(dz) \\
\nonumber
&\ge \int \mathbb P_z[\tau_{\Omega^c}>u'] \mu^h_{y'}(dz)- \Vert \mu^h_{x'}-\mu^h_{y'}\Vert\\
\nonumber
&= \mathbb P_{y'}[\tau_{\Omega^c}>u'+ \tau_{ {\mathbf{C}}^c_{\text{min}}(\eta_0)} ]- \Vert\mu^h_{x'}-\mu^h_{y'}\Vert\\
\label{eq.CHE2}
&\ge \mathbb P_{y'}[\tau_{\Omega^c}>u'+v] - \mathbb P_{y'}[ \tau_{  {\mathbf{C}}^c_{\text{min}}(\eta_0) }>v] - \Vert \mu^h_{x'}-\mu^h_{y'}\Vert.
\end{align}
Let $u,v>0$ and $x,y \in \overline{\mathbf{C}}_{\text{min}}(\eta_*)$. If $u-m_h\le  0$, $ \mathbb P_x[\tau_{\Omega^c}>u]\le 1=\mathbb P_{y}[\tau_{\Omega^c}>u -m_h]$. In addition, using \eqref{eq.CHE} and \eqref{eq.CHE2} with $(x',y',u',v)=(x,y,u,m_h)$ and  also with $(x',y',u',v)=(y,x,u-m_h,m_h)$ (when $u-m_h> 0$), we deduce that for all $x,y\in   \overline{\mathbf{C}}_{\text{min}}(\eta_*)$ and  all $u>0$:
 \begin{align}
 \label{eq.L3-galves-2}
 \mathbb P_x[\tau_{\Omega^c}>u]\le \mathbb P_{y}[\tau_{\Omega^c}>u -m_h] +  r_1  \text{ and } \mathbb P_x[\tau_{\Omega^c}>u]\ge \mathbb P_{y}[\tau_{\Omega^c}>u +m_h] + r_2,  
 \end{align}
 where,  for $j\in \{1,2\}$, $r_j$ is independent of  $u>0$ and $x,y\in  \overline{\mathbf{C}}_{\text{min}}(\eta_*)$,  and satisfies $\vert   r_j\vert \le   e^{-\frac ch}$ for some  $c>0$ independent of $h$. Notice that \eqref{eq.L3-galves-2} implies~\eqref{eq.L3-galves}  when $ K=  \overline{\mathbf{C}}_{\text{min}}(\eta_*)$. 
 Let us mention that the proof of \eqref{eq.L3-galves-2} is inspired by the one of~\cite[Lemma 3]{galves-olivieri-vares-87}. 
 
  Let us now prove~\eqref{eq.L3-galves} for an arbitrary $K\subset   \mathcal A(\{x_0\})$. Take such a $K$ and consider $T_K \ge 0$ as in the proof of Theorem \ref{th:leveling}. We have for every $x\in K$ and $y\in \overline{\mathbf{C}}_{\text{min}}(\eta_*)$, using the Markov property,~\eqref{eq.condMVD},  \eqref{eq.condOmega}, and the second inequality in \eqref{eq.L3-galves-2},
  \begin{align*}
  \mathbb P_x[\tau_{\Omega^c}>u] &\ge \mathbb P_x[\tau_{\Omega^c}>u+T_K] \\
  &\ge \mathbb P_x\Big[\tau_{\Omega^c}>u+T_K, \, \sup_{t\in [0,T_K]}|X_t-\varphi_t(x)|\le \delta\Big ]\\
  &=  \mathbb E_x\big [ \mathbb P_{X_{T_K}}[\tau_{\Omega^c}>u]  \, \mathbf 1_{\sup_{t\in [0,T_K]}|X_t-\varphi_t(x)|\le \delta}  \big]\\
  &\ge \big(\mathbb P_{y}[\tau_{\Omega^c}>u+m_h] +r_2\big)(1+O(e^{-\frac ch})).
  \end{align*}
  This proves the second inequality in \eqref{eq.L3-galves}. 
  Now let $h>0$ be small enough so that $m_h>T_K$. Then, using the Markov property,~\eqref{eq.condMVD},  \eqref{eq.condOmega}, and the first inequality in \eqref{eq.L3-galves-2},  it holds for all $u'>0$,  $x\in K$, and $y\in \overline{\mathbf{C}}_{\text{min}}(\eta_*)$:
  \begin{align*}
  \mathbb P_x[\tau_{\Omega^c}>u'+m_h]&\le   \mathbb P_x[\tau_{\Omega^c}>u'+T_K]\\
  &=  \mathbb P_x\Big[\tau_{\Omega^c}>u'+T_K,\sup_{t\in [0,T_K]}|X_t-\varphi_t(x)|\le \delta\Big]\\
  &\quad +\mathbb P_x\Big[\tau_{\Omega^c}>u'+T_K, \sup_{t\in [0,T_K]}|X_t-\varphi_t(x)|> \delta\Big]    \\
  &=\mathbb P_x\Big[\tau_{\Omega^c}>u'+T_K,\sup_{t\in [0,T_K]}|X_t-\varphi_t(x)|\le \delta\Big] +O(e^{-\frac ch})\\
  &=    \mathbb E_x\big [ \mathbb P_{X_{T_K}}[\tau_{\Omega^c}>u']\, \mathbf 1_{\sup_{t\in [0,T_K]}|X_t-\varphi_t(x)|\le \delta}    \big] +O(e^{-\frac ch}) \\
  &\le \big(\mathbb P_{y}[\tau_{\Omega^c}>u'-m_h] +r_1\big)(1+O(e^{-\frac ch}))+ O(e^{-\frac ch}) . 
  \end{align*}
Pick $u>0$. Then, the first inequality in~\eqref{eq.L3-galves} is a consequence of the previous inequality when $u-2m_h>0$ (use it with $u'=u-m_h>0$) and of the fact that when   $u-2m_h\le  0$, $ \mathbb P_x[\tau_{\Omega^c}>u]\le 1=\mathbb P_{y}[\tau_{\Omega^c}>u -2m_h]$. 
 This concludes the proof of~\eqref{eq.L3-galves}.   
 
 We are now in position to prove Equation~\eqref{eq.loiP}. 
  According to~\eqref{eq.expo-law}, it holds for all $s\in\mathbb R$, 
\begin{equation}
\label{eq.GT}
 \mathbb P_{\nu_h}[\tau_{\Omega^c}>s]=e^{-\lambda_{1,h}^L\max(s,0)},
   \end{equation}
and, according to~\eqref{eq.exLL}, there exists $c>0$ such that  for all $h$ small enough and for all $s\in \mathbb R$:
\begin{align}
\nonumber
\mathbb P_{\nu_h}[\tau_{\Omega^c}>s]   &=\frac{\int_{ \Omega}  \mathbb P_{y}[\tau_{\Omega^c}>s]\, u_{1,h}^{P^*}(y) \, e^{-\frac{1}{h}f(y)} dy }{\int_{\Omega}u_{1,h}^{P^*}e^{-\frac f h }} \\
 \label{eq.Loi-o}
 &=    \frac{\int_{  \overline{\mathbf{C}}_{\text{min}}(\eta_*)} \!\!\!  \mathbb P_{y}[\tau_{\Omega^c}>s]\, u_{1,h}^{P^*}(y)\, e^{-\frac{1}{h}f(y)} dy }{\int_{\Omega}u_{1,h}^{P^*}e^{-\frac f h }}+     O(e^{-\frac ch}).
\end{align}
Moreover, from~\eqref{eq.info}, there exists $c>0$ such that  for every $h$ small enough:
$$\lambda_{1,h}^L m_h\le e^{-\frac ch}.$$
Consider $t> 0$ and $x\in K\subset \mathcal A(\{x_0\})$.
Taking  $s= {t}/{\lambda_{1,h}^L}  + m_h>0$ in~\eqref{eq.GT} and using~\eqref{eq.L3-galves}  and \eqref{eq.Loi-o},
there exists $h_{0}>0$ which does not depend on $t> 0$ and on $x\in K$
such that, taking $c>0$ smaller if necessary (but not depending on $t> 0$ and on $x\in K$),  it holds for every $h\in(0,h_{0}]$:
\begin{align*}
  \mathbb P_{x}[\tau_{\Omega^c}> {t}/{\lambda_{1,h}^L}]  \ge e^{-\lambda_{1,h}^L  ({t}/{\lambda_{1,h}^L}+ m_h )} - e^{-\frac ch}
   \,\ \text{and then}\,\  \mathbb P_{x}[\tau_{\Omega^c}> {t}/{\lambda_{1,h}^L}] - e^{-t} \ge -\lambda_{1,h}^L m_h -e^{-\frac ch}
  \ge -2e^{-\frac ch}.
\end{align*}
Similarly, taking now
$s= {t}/{\lambda_{1,h}^L}  - 2m_h$ and $h_{0}>0$ smaller if necessary (but not depending on $t> 0$ and on $x\in K$), it holds for every $h\in(0,h_{0}]$:
\begin{align*}
  \mathbb P_{x}[\tau_{\Omega^c}> {t}/{\lambda_{1,h}^L}]  - e^{-t}\le e^{-\lambda_{1,h}^L \max({t}/\lambda_{1,h}^L- 2m_h,0  )}
  + e^{-\frac ch} - e^{-t}
& \leq \begin{cases} 3 \lambda_{1,h}^L m_h +e^{-\frac ch}  & \text{if $ {t}>2{\lambda_{1,h}^L}m_h$}\\
 t+ e^{-\frac ch}  & \text{if $ {t}\leq 2{\lambda_{1,h}^L}m_h$}\end{cases}  \\
 & \leq 4 e^{-\frac ch} \,.
\end{align*}
Hence, for every compact $K\subset \mathcal A(\{x_0\})$, there exists $c>0$ and  $h_0>0$ such that for all $h\in (0,h_0]$:
$$
\sup_{x\in K,t\in\mathbb R^{+}} \vert \mathbb P_{x}[\tau_{\Omega^c}> {t}/{\lambda_{1,h}^L}] - e^{-t}\vert \le  e^{-\frac ch}, 
$$
which completes the proof of \eqref{eq.loiP}.

\section{Proof of Theorem~\ref{th:main2}}
\label{sec.VP}

In this last section, we prove Theorem~\ref{th:main2}. More precisely, we prove the following equivalent result 
on the principal eigenvalue $\lambda_{1,h}^P$  of  $P_h$ (see \eqref{eq.unitary} and the lines below, and~Proposition~\ref{pr.spectre}). 

\begin{theorem}\label{th:VP}
Assume  \eqref{ortho}, \eqref{well},  \eqref{div}, and \eqref{normal}. 
Then, the principal eigenvalue $\lambda_{1,h}^P$  of  $P_h$  satisfies, when  $h\to 0$:
$$
\lambda_{1,h}^P=\big(\kappa^P_1 \,h^{\frac12}  +  \kappa^P_2\, h + O(h^{\frac54})\big)   \, e^{-\frac 2h (\min_{\pa \Omega}f -f(x_0))},
$$ 
 where   $\kappa^P_1=2\kappa^L_1$ and $\kappa^P_2=2\kappa^L_2$ (see \eqref{eq.Pre1}),
and the error term $O(h^{\frac54})$ is actually of order $O(h^{\frac32})$ when $\kappa^P_1=0$ or $\kappa^P_2=0$, i.e. when $\nabla f(z)= 0$ for every $z\in \pa \mathbf{C}_{{\rm min}} \cap \pa \Omega$
or $\nabla f(z)\neq 0$ for every $z\in \pa \mathbf{C}_{{\rm min}} \cap \pa \Omega$.
   \end{theorem}

\subsection{General strategy}

\noindent
In order to prove Theorem~\ref{th:VP}, we want to construct, for every $h$ small enough, a 
very accurate approximation $\mathsf f_{1,h}$ of the eigenmode $u_{1,h}^{P}$ of  $P_h$.
The next proposition gives conditions ensuring that such an approximation  is sufficiently accurate.

    \begin{proposition}
    \label{pr.strategy}
  Assume   \eqref{ortho} and   \eqref{well}. Assume moreover that, for all $h$ small enough, there exists a $L^2(\Omega)$-normalized function $\mathsf f_{1,h} \in D(P_h)$ such that the following properties hold:
\begin{align}
\label{E1}\tag{E1} &\langle P_h \mathsf f_{1,h},\mathsf f_{1,h} \rangle_{L^2(\Omega)}  =  \big(\kappa^P_1h^{\frac12} \, +  \kappa^P_2 h + O(h^{\frac32})\big)   \, e^{-\frac 2h (\min_{\pa \Omega}f -f(x_0))},\\
\label{E2}\tag{E2}  &\Vert  P_{h} \mathsf f_{1,h}  \Vert _{L^2(\Omega)} ^2= O(h^2)\,   \langle  P_{h} \mathsf f_{1,h},\mathsf f_{1,h} \rangle_{L^2(\Omega)} , \\
\label{E3}\tag{E3}  & \Vert  P_{h}^* \mathsf f_{1,h}  \Vert _{L^2(\Omega)} ^2= 
\big(\kappa^P_1h^{\frac12} \,O(h^{2}) +  \kappa^P_2 h\,O(h) \big)   \, e^{-\frac 2h (\min_{\pa \Omega}f -f(x_0))}.
\end{align}
   Then, the asymptotic equivalent of Theorem~\ref{th:VP} holds, i.e. $$
\lambda_{1,h}^P=\big(\kappa^P_1 \,h^{\frac12} +  \kappa^P_2 \,h + O(h^{\frac54})\big)   \, e^{-\frac 2h (\min_{\pa \Omega}f -f(x_0))}\ \ \ \text{when $h\to 0$,}$$
where the error term $O(h^{\frac54})$ is actually of order $O(h^{\frac32})$ when $\kappa^P_1=0$ or $\kappa^P_2=0$.
   \end{proposition}

   \begin{proof} 
According to the argument leading to \eqref{eq.1-Pi} and to \eqref{E1}, \eqref{E2}, we have, for some $C,c>0$ 
and every  $h>0$ small enough:
\begin{equation}
\label{eq.1-Pi'}
 \Vert  (1-\pi_h^P) \mathsf f_{1,h} \Vert_{L^2(\Omega)}\le  C h^{-1}  \Vert P_h \mathsf f_{1,h}   \Vert_{L^2(\Omega)}
 \ \ \ \text{and thus}\ \ \ \pi_h^P \mathsf f_{1,h} = \mathsf f_{1,h} + O(e^{-\frac ch})\ \ \text{in $L^{2}(\Omega)$}.
\end{equation}
Since $P_{h}  \pi_h^P \mathsf f_{1,h}=\lambda_{1,h}^P \pi_h^P\mathsf f_{1,h}$, it follows from the second estimate of \eqref{eq.1-Pi'} that
\begin{align*}
\lambda_{1,h}^P &= \frac{\langle P_{h}  \pi_h^P \mathsf f_{1,h},\pi_h^P \mathsf f_{1,h}\rangle_{L^2(\Omega)}}{\|\pi_h^P 
\mathsf f_{1,h}\|^{2}_{L^{2}(\Omega)}}\\
&=    (1+O(e^{-\frac ch}))\Big[\langle P_{h}   \mathsf f_{1,h},\mathsf f_{1,h} \rangle_{L^2(\Omega)}+ \langle P_{h} ( \pi_h^{P}-1)  \mathsf f_{1,h},\mathsf f_{1,h} \rangle_{L^2(\Omega)} +  \langle P_{h}  \pi_h^{P}  \mathsf f_{1,h},( \pi_h^{P}-1)\mathsf f_{1,h} \rangle_{L^2(\Omega)}\Big].
\end{align*}
Moreover \eqref{eq.1-Pi'}, $\pi_h^{P}=O(1)$ (see \eqref{eq.pih==}), and the Cauchy-Schwarz inequality imply:
$$| \langle P_{h}  \pi_h^{P}  \mathsf f_{1,h},( \pi_h^{P}-1)\mathsf f_{1,h} \rangle_{L^2(\Omega)}|=
| \langle \pi_h^{P} P_{h}   \mathsf f_{1,h},( \pi_h^{P}-1)\mathsf f_{1,h} \rangle_{L^2(\Omega)}|
=\Vert P_h\mathsf f_{1,h}   \Vert_{L^2(\Omega)}^2 \, O(h^{-1}  ) $$
and
$$|\langle P_{h} ( \pi_h^{P}-1) \mathsf f_{1,h},\mathsf f_{1,h} \rangle_{L^2(\Omega)}|= |\langle  ( \pi_h^{P}-1)  \mathsf f_{1,h},P_{h} ^*\mathsf f_{1,h} \rangle_{L^2(\Omega)}|= \Vert P_h\mathsf f_{1,h}   \Vert_{L^2(\Omega)}   \Vert P_h^*\mathsf f_{1,h}   \Vert_{L^2(\Omega)}\, O(h^{-1}  ).$$   
Using in addition \eqref{E1}, \eqref{E2}, and \eqref{E3}, it follows that
$$
\lambda_{1,h}^P=   (1+O(e^{-\frac ch}))\langle P_{h}   \mathsf f_{1,h},\mathsf f_{1,h} \rangle_{L^2(\Omega)}(1+O(h^{\ell})) = \langle P_{h}   \mathsf f_{1,h},\mathsf f_{1,h} \rangle_{L^2(\Omega)}(1+O(h^{\ell})),
$$
where $\ell=\frac12$ when $\kappa^P_1=0$ (and thus $\kappa^P_2\neq0$),
$\ell=1$ when $\kappa^P_2=0$ (and thus $\kappa^P_1\neq0$),
and  $\ell=\frac34$ when $\kappa^P_1\kappa^P_2\neq0$.
This leads to the statement of Proposition~\ref{pr.strategy}.
\end{proof}

\subsection{Proof of Theorem~\ref{th:VP}}
From now on, we assume  \eqref{ortho}, \eqref{well},  \eqref{div}, and \eqref{normal}. 
According to Proposition \ref{pr.strategy}, 
it is sufficient 
to construct a quasi-mode $\mathsf f_{1,h}$ satisfying \eqref{E1}, \eqref{E2}, and \eqref{E3}
(see Proposition~\ref{pr.QM-1} below).
The construction below is strongly inspired by to the ones made in \cite{DoNe2,LePMi20}.

\subsubsection{System of coordinates near the points of $\pa \mathbf{C}_{{\rm min}}\cap \pa \Omega$}
\label{sec:coord}

Recall that $\pa \mathbf{C}_{{\rm min}}\cap \pa \Omega\neq \emptyset$ (see \eqref{well}) and that   $\pa \mathbf{C}_{{\rm min}}\cap \pa \Omega$ has a finite cardinality (see \eqref{eq.Ccons}). 
Take  $z\in \pa \mathbf{C}_{{\rm min}}\cap \pa \Omega$.  
There exists  a neighborhood~$\mathsf V_z$   of $z$ in~$\overline \Omega$ and  a   coordinate system 
\begin{equation}\label{eq.cv-pa-omega-nablafnon0}
p\in \mathsf V_z \mapsto v=(v',v_d)=(v_1,\ldots,v_{d-1},v_{d})\in \mathbb R^{d-1}\times \mathbb R_-
\end{equation}
 such that 
\begin{equation}\label{eq.cv-pa-omega-nablafnon02}
v(z)=0, \ \ \{p\in \mathsf V_z, \, v_d(p)<0\}=  \Omega\cap  \mathsf V_z, \ \{p\in \mathsf V_z,\,  v_d(p)=0\}=\pa \Omega \cap \mathsf V_z,
\end{equation}
 and
\begin{equation*}
\forall i,j\in\{1,\dots,d\},\ \ \ g_z\Big (\frac{\pa}{\pa v_i}(z),\frac{\pa}{\pa v_j}(z)\Big )=\delta_{ij}
\quad\text{and}
\quad
\frac{\pa}{\pa v_d}(z) = n_\Omega(z),
 \end{equation*}
where $ g_z$ is the metric tensor in the new coordinates.
We denote by $G=(G_{ij})_{1\leq i,j\leq d}$ its matrix, by $G^{-1}=(G^{ij})_{1\leq i,j\leq d}$ its inverse,
and by $(\mathsf e_1,\dots, \mathsf e_d)= ({}^t(1,0,\ldots,0),\dots,{}^t(0,\ldots,0,1)) $ the canonical basis of $\mathbb R^{d}$
so that, defining $J:=\jac\, v^{-1}$, we have  
\begin{equation}
\label{eq.G1}
G=\, ^t J J\,,\ \ G(0) =(\delta_{ij})\ \ \text{i.e.}\ \ {}^t J(0)=J^{-1}(0)\,,\ \ \text{and}\ \ n_\Omega(z) = J(0)\mathsf e_d\,.
 \end{equation}

In addition,  
defining $\hat  f:=f\circ v^{-1}$ the function $f$ in the new coordinates:
\begin{itemize}
\item[] \textbf{Case 1, when $\nabla f(z)\neq 0$:} 
According for example to~\cite[Section 3.4]{HeNi1}, 
the $v$-coordinates can be chosen such that
\begin{equation}\label{eq.cv-pa-omega-nablafnon03}
\hat f(v',v_d)=f(z)+\mu(z) v_d+ \frac 12 \,v' \Hess \hat f|_{\{v_d=0\}}(0)\,  {}^tv',
\end{equation} 
where we recall  that  $\mu(z):=\pa_{n_{\Omega}}f(z)>0$ and that, thanks to~\eqref{normal},    $0$ is   a non degenerate  (global) minimum of $\hat f|_{\{v_d=0\}}$.\medskip

\item[] \textbf{Case 2, when $\nabla f(z)= 0$:} 
We have $\nabla (\hat f+ \vert \mu(z)\vert v_d^2)(0)=0$ and, according to \eqref{eq.G1}:
\begin{equation}\label{eq.det--}
\Hess(\hat f+ \vert \mu(z)\vert v_d^2)(0)= {}^t J(0)\big(\Hess f(z)+ 2 \vert \mu(z)\vert n_\Omega(z)n_\Omega(z)^*\big)J(0),
\end{equation}
where we recall that,
from  \eqref{normal},   $n_\Omega(z)$ is an eigenvector associated with the negative eigenvalue~$\mu(z)$ of $\Hess f(z)+\, ^t\mathsf L(z)$.
Note also that the matrix in \eqref{eq.det--}
 is positive definite according to  Lemma~\ref{le.lep-michel}.
 \end{itemize}

\noindent 
In particular, up to choosing   $\mathsf V_z$   smaller, one can assume
that when $\nabla f(z)\neq 0$,
\begin{equation}\label{eq.min_coo1}
  \argmin_{  v(\overline{ \mathsf V_z })} \big( \hat f(v) -2\mu(z)v_d\big)=\{ 0\},
\end{equation} 
and when $\nabla f(z)=0$,
\begin{equation}\label{eq.min_coo2}
 \argmin_{  v(\overline{ \mathsf V_z })} \big( \hat f(v) +\vert \mu(z)\vert v_d^2\big)=\{ 0\} .
\end{equation} 
\noindent
 For $\delta_1>0$ and $\delta_2>0$ small enough, one finally defines  the following neighborhood of $z$ in~$\pa \Omega$,
\begin{equation*}\label{eq.set11}
\mathsf V_{\pa \Omega}^{\delta_2}(z):=\big \{p\in \mathsf V_z, \, v_d(p)=0\text{ and }  \vert v'(p)\vert \le \delta_2\big \} \, \text{ (see~\eqref{eq.cv-pa-omega-nablafnon0}-\eqref{eq.cv-pa-omega-nablafnon02})},
 \end{equation*}
 and  the following neighborhood of $z$  in $\overline \Omega$,
\begin{equation}\label{eq.vois-11-pc}
\mathsf V^{\delta_1,\delta_2}_{{\,  \overline \Omega  }}( z)=\big \{p\in \mathsf V_z,   \vert v'(p)\vert \le \delta_2  \text{ and } v_d(p)\in  [-2\delta_1, 0]\big  \}.
\end{equation}

The set defined in  \eqref{eq.vois-11-pc}   is a cylinder centered at~$z$ in the $v$-coordinates.
 Up to choosing $\delta_1>0$ and $\delta_2>0$ smaller, we can assume the cylinders
 $\mathsf V^{\delta_1,\delta_2}_{{\,  \overline \Omega  }}( z)$, $z\in \pa \mathbf{C}_{{\rm min}}\cap \pa \Omega$,
   pairwise disjoint. Since $f(z)=\min_{\pa \Omega}f>f(x_0)$,    we can also assume that    
   \begin{equation}\label{eq.inclur3}
  \min_{\mathsf V^{\delta_1,\delta_2}_{\overline \Omega}(z)}f>f(x_0)\ \text{ (so in particular    $x_0\notin \mathsf V^{\delta_1,\delta_2}_{\overline \Omega}(z)$)},  
  \end{equation}
  and, in view of \eqref{eq.OpaOmega}, 
   \begin{equation}\label{eq.min-v}
   \argmin_{\, \mathsf V^{\delta_2}_{\pa \Omega}(z)}f= \{z\}.
  \end{equation}
 The parameter $\delta_2>0$ is  now kept fixed.
Finally, according to \eqref{eq.min-v}
and up to choosing $\delta_1>0$ smaller, there exists $r>0$  such that: 
\begin{equation}\label{eq.inclur1}
 \big \{p\in \mathsf V_z,   \vert v'(p)\vert =\delta_2  \text{ and } v_d(p)\in  [-2\delta_1, 0]\big  \} \subset\{f\ge f(z)+r\}.   
\end{equation}

 We end  this section by defining 
 locally near each $z\in \pa   \mathbf{C}_{{\rm min}}\cap \pa \Omega$
 a function $\varphi_z$  in 
  the above $v$-coordinates,
  and used in the next section to define the quasi-mode $\mathsf f_{1,h}$ near $z$. 
 Let $\chi\in C^\infty({\mathbb R}^{-},[0,1])$ be  a cut-off  function    such that 
\begin{equation}\label{eq.chi-cut}
\text{supp } \chi\subset  [- \delta_1 ,  0] \, \text{ and } \,   \chi=1 \text{ on } \Big [-\frac{\delta_1}{2}, 0\Big].
  \end{equation}
For every $z\in \pa   \mathbf{C}_{{\rm min}}\cap \pa \Omega$,  the  function $\varphi_z$ 
is defined as follows
 (see \eqref{eq.cv-pa-omega-nablafnon0},  \eqref{eq.cv-pa-omega-nablafnon02}, and \eqref{eq.vois-11-pc}): 
\begin{itemize}
\item[] \textbf{Case 1, when $\nabla f(z)\neq 0$:}
 \begin{equation}\label{eq.qm-local11}
\forall v=(v',v_d)  \in  v\big (\mathsf V^{\delta_1,\delta_2}_{\,\overline \Omega}(z)\big), \  \  \varphi_{z} (v',v_d):=\frac{  \int_{v_d}^0\chi(t)e^{\frac 2 h \mu(z)\, t}  dt}{\int_{-2\delta_1}^0\chi(t)\, e^{\frac 2 h \mu(z)\, t}  dt},
  \end{equation}
  where we recall that  $\mu(z)=\partial_{n_\Omega}f(z)>0$, see \eqref{eq.cv-pa-omega-nablafnon03}.\medskip

\item[] \textbf{Case 2, when $\nabla f(z)= 0$:}
   \begin{equation}\label{eq.qm-local12}
\forall v=(v',v_d)  \in  v\big (\mathsf V^{\delta_1,\delta_2}_{\,\overline \Omega}(z)\big), \  \  \varphi_{z} (v',v_d):=\frac{  \int_{v_d}^0\chi(t)e^{-\frac 1h\vert \mu(z)\vert  \, t^2}  dt}{\int_{-2\delta_1}^0\chi(t)\, e^{-\frac 1h\vert \mu(z)\vert  \, t^2}  dt},
  \end{equation}
  where we recall that~$\mu(z)$ is the   negative eigenvalue of $\Hess f(z)+ {}^t\mathsf L(z)$, see \eqref{eq.det--}.
\end{itemize}  
  
In both cases: 
      \begin{equation}\label{eq.qm-local1-property12}
  \left\{
    \begin{array}{ll}
        \vspace{0.2cm}
    \varphi_{z}\in   C^\infty\big (v\big (\mathsf V^{\delta_1,\delta_2}_{\,\overline \Omega}(z)\big),[0,1]\big )\ \ \text{only depends on $v_{d}$, $\varphi_z(v',0)=0$, and }\\
\forall (v',v_d)\in v\big (\mathsf V^{\delta_1,\delta_2}_{\,\overline \Omega}(z)\big),\, \varphi_{z}(v',v_d)=1 \text{  when  } v_d \in [-2\delta_1,-\delta_1].
 \end{array}
\right.
\end{equation}

 \subsubsection{Definition of the quasi-mode $\mathsf f_{1,h}$}

We now define $\mathsf f_{1,h}$, using the  $v$-coordinates and the above $\varphi_z$, $z\in \pa   \mathbf{C}_{{\rm min}}\cap\pa \Omega$. 
Before, we recall that we defined in
\eqref{eq.vois-11-pc} pairwise disjoint   cylinders around the  $z\in \pa   \mathbf{C}_{{\rm min}}\cap\pa \Omega$ which satisfy
\eqref{eq.inclur3}, \eqref{eq.min-v}, and~\eqref{eq.inclur1}. 
On the other hand,  for every
$p\in \pa   \mathbf{C}_{{\rm min}}\setminus\pa \Omega$:
$p\in \Omega$ and thus $\nabla f(p)\neq 0$, which implies 
that $\{f<f(p)\}\cap   B(p,r)$ is connected for every~$r>0$ small enough and thus included in $ \mathbf{C}_{{\rm min}}$.

These considerations imply the existence of the following subsets 
$\mathbf{C}_{{\rm low}}$ and~$\mathbf{C}_{{\rm up}}$ of $\Omega$.

\begin{proposition}\label{pr.omega1}  Assume  \eqref{ortho}, \eqref{well},  \eqref{div}, and \eqref{normal}.  Then,  there exist two $\mathcal C^\infty$  connected open sets~$\mathbf{C}_{{\rm low}}$ and~$\mathbf{C}_{{\rm up}}$ of $\Omega$  satisfying the following properties:
 \begin{enumerate}
 \item  It holds
  $\overline{\mathbf{C}}_{{\rm min}} \subset \mathbf{C}_{{\rm up}} \cup \pa \Omega\  \text{ and }\ \argmin_{\,  \overline{\mathbf{C}}_{{\rm up}}}f =\{x_0\}.$
  \item The set $\overline{\mathbf{C}}_{{\rm up}}$ is a neighborhood  in $\overline\Omega$ of each $ \mathsf V^{\delta_1,\delta_2}_{\,\overline \Omega}(z)$, $z\in \pa   \mathbf{C}_{{\rm min}}\cap\pa \Omega$.
 \item  It holds  $\overline{\mathbf{C}}_{{\rm low}}\subset \mathbf{C}_{{\rm up}}$ and the strip $\overline{\mathbf{C}}_{{\rm up}}\setminus \mathbf{C}_{{\rm low}}$ satisfies
\begin{equation}\label{stripS}
\exists c>0\,,\ f\geq f(x_{0})+c \ \text{on}\ \overline{\mathbf{C}}_{{\rm up}}\setminus \mathbf{C}_{{\rm low}}
\quad\text{and}\quad
 \overline{\mathbf{C}}_{{\rm up}}\setminus \mathbf{C}_{{\rm low}}= 
 \bigcup \limits_{z\in \pa   \mathbf{C}_{{\rm min}}\cap\pa \Omega}  \mathsf V^{\delta_1,\delta_2}_{\,\overline \Omega}(z)   \  \bigcup \   \mathsf O,
\end{equation}
where the subset $\mathsf O$ of $\overline\Omega	$  is such that:
\begin{equation*}\label{eq.SS}
\exists c>0\,, \  f\ge \min_{\pa \Omega}f + c  \ \text{on}\ \mathsf O.
\end{equation*}
\end{enumerate}
\end{proposition}

\begin{figure}[h!]
\begin{center}
\begin{tikzpicture}[scale=0.7]
\tikzstyle{vertex}=[draw,circle,fill=black,minimum size=6pt,inner sep=0pt]
\draw (-0.2,-2.8)--(-0.8,2.8);
\draw[->,thick] (-0.5,0)--(-2,-0.2);
 \draw (-3.3,-0.2) node[]{\tiny{$ n_\Omega(z_1)$}}; 
 \draw[ultra thick]  (-0.63,1.3)--(0.3,1.4);
  \draw[ultra thick]  (-0.4,-1)--(0.5,-0.9); 
\draw  (0.3,1.4)--(0.5,-0.9);
 \draw (0,4.6) node[]{$\pa \Omega$}; 
    \draw (-0.5,0)   node {\Large{$\bullet$}} ; 
   \draw (5,0)   node {\Large{$\bullet$}} ;
  \draw (5.3,-0.4) node[]{$ x_0$};
    \draw (2,-5.8) node[]{$\mathbf{C}_{{\rm up}}$};
    \draw [->] (2,-5.2) to (2,-3) ;
    \draw [->] (2,-5.2) to (2.5,-2.3) ;  
  \draw [dotted, very thick]   (10.98 ,3.2) -- (11.7 ,3.2);
 \draw  (14.6 ,3.2) node[]{{\small $\pa \mathbf{C}_{{\rm min}}\subset \{f= \min_{\pa \Omega}f\}$}}; 
     \draw (7,2.4) node[]{$\mathbf{C}_{{\rm low}}$};
     \draw (7.3,-0.2) node[]{$\mathbf{C}_{{\rm low}}$};
      \draw (7,-2.4) node[]{$\mathbf{C}_{{\rm low}}$};
      \draw (5,0.8) node[]{{\tiny $\mathbf{C}_{{\rm min}}= \overline \Omega\cap\big\{ f<\min_{\pa \Omega}f \big\}$}}; 
  \draw[very thick, dotted]  (-0.5,0) to (2,-1) ;
   \draw[very thick, dotted]  (-0.6,0) to (2,1) ;
  \draw (10,-2.8)--(10,2.8);
\draw[ultra thick]  (10,1.5)--(9,1.5);
\draw[ultra thick]  (10,-1.5)--(9,-1.5);
\draw (9,1.5)--(9,-1.5);  
  \draw (10,0)  node[vertex,label=east: {$z_2$}](v){};
  \draw[very thick, dotted]  (8,1.4) ..controls (10.6,0.5) and   (10.6,-0.5)  .. (8,-1.4) ; 
   \draw[very thick, dotted]  (2,1) ..controls (4,2) and   (6,2)  .. (8,1.4) ;
\draw[very thick, dotted]  (2,-1) ..controls (4,-2) and   (6,-2)  .. (8,-1.4) ;
\draw[thick]  (0.3,1.4) ..controls (0.4,3.8) and   (9,3.9)  .. (9,1.5);
\draw[thick]  (0.5,-0.9) ..controls (0.4,-3.8) and   (9,-3.9)  .. (9,-1.5);
\draw[thick]  (-0.68,1.7) ..controls (-0.4,4.5) and   (10,5.4)  .. (9.99,1.9);
\draw[thick]  (-0.32,-1.7) ..controls (0,-4.5) and   (10,-5.9)  .. (9.99,-2.22);
\draw[thick]  (-0.32,-1.7)--(-0.68,1.7);
\draw[thick]  (9.99,-2.22)--(9.99,1.9);
\draw (1,-2.6) node[]{\tiny{$\mathsf O $}};
\draw (6,3.5) node[]{\tiny{$\mathsf O $}};
\draw (12.9,0.8) node[]{$\mathsf V^{\delta_1,\delta_2}_{\overline \Omega}(z_2)$};
\draw [->] (11.7,0.7) to (9.7,0.99) ;
\draw (-4.3,-3.2) node[]{$\mathsf V^{\delta_1,\delta_2}_{\overline \Omega}(z_1) $};
\draw [->] (-3.5,-2.7) to (0,-0.6) ;
 \draw (-0.3,0.6) node[]{${\small z_1}$};
\draw  (-0.8,2.8) ..controls (-0.9,5.5) and   (10,4.8)  .. (10,2.8);
\draw  (-0.2,-2.8) ..controls (0,-5.5) and   (10,-7)  .. (10,-2.8);
\end{tikzpicture}
\caption{Schematic representation of $\mathbf{C}_{{\rm low}}$, $\mathbf{C}_{{\rm up}}$, and $\mathsf O$ (see Proposition~\ref{pr.omega1}). On the figure, $\pa \mathbf{C}_{{\rm min}}\cap \pa \Omega=\{z_1,z_2\}$ with $\nabla f(z_1)=0 $ and  $\vert \nabla f(z_2)\vert\neq 0$. }
 \label{fig:S}
 \end{center}
\end{figure}

We refer to Figure~\ref{fig:S} for a schematic representation of  $\mathbf{C}_{{\rm low}}$, $\mathbf{C}_{{\rm up}}$, and $\mathsf O$. 
Notice that  Proposition~\ref{pr.omega1}
implies
\begin{equation}\label{eq.minC}
\argmin \limits_{\,  \overline{\mathbf{C}}_{{\rm up}}}f=\argmin\limits_{\,  \overline{\mathbf{C}}_{{\rm low}}}f=\{x_0\}.
\end{equation}

Using the above sets $\mathbf{C}_{ {\rm up} }$ and $ \mathbf{C}_{{\rm low}}$, we define  a
function $\phi_{1,h}:\overline \Omega \to [0,1]$  as follows.

\begin{enumerate}
\item[(i)] For every $z\in \pa   \mathbf{C}_{{\rm min}}\cap\pa \Omega$, $\phi_{1,h}$  is defined on the cylinder $\mathsf V^{\delta_1,\delta_2}_{\,\overline \Omega}(z)$ 
(see \eqref{eq.vois-11-pc})
by
 \begin{equation}\label{eq.qm-local1x}
\forall p \in   \mathsf V^{\delta_1,\delta_2}_{\,\overline \Omega}(z), \  \phi_{1,h}(p):= \varphi_{z} (v(p)),  \  
\text{ see~\eqref{eq.qm-local11} and  \eqref{eq.qm-local12}}.
  \end{equation}
\item[(ii)] From~\eqref{eq.qm-local1-property12}, \eqref{stripS},  and  the fact that $ \overline{\mathbf{C}}_{ {\rm low} }\subset  \mathbf{C}_{ {\rm up} }$  (see Proposition~\ref{pr.omega1}),    
the above function $\phi_{1,h}$ satisfying \eqref{eq.qm-local1x}   can be extended to $\overline \Omega$  so  that
\begin{equation}\label{eq.psix=10}
 \phi_{1,h}=0 \text{ on } \overline \Omega\setminus\mathbf{C}_{ {\rm up} }, \ \ \ \phi_{1,h}=1 \text{ on }  {\mathbf{C}_{ {\rm low} }}, \text{ and } \phi_{1,h}\in C^\infty(\overline \Omega,[0,1]). 
\end{equation}
Moreover, in view of \eqref{eq.qm-local11},~\eqref{eq.qm-local12}, and~\eqref{stripS},
 $\phi_{1,h}$ can be chosen on~$\mathsf O$ such that, for some $C>0$ and   for  every $h$ small enough,  
\begin{equation} 
\label{eq.nablapsix=0}
\forall \alpha \in \mathbb N^d, \, \vert \alpha \vert \in \{1,2\}, \,    \Vert \pa^\alpha  \phi_{1,h} \Vert_{L^\infty(\mathsf O )} \le Ch^{-2}. 
\end{equation}
\end{enumerate}
Notice that~\eqref{eq.psix=10} implies
\begin{equation}\label{eq.psix=nabla}
\supp \nabla  \phi_{1,h} \subset \overline{\mathbf{C}}_{ {\rm up} }\setminus \mathbf{C}_{ {\rm low} }.
\end{equation}

We are now in position to define the quasi-mode $\mathsf f_{1,h}$ for $P_h$. 

\begin{definition}\label{de.QM}
Assume \eqref{ortho}, \eqref{well},  \eqref{div}, and \eqref{normal}. Let $\phi_{1,h}$ be the above function
satisfying \eqref{eq.qm-local1x}--\eqref{eq.nablapsix=0}. We define: 
$$\mathsf f_{1,h}:= \frac{\phi_{1,h} \, e^{-\frac fh}}{ Z_{1,h} } \,, \  \text{where}\ \ Z_{1,h}:= \Vert \phi_{1,h} \, e^{-\frac fh}     \Vert_{L^2(\Omega)}.$$ 
\end{definition}

\subsubsection{Quasi-modal estimates}

\begin{proposition}\label{pr.QM-1}
Assume \eqref{ortho}, \eqref{well},  \eqref{div}, and \eqref{normal}. Let $\mathsf f_{1,h}$ be as introduced in Definition~\ref{de.QM}.  Then,  $\mathsf f_{1,h}$ belongs to $D(P_h)$ and satisfies 
\eqref{E1}, \eqref{E2}, and \eqref{E3} of Proposition \ref{pr.strategy}.
In particular, Theorem~\ref{th:VP} holds true.
\end{proposition}
\begin{proof}
First of all, the relation \eqref{eq.psix=10} implies
$\mathsf f_{1,h} \in C^\infty(\overline \Omega,\mathbb R^{+})$ and $\mathsf f_{1,h}=0$ on $\pa \Omega$, and thus,
since $\mathbf{C}_{ {\rm up} }\subset \Omega$ (see Proposition~\ref{pr.omega1}),
\begin{equation} 
\label{eq.nablapsix-domaine}
\mathsf f_{1,h}\in D(P_h)=H^2(\Omega)\cap H^1_0(\Omega),
\end{equation}

In the following, $c>0$ is a constant independent of $h>0$ which can change from one occurrence to another.
The proof is divided into three steps. 
\medskip

\noindent
\textbf{Step 1. The function  $\mathsf f_{1,h}$ satisfies 
\eqref{E1}.}

\noindent 
\textbf{Asymptotic equivalent of $Z_{1,h}$.} 
From  Definition~\ref{de.QM} and \eqref{eq.psix=10}, we have
\begin{align*}Z_{1,h}^2&=\int_\Omega \phi_{1,h}^2  \,e^{-\frac 2h f}
=\int_{\mathbf{C}_{ {\rm low} }} \phi_{1,h}^2  \,e^{-\frac 2hf}+\int_{\mathbf{C}_{ {\rm up} }\setminus\mathbf{C}_{ {\rm low} }} \phi_{1,h}^2  \,e^{-\frac 2h f} = \int_{\mathbf{C}_{ {\rm low} }} \phi_{1,h}^2  \,e^{-\frac 2hf} +O(e^{-\frac 2h (f(x_0)+c)}),
\end{align*}
 where we used $\Ran\phi_{1,h}  \subset [0,1]$  and $f\ge f(x_0)+c$ on $\mathbf{C}_{ {\rm up} }\setminus\mathbf{C}_{ {\rm low} }$ (see~\eqref{stripS}). 
 Moreover, using $\phi_{1,h} = 1$ on $\mathbf{C}_{ {\rm low} }$
 and   \eqref{eq.minC},  the standard Laplace method implies that when $h\to 0$, 
\begin{equation}\label{Zx}
Z_{1,h}^2=(\pi h)^{\frac d2} \, \big(  \det\Hess  f(x_0)\big )^{-\frac 12}          \ e^{-\frac 2h f(x_0)} \big(1+O(h)\big).
\end{equation}
\textbf{Asymptotic equivalent of $\langle P_{h} \mathsf f_{1,h},\mathsf f_{1,h} \rangle_{L^2(\Omega)}$.} 
First, using~\eqref{eq.nablapsix-domaine} and~\eqref{eq.fried},
$$\langle P_{h} \mathsf f_{1,h},\mathsf f_{1,h} \rangle_{L^2(\Omega)} = \int_{\Omega} \vert \nabla_{f,h} \, \mathsf f_{1,h}\vert^2.$$ 
In addition, from  Definition~\ref{de.QM} and~\eqref{eq.psix=nabla},  $
\nabla_{f,h} \mathsf f_{1,h}= Z_{1,h}^{-1} \, h e^{-\frac fh} \nabla\phi_{1,h} $  is supported in $\overline{\mathbf{C}}_{ {\rm up} }\setminus \mathbf{C}_{ {\rm low} }$. 
Hence, from (3) in Proposition~\ref{pr.omega1},~\eqref{eq.nablapsix=0},
and~\eqref{Zx},
we have for every $h$ small enough: 
\begin{align}
\nonumber
\langle P_{h} \mathsf f_{1,h},\mathsf f_{1,h} \rangle_{L^2(\Omega)}  &=
\sum_{z\in  \pa  \mathbf{C}_{ {\rm min} }\cap \pa \Omega } \,  \,  \int_{ \mathsf V_{\,\overline \Omega}^{\delta_1,\delta_2}(z)} \vert \nabla_{f,h}   \,  \mathsf f_{1,h}\vert^2 +   O\big (  e^{-\frac 2h (\min_{\pa \Omega}f-f(x_0) +c)}\big )\\
\label{eq.chiUU0}
&= \sum_{z\in  \pa  \mathbf{C}_{ {\rm min} }\cap \pa \Omega } \,  \,Z_{1,h}^{-2}  \,h^2\int_{\mathsf V^{\delta_1,\delta_2}_{\,\overline \Omega}(z) }  \big \vert  \nabla \phi_{1,h}\big \vert ^2\, e^{-\frac 2h f} +   O\big (  e^{-\frac 2h (\min_{\pa \Omega}f-f(x_0) +c)}\big ).
\end{align} 

Let now $z$ belong to $\pa \mathbf{C}_{ {\rm min} }\cap \pa \Omega$ and recall the coordinates $p\mapsto v(p)$ defined in Section~\ref{sec:coord},
see \eqref{eq.cv-pa-omega-nablafnon0}--\eqref{eq.G1}. We also define $  \boldsymbol{\hat \ell} :=\boldsymbol{\ell}\circ v^{-1}$.
With these coordinates, we have on  $\mathsf V_{\,\overline \Omega}^{\delta_1,\delta_2}(z)$:   
 \begin{equation}\label{eq:gradvarV}
 (\nabla  f)(v^{-1}) =  {}^tJ^{-1}    \nabla\hat  f ,\ (\nabla\phi_{1,h})(v^{-1})  =  {}^tJ^{-1}  \, \nabla  \varphi_z , \text{ and } (\jac \,\boldsymbol{\ell}) (v^{-1}) =\jac \,  \boldsymbol{\hat \ell}  \,   J^{-1}   . 
 \end{equation}

\textbf{Case 1, when $\nabla f(z)\neq 0$:} Using \eqref{eq.G1}, \eqref{eq:gradvarV}, and \eqref{eq.qm-local11}, we have
\begin{equation}\label{eq.chiUU02}
\int_{\mathsf V^{\delta_1,\delta_2}_{\,\overline \Omega}(z) }  \big \vert  \nabla \phi_{1,h}\big \vert ^2\, e^{-\frac 2h f}   =
\frac{ \int_{\vert v'\vert \le \delta_2}\int_{-2\delta_1}^0 G^{dd}(v)  \, \chi^2(v_d) \, \sqrt{|G|}(v) \, e^{-\frac 2h (\hat f (v)-2 \mu(z)  v_d)}   dv  }
{ \Big(\int_{-2\delta_1}^0\chi(t)\, e^{\frac 2h \mu(z)\, t}  dt\Big)^2  }
\end{equation}
and a straightforward computation shows that, when $h\rightarrow 0$ (see~\eqref{eq.chi-cut}), 
\begin{equation}\label{eq:normal1}
N_z:=\int_{-2\delta_1}^0\chi(t)\, e^{\frac 2h \mu(z)\, t}  dt=\frac h{2\mu(z)}\big(1+O(e^{-\frac ch})\big).
\end{equation}
On the other hand, using $G(0)=(\delta_{ij})$, \eqref{eq.chi-cut}, \eqref{eq.cv-pa-omega-nablafnon03}, and \eqref{eq.min_coo1}, the Laplace method leads to
\begin{equation*}
\begin{split}
\int_{\vert v'\vert \le \delta_2}\int_{-2\delta_1}^0 G^{dd}  \, \chi^2\, \sqrt{|G|} \, e^{-\frac 2h (\hat f -2 \mu(z)  v_d)}dv&=
\big(1+O(h)\big)\int_{\mathbb R^{d-1}}\int_{-\infty}^0e^{-\frac 2h (\hat f -2 \mu(z)  v_d)}dv\\
&=\big(1+O(h)\big)\frac{h}{2\mu(z)} \frac{ (\pi h)^{\frac{d-1}{2} }   e^{-\frac 2h  \hat f(0)}}{\big(\det \Hess \hat f|_{\{v_d=0\}}(0)\big)^{\frac12}}.
\end{split}
\end{equation*}
 Combining this equation with \eqref{Zx}, \eqref{eq.chiUU02}, and \eqref{eq:normal1} (recall that $f(z)=\min_{\pa \Omega}f$), we get
\begin{equation}\label{eq:dirichQM1}
  \frac{h^2}{Z_{1,h}^{2}}\int_{\mathsf V^{\delta_1,\delta_2}_{\,\overline \Omega}(z) }   \vert  \nabla \phi_{1,h} \vert^2\, e^{-\frac 2h f}   =
\frac{2\partial_{ n_\Omega}f(z)}{\sqrt{\pi}} \frac {\sqrt{ \det\Hess f(x_0)}}{\sqrt {\det  \Hess f|_{\partial\Omega}(z)}}
\sqrt h  \, e^{-\frac 2h (\min_{\pa \Omega}f-f(x_0))}
\big(1+O(h)\big).
\end{equation}

\textbf{Case 2, when $\nabla f(z)= 0$:}
Thanks to  \eqref{eq.G1}, \eqref{eq:gradvarV}, and \eqref{eq.qm-local12}, we have
\begin{equation*}\label{eq.chiUU1}
\int_{\mathsf V^{\delta_1,\delta_2}_{\,\overline \Omega}(z) }  \vert  \nabla \phi_{1,h} \vert^2\, e^{-\frac 2h f}   =\frac{ \int_{\vert v'\vert \le \delta_2}\int_{-2\delta_1}^0 G^{dd}(v)  \, \chi^2(v_{d}) \, \sqrt{|G|} (v) \, e^{-\frac 2h (\hat f (v) +\vert \mu(z)\vert  v_d^2)}  dv   }{ \Big(\int_{-2\delta_1}^0\chi(t)\, e^{-\frac 1h\vert \mu(z)\vert  \, t^2}  dt\Big)^2  },
\end{equation*}
where the denominator of the r.h.s. satisfies  in  the limit $h\to 0$  (see~\eqref{eq.chi-cut}),
\begin{align}\label{eq.chiUU2}
 N_z:=\int_{-2\delta_1}^0\chi(t)\, e^{-\frac 1h\vert \mu(z)\vert  \, t^2}  dt=\frac{\sqrt{\pi h} }{2\sqrt{\vert \mu (z)\vert}} \big(1+O(e^{-\frac ch})\big).
\end{align}
Furthermore, using $G(0)=(\delta_{ij})$, \eqref{eq.chi-cut}, \eqref{eq.det--}, and \eqref{eq.min_coo2}, the Laplace method gives,
when $h\to0$,
\begin{equation*}\label{eq.chiUU3}
 \int_{\vert v'\vert \le \delta_2}\int_{-2\delta_1}^0 G^{dd}  \, \chi^2 \, \sqrt{|G| } \, e^{-\frac 2h (\hat  f +\vert \mu(z)\vert  v_d^2)} dv = \frac{ (\pi h)^{\frac{d}{2} }  e^{-\frac 2h  \hat f(0)}}{\sqrt{   \det \Hess (  \hat  f +\vert \mu(z)\vert  v_d^2  )(0) }}
 \big(\frac12+ O ( \sqrt h) \big),
 \end{equation*}
 where, from the second item in Lemma~\ref{le.lep-michel} and \eqref{eq.det--},   $ \det \Hess (  \hat  f +\vert \mu(z)\vert  v_d^2  )(0)=-\det \Hess f(z)$. We refer to~\cite[Remark 25]{DoNe2}  for an explanation on the optimality of the remainder term $O(\sqrt h)$ in the previous equality. Using in addition
 \eqref{eq.chiUU2}, we obtain
 \begin{equation}\label{eq:dirichQM2}
\frac{h^2}{Z_{1,h}^{2}}\int_{\mathsf V^{\delta_1,\delta_2}_{\,\overline \Omega}(z) }  \vert  \nabla \phi_{1,h} \vert^2\, e^{-\frac 2h f}=   
\frac {2|\mu(z)| }{  \pi  } \,   \frac{\sqrt{ \det\Hess f(x_0)}}{\sqrt{  |\det\Hess f (z) | } } \,h\, e^{-\frac 2h (\min_{\pa \Omega}f-f(x_0))}  \big(1+ O(\sqrt h)\big).
 \end{equation}
Finally, \eqref{eq.chiUU0}, \eqref{eq:dirichQM1}, and \eqref{eq:dirichQM2} imply that $\mathsf f_{1,h}$ satisfies 
\eqref{E1}. 
\medskip

\noindent
\textbf{Step 2. The function  $\mathsf f_{1,h}$ satisfies 
\eqref{E2}.}

\noindent
Recall that $\nabla_{f,h} \mathsf f_{1,h}= Z_{1,h}^{-1} \, h e^{-\frac fh} \nabla\phi_{1,h} $
 is supported in $\overline{\mathbf{C}}_{ {\rm up} }\setminus \mathbf{C}_{ {\rm low} }$, 
 so the same holds for
$P_h\mathsf f_{1,h}=(\nabla_{f,h}^{*}+2\boldsymbol{\ell} \cdot )\nabla_{f,h}\mathsf f_{1,h}$. 
Thus, Proposition~\ref{pr.omega1},~\eqref{eq.nablapsix=0}, and~\eqref{Zx} imply that for $h$ small enough, 
 \begin{align}\label{eq.decoPhf}
\int_{\Omega}    \vert P_h   \,  \mathsf f_{1,h}\vert^2&=
\sum_{z\in \pa \mathbf{C}_{ {\rm min} }\cap \pa \Omega} \,  \,  \int_{ \mathsf V_{\,\overline \Omega}^{\delta_1,\delta_2}(z)} \vert  P_h   \,  \mathsf f_{1,h}\vert^2 +O\big (  e^{-\frac 2h (\min_{\pa \Omega}f-f(x_0) +c)}\big ).
\end{align} 
Since $\div\boldsymbol{\ell}=0$,  the same relation holds when replacing $ P_h   \,  \mathsf f_{1,h}$ by $P_h^* \mathsf f_{1,h}=(\Delta_{f,h}-2\boldsymbol{\ell} \cdot \nabla_{f,h}) \mathsf f_{1,h}$.

Let now $z$ belong to $\pa \mathbf{C}_{ {\rm min} }\cap \pa \Omega$.
Using the relations  $ \Delta_{f,h}=2h e^{-\frac fh}(-\frac h2 \Delta+\nabla f\cdot \nabla)e^{\frac fh}$
and~\eqref{eq:gradvarV}  
with $\varphi_z$ only  depending  on the variable $v_d$, 
we get  in the $v$-coordinates on $v( \mathsf V_{\,\overline \Omega}^{\delta_1,\delta_2}(z))$:
\begin{align*}
(\Delta_{f,h} \mathsf f_{1,h})\circ v^{-1}
&= \frac {2h\, e^{-\hat f/h}}{Z_{1,h}}\Big[\frac {-h}{2\sqrt {\vert G\vert}}\div\big(\sqrt{\vert G\vert}\,G^{-1}\nabla\varphi_z\big)+  \sum_{i,j}G^{ij}\partial_{v_j}\varphi_z\partial_{v_i} \hat f\Big]\\
 &=	\frac {h\, e^{-\hat f/h}}{Z_{1,h}}\Big[\frac {-h}{\sqrt {\vert G\vert}}\sum_{i}\partial_{v_i}(\sqrt{\vert G\vert}G^{id}\partial_{v_d}\varphi_z)+  2\partial_{v_d}\varphi_z\sum_{i}G^{id} \partial_{v_i} \hat f\Big].
\end{align*}
Moreover, recall that $\varphi_z(v)= \int_{v_d}^0\chi(t)e^{-\frac 1h\theta (t)}dt/N_z$,
where $\theta(t)=-2\mu(z) t$ when $\nabla f(z)\neq 0$ and $\theta(t)=\vert\mu(z)\vert t^2$ when $\nabla f(z)=0$ (see~\eqref{eq.qm-local11} and  \eqref{eq.qm-local12}), so that 
\begin{equation}
\label{eq.varphi'}
\pa_{v_d}\varphi_z(v)=-\frac{1}{N_z}\chi(v_{d}) e^{-\frac{\theta(v_{d})} h}\ \ \text{and} \ \ 
\pa^2_{v_d}\varphi_z(v)=-\frac{1}{N_z}\chi'(v_{d}) e^{-\frac {\theta(v_{d})} h} +\frac{1}{h\,N_z}\chi(v_{d}) \theta'(v_{d}) e^{-\frac{\theta(v_{d})} h}.
\end{equation} Hence, 
we have on $v(\mathsf V^{\delta_1,\delta_2}_{\,\overline \Omega}(z))$:
\begin{equation}\label{eq:DfhQM2}
\begin{split}
(\Delta_{f,h} \mathsf f_{1,h})\circ v^{-1}&= \frac {he^{-\frac 1h(\hat f+\theta)}}{N_z Z_{1,h}}\Big[\frac {h}{{\sqrt {\vert G\vert}}}\sum_i\partial_{v_i}(\sqrt{\vert G\vert}G^{id})\chi(v_d)\\
&\qquad \quad-\chi(v_d)\big(G^{dd}\theta'(v_d)+2\sum_jG^{jd}\partial_{v_j} \hat f\big)+h G^{dd}\chi'(v_d)\Big].
\end{split}
\end{equation}
Besides, 
we deduce from $\boldsymbol{\ell}\cdot\nabla_{f,h}\mathsf f_{1,h}=\frac {h\, e^{-\frac{ f} h}}{Z_{1,h}}\boldsymbol{\ell}\cdot \nabla \phi_{1,h}$,~\eqref{eq:gradvarV}, \eqref{eq.varphi'}, and \eqref{eq.G1}
 that   on $v(\mathsf V^{\delta_1,\delta_2}_{\,\overline \Omega}(z))$:
\begin{align}
\nonumber
(2\boldsymbol{\ell}\cdot\nabla_{f,h}\mathsf f_{1,h})\circ v^{-1}&=-\frac {he^{-\frac 1h(\hat f+\theta)} }{N_z Z_{1,h} }\ \chi(v_{d})  \big (\, \big [2 \boldsymbol{ \hat \ell} (0)+2 \jac\,  \boldsymbol{ \hat \ell} (0)   v\big ]\cdot  \, ^tJ^{-1} \,   \mathsf e_d\, +O(|v|^2)\, \big)\\
\label{eq.hnablafh}
&=-\frac {he^{-\frac 1h(\hat f+\theta)} }{N_z Z_{1,h} }\ \chi(v_{d})  \big (   2\boldsymbol{\hat  \ell}(0)\cdot  \, ^tJ^{-1} \,   \mathsf e_d + 2\jac\,\boldsymbol{\hat  \ell}(0)    v\cdot  J(0)\,   \mathsf e_d\, +O(|v|^2)\, \big).
\end{align} 
To go further in the computation of $P_h\mathsf f_{1,h}$ on $v(\mathsf V^{\delta_1,\delta_2}_{\,\overline \Omega}(z))$, let us   consider the two cases $\nabla f(z)\neq 0$ and $\nabla f(z)=0$ separately.

\textbf{Case 1, when $\nabla f(z)\neq 0$:} Since $G=(\delta_{ij})+O(|v|)$ (see \eqref{eq.G1}), $\partial_{v_j} \hat f=O(|v|)$ when $1\leq j \leq d-1$, and $\partial_{v_d} \hat f=\mu(z)$ (see \eqref{eq.cv-pa-omega-nablafnon03}),   we have 
  $$\sum_{j=1}^dG^{jd}\partial_{v_j} \hat f = G^{dd}\mu(z)+O(|v|^2).$$ 
Since moreover $\theta'(v_d)=-2\mu(z)$,  we deduce from \eqref{eq:DfhQM2} that
  \begin{align*} 
(\Delta_{f,h}\mathsf f_{1,h})\circ v^{-1}= \frac {h\chi(v_d)\, e^{-\frac 1h(\hat f+\theta)}}{N_z Z_{1,h}}\big [O(h)+  O(|v|^2)  \big]+ \frac {h^2e^{-\frac 1h(\hat f+\theta)} }{N_z Z_{1,h} } G^{dd}\chi'(v_d).
\end{align*}
Recall that  \eqref{normal} implies $  \boldsymbol{\hat \ell}(0)=0$. Hence, 
a Taylor expansion around $v=0$ of 
the relation~$\boldsymbol{\hat\ell}\cdot {}^tJ^{-1}    \nabla\hat  f =(\boldsymbol{\ell} \cdot \nabla f) \circ v^{-1}=0$ 
(see
\eqref{eq:gradvarV}) 
shows that, for all $v\in\mathbb R^{d}$,  $\jac\,   \boldsymbol{\hat \ell}(0)  \,   v\cdot   {}^tJ^{-1}(0)  \nabla\hat  f(0)=0$,
and then, using \eqref{eq.G1} and \eqref{eq.cv-pa-omega-nablafnon03}, $\jac\,   \boldsymbol{\hat \ell}(0)  \,   v\cdot   J(0)  \mathsf e_d=0$.
Thus, using \eqref{eq.hnablafh},
  $$(2\boldsymbol{\ell}\cdot\nabla_{f,h}\mathsf f_{1,h})\circ v^{-1} 
 =-\frac {he^{-\frac 1h(\hat f+\theta)} }{N_z Z_{1,h} }\ \chi(v_{d})  \times  O(|v|^2). $$
 Consequently, 
  \begin{align*} 
  (P_h\mathsf f_{1,h})\circ v^{-1}= \frac {h\chi(v_d)\, e^{-\frac 1h(\hat f+\theta)}}{N_z Z_{1,h}}\big[O(h)+  O(|v|^2)  \big] + \frac {h^2e^{-\frac 1h(\hat f+\theta)} }{N_z Z_{1,h} } G^{dd}\chi'(v_d).
\end{align*} 
Since $\chi'=0$ in a neighborhood of $0$ in $\mathbb R^{-}$ (see \eqref{eq.chi-cut}), we obtain from \eqref{eq.min_coo1},~\eqref{Zx},~\eqref{eq:normal1}, and   the Laplace method that  when $h\to 0$:
\begin{align}
\nonumber
 \int_{ \mathsf V_{\,\overline \Omega}^{\delta_1,\delta_2}(z)}  \vert P_h\mathsf f_{1,h}\vert^2 &=\frac 1{N_z^2 Z_{1,h}^2}\int_{ v\big (\mathsf V^{\delta_1,\delta_2}_{\,\overline \Omega}(z)\big)} O(h^4+h^2|v|^4)e^{-\frac 2 h(\hat f-2\mu(z)v_d)}dv +    O  (e^{-\frac ch} )\, e^{-\frac 2h (\min_{\pa \Omega}f-f(x_0))}\\
 \label{eq.E22-neq0}
 &= O( h^\frac52)\,  e^{-\frac 2h (\min_{\pa \Omega}f-f(x_0))} =O(h^2)
\int_{\mathsf V^{\delta_1,\delta_2}_{\,\overline \Omega}(z) }   \vert  \nabla \mathsf f_{1,h}\vert ^2\, e^{-\frac 2h f},
\end{align}
where we used \eqref{eq:dirichQM1} to get the last equality. 

  \textbf{Case 2, when $\nabla f(z)= 0$:} From \eqref{eq.hnablafh} and $  \boldsymbol{\hat \ell}(0)=0$ (see \eqref{eq.incluw}), we have 
 $$
 ( 2\boldsymbol{\ell}\cdot\nabla_{f,h}\mathsf f_{1,h})\circ v^{-1} =-\frac {he^{-\frac 1h(\hat f+\theta)} }{N_z Z_{1,h} }\ \chi(v_{d})  \big ( 2 v\cdot    {}^t\jac\,   \boldsymbol{\hat \ell} (0)   J(0)  \mathsf e_d +O(|v|^2)\, \big).
$$
Therefore, using \eqref{eq:DfhQM2}, $G=(\delta_{ij})+O(|v|)$ (see \eqref{eq.G1}) and $\partial_{v_j} \hat f=O(|v|)$ for all  $j=\{1,\ldots,d\}$:
  \begin{align*}
  (P_h\mathsf f_{1,h})\circ v^{-1}&=\frac {he^{-\frac 1h(\hat f+\theta)} }{N_z Z_{1,h} } \chi(v_d)\Big[ O(h)-2G^{dd}|\mu(z)|v_d-2\sum_jG^{jd}\partial_{v_j} \hat f \\
  &\quad -2v\cdot    {}^t\jac\,   \boldsymbol{\hat \ell}(0) J(0)  \mathsf e_d +O(|v|^2)\Big] + \frac {h^2e^{-\frac 1h(\hat f+\theta)} }{N_z Z_{1,h} } G^{dd}\chi'(v_d)\\
  &=\frac {2he^{-\frac 1h(\hat f+\theta)} }{N_z Z_{1,h} } \chi(v_d)\Big[ O(h)- |\mu(z)|v_d- \partial_{v_d} \hat f \\
  &\quad -v\cdot    {}^t\jac\,   \boldsymbol{\hat \ell}(0)     J(0)  \mathsf e_d +O(|v|^2)\Big] + \frac {h^2e^{-\frac 1h(\hat f+\theta)} }{N_z Z_{1,h} } G^{dd}\chi'(v_d).
  \end{align*}
We have moreover $ \partial_{v_d} \hat f=  v\cdot \Hess \hat f(0)\mathsf e_d+O(|v|^2)$
 and \eqref{normal} implies  $[\Hess f(z)+ \, ^t\jac \, \boldsymbol{\ell} (z)] n_\Omega(z)= \mu(z) n_\Omega(z)$, which becomes in the $v$-coordinates, using  \eqref{eq.G1} (see also \eqref{eq.det--}):
\begin{align*}  \big(\Hess \hat f(0)+ {}^t\jac\,   \boldsymbol{\hat \ell}(0) \, J(0)\big)\mathsf e_d
&=    {}^tJ(0)\big(\Hess f(z)+ \, ^t\jac \, \boldsymbol{\ell} (z)\big)n_\Omega(z) \\
&=\mu(z) {}^tJ(0)J (0)\mathsf e_d=\mu(z)  \mathsf e_d.
\end{align*}
It follows that  $|\mu(z)|v_d+ \partial_{v_d} \hat f 
 +v\cdot    {}^t\jac\,   \boldsymbol{\hat \ell}(0)     J(0)  \mathsf e_d = O(|v|^2)$ and consequently, 
   \begin{align*} 
  (P_h\mathsf f_{1,h})\circ v^{-1}= \frac {2h\chi(v_d)\, e^{-\frac 1h(\hat f+\theta)}}{N_z Z_{1,h}} [O(h)+  O(|v|^2)]   + \frac {h^2e^{-\frac 1h(\hat f+\theta)} }{N_z Z_{1,h} } G^{dd}\chi'(v_d).
\end{align*} 
Hence, since $\chi'=0$ around $0$, it follows from \eqref{eq.min_coo2},~\eqref{Zx},~\eqref{eq.chiUU2}, \eqref{eq:dirichQM2}, and  from the  Laplace method that when $h\to 0$,
\begin{align}
\nonumber
 \int_{ \mathsf V_{\,\overline \Omega}^{\delta_1,\delta_2}(z)}  \vert P_h\mathsf f_{1,h}\vert^2 &=\frac{   h^{\frac{d}{2} }  O(h^4)   }{  h\,  h^{\frac{d}{2} }  }      \,    e^{-\frac 2h (\min_{\pa \Omega}f-f(x_0))}  +    O  (e^{-\frac ch} )\, e^{-\frac 2h (\min_{\pa \Omega}f-f(x_0))}\\
 \label{eq.E22}
 &= O( h^3) e^{-\frac 2h (\min_{\pa \Omega}f-f(x_0))} =O(h^2)
\int_{\mathsf V^{\delta_1,\delta_2}_{\,\overline \Omega}(z) }   \vert  \nabla \mathsf f_{1,h} \vert ^2\, e^{-\frac 2h f}.
\end{align}
Plugging \eqref{eq.E22-neq0} and \eqref{eq.E22}  into \eqref{eq.decoPhf}, and  using \eqref{eq.chiUU0} and \eqref{E1}, then leads to:
$$
\int_{\Omega}    \vert P_h   \,  \mathsf f_{1,h}\vert^2=O(h^2)\langle P_h\mathsf f_{1,h},\mathsf f_{1,h}\rangle.
$$ 
Therefore $\mathsf f_{1,h}$ satisfies 
\eqref{E2}. 
\medskip

\noindent
\textbf{Step 3. The function $\mathsf f_{1,h}$ satisfies 
\eqref{E3}.}

\noindent
Recall that $P_{h}^*= \Delta_{f,h}-2\boldsymbol{\ell}\cdot \nabla_{f,h}$ according to  Proposition~\ref{pr.spectre} and to \eqref{div}. 
Therefore, the  computations of the previous step show that, on any $v(\mathsf V_{\,\overline \Omega}^{\delta_1,\delta_2}(z))$, $z\in\pa \mathbf{C}_{ {\rm min} }\cap \pa \Omega$:  
$$ (P_{h}^* \mathsf f_{1,h})\circ v^{-1}= \begin{cases}\frac {h\chi(v_d)\, e^{-\frac 1h(\hat f+\theta)}}{N_z Z_{1,h}} [O(h)+  O(|v|^{2})]  + \frac {h^2e^{-\frac 1h(\hat f+\theta)} }{N_z Z_{1,h} } G^{dd}\chi'(v_d)\ \ &\text{when $\nabla f(z)\neq 0$,}\\
\frac {h\chi(v_d)\, e^{-\frac 1h(\hat f+\theta)}}{N_z Z_{1,h}} [O(h)+  O(|v|)]  + \frac {h^2e^{-\frac 1h(\hat f+\theta)} }{N_z Z_{1,h} } G^{dd}\chi'(v_d)\ \ &\text{when $\nabla f(z)= 0$.}
\end{cases}$$
It follows that, when $h\to0$,
$$\int_{ \mathsf V_{\,\overline \Omega}^{\delta_1,\delta_2}(z)} \vert  P_h^*   \,  \mathsf f_{1,h}\vert^2= \begin{cases} O(h^2)
\int_{\mathsf V^{\delta_1,\delta_2}_{\,\overline \Omega}(z) }   \vert  \nabla \mathsf f_{1,h}\vert ^2\, e^{-\frac 2h f}  \ \ &\text{when $\nabla f(z)\neq 0$,}\\
  O(h) \int_{\mathsf V^{\delta_1,\delta_2}_{\,\overline \Omega}(z) }   \vert  \nabla \mathsf f_{1,h}\vert ^2\, e^{-\frac 2h f}   \ \ &\text{when $\nabla f(z)= 0$,}\end{cases}$$
  and hence, according to \eqref{eq.decoPhf} (with $P_h$ replaced by $P_h^*$), \eqref{eq:dirichQM1}, and \eqref{eq:dirichQM2}:
$$  
  \int_{ \Omega} \vert  P_h^*   \,  \mathsf f_{1,h}\vert^2=
  \begin{cases} O(h)\,h\,e^{-\frac 2h (\min_{\pa \Omega}f -f(x_0))}  \ \ &\text{when $\kappa^P_1= 0$,}\\
  O(h^{2})\,h^{\frac12} \,e^{-\frac 2h (\min_{\pa \Omega}f -f(x_0))}   \ \ &\text{when $\kappa^P_2= 0$,}\\
  O(h^{\frac32})\,h^{\frac12} \,e^{-\frac 2h (\min_{\pa \Omega}f -f(x_0))}   \ \ &\text{when $\kappa^P_1\neq 0$ and $\kappa^P_2\neq 0$.}
  \end{cases}
 $$
This proves that $\mathsf f_{1,h}$ satisfies 
\eqref{E3} and completes the proof of Proposition~\ref{pr.QM-1}. 
\end{proof}

\bigskip

\noindent
\section*{Appendix}

 In this appendix, we prove Proposition~\ref{pr.spectre}.

\begin{proof}[Proof of Proposition~\ref{pr.spectre}]
 Let $h>0$ be fixed. Let us first prove the first item in Proposition~\ref{pr.spectre} and take $u\in D(P_{h})= H^{2}(\Omega)\cap H^1_0(\Omega)$. Since $\boldsymbol{\ell}\cdot\nabla f=0$ 
 and then $\boldsymbol{\ell}\cdot\nabla_{f,h}=h\,\boldsymbol{\ell}\cdot\nabla$ according to \eqref{ortho} and to the relation $\nabla_{f,h}:= h\, e^{-\frac f h}\nabla e^{\frac f h}=h\nabla + \nabla f$,   it holds
\begin{align*}
\int_\Omega  (\boldsymbol{\ell} \cdot \nabla_{f,h}  u )\,   \overline u =-\int_\Omega   u \,(\boldsymbol{\ell} \cdot \nabla_{f,h}  \overline  u) - h\int_\Omega(\div \boldsymbol{\ell})|u|^{2}.
\end{align*}
Therefore,  one has $2\Re \langle  \boldsymbol{\ell}  \cdot \nabla_{f,h} u , u \rangle_{L^2(\Omega)}=- h\int_\Omega(\div \boldsymbol{\ell})|u|^{2}$, and thus, using \eqref{eq.Witten}
and \eqref{eq.unitary}:
\begin{equation}\label{eq.energy}
\forall u \in H^{2}(\Omega)\cap H^1_0(\Omega)\,,\ \ \  \Re \langle  P_hu , u \rangle_{L^2(\Omega)}=\int_{\Omega} \vert \nabla_{f,h} u\vert^2- h\int_\Omega(\div \boldsymbol{\ell})|u|^{2}.
\end{equation}
This implies that $ P_h+h\|\div\boldsymbol{\ell}\|_{\infty}:H^{2}(\Omega)\cap H^1_0(\Omega)\subset L^2(\Omega)\to L^2(\Omega)$ is accretive.  
Using moreover the Lax-Milgram Theorem 
and the elliptic regularity of $P_{h}$, the operator $P_h+\lambda$ is invertible for $\lambda>0$ large enough. Thus, $P_h$ is maximal quasi-accretive and is in particular closed.
In addition, from  the compact injection $H^1_0(\Omega)\subset L^2(\Omega)$, $P_h$ has  a  compact resolvent.

Let us now prove that $P_h$ is sectorial. For all $u\in H^{2}(\Omega)\cap H^1_0(\Omega)$, it holds 
$$ \Im \langle P_hu , u \rangle_{L^2(\Omega)} = \Im \int_{\Omega}  (2\boldsymbol{\ell} \cdot \nabla_{f,h}  u )\,  \overline u.$$
Consequently, there exists $C>0$ such that for all $u\in H^{2}(\Omega)\cap H^1_0(\Omega)$ and all $\ve>0$, one has 
$$\vert  \Im \langle P_hu , u \rangle_{L^2(\Omega)} \vert \le C \Vert \nabla_{f,h}u\Vert_{L^2}\Vert u\Vert_{L^2(\Omega)}
\le C\Big[ \frac{\ve}{2} \Vert \nabla_{f,h}u\Vert_{L^2(\Omega)}^2+\frac{1}{2\ve}\Vert u\Vert^2_{L^2(\Omega)}\Big].$$
Taking $\lambda >0$ and choosing $\ve >0$ such that $1-\ve \frac{\lambda C}{2}\geq \frac 12$, one has, using \eqref{eq.energy},
\begin{align*}
\Re  \langle P_hu , u \rangle_{L^2(\Omega)}  -\lambda \vert \Im \langle  P_hu , u \rangle_{L^2(\Omega)} \vert \ge \frac 12 \Vert \nabla_{f,h}u\Vert_{L^2(\Omega)}^2- \big(\frac{\lambda C}{2\ve} +h\|\div\boldsymbol{\ell}\|_{\infty}\big)\Vert  u\Vert_{L^2(\Omega)}^2.\end{align*}
Therefore,  for some $ a_{h}\in \mathbb R$, $\Re  \langle (P_h+a_{h})u , u \rangle_{L^2(\Omega)}  \ge\lambda \vert \Im \langle P_h u , u \rangle_{L^2(\Omega)} \vert$. The  numerical range of $P_h$ 
is then included in  the sector $\{\mathsf z\in \mathbb C,  \vert \Im \mathsf z\vert \leq \lambda^{-1} \Re (\mathsf z+a_{h}) \}$, so $P_h$ is sectorial. 

Let us now prove the second item in Proposition~\ref{pr.spectre}.
   With the previous arguments, the formal adjoint 
 $$P_h^\dag = \Delta_{f,h} - 2h\,\boldsymbol{\ell} \cdot \nabla - 2h\div\boldsymbol{\ell}   =\Delta_{f,h}- 2\boldsymbol{\ell} \cdot \nabla_{f,h}- 2h\div\boldsymbol{\ell}$$  
   of $P_{h}$ endowed with
the domain $D(P_h)=H^{2}(\Omega)\cap H^1_0(\Omega)$ is also maximal quasi-accretive, with a compact resolvent, and sectorial.
To conclude, it thus just remains to show that
$(P_h^\dag,D(P_h))=(P_{h}^{*},D(P_{h}^{*}))$, where
$P_{h}^{*}:D(P_{h}^{*})\to L^{2}(\Omega)$ is the adjoint of $P_h$.
But, for any $u,v\in D(P_{h})=H^{2}(\Omega)\cap H^1_0(\Omega)$,
we have by integration by parts
$$
\langle P_hu , v \rangle_{L^2(\Omega)}=\langle u, P_h^\dag v  \rangle_{L^2(\Omega)}\,,
$$
which implies, by definition of $P_{h}^{*}:D(P_{h}^{*})\to L^{2}(\Omega)$,  that 
$$(P_h^\dag,D(P_h))\subset (P_{h}^{*},D(P_{h}^{*})).$$
Since moreover $(P_{h}^{*},D(P_{h}^{*}))$ is
maximal quasi-accretive (since $P_{h}$ is)
 as well as $(P_h^\dag,D(P_h))$,
 it  necessarily holds  
$(P_h^\dag,D(P_h))=(P_{h}^{*},D(P_{h}^{*}))$.

Let us lastly prove the third item in Proposition~\ref{pr.spectre}.  First, by standard results on elliptic regularity (see e.g. \cite[Section 6.3]{Eva}), any eigenfunction  $u\in H^2(\Omega)\cap H^1_0(\Omega)$
of $P_{h}$ (resp. of $P_{h}^{*}$) belongs  to~$\mathcal C^\infty (\overline \Omega)$. Moreover, according to~\cite[Theorems 1.3, 1.4, and 2.7]{du2006order} (see also the slightly weaker result stated in \cite[Theorem 3 in Section 6.5.2]{Eva}), $P_h$ 
(resp. $P_{h}^{*}$)
admits a real eigenvalue $\lambda_{1,h}^P $ (resp. $\lambda_{1,h}^{P^{*}} $) with algebraic multiplicity one such that:
\begin{itemize}
\item there exists an associated eigenfunction $u_{1,h}^P$ (resp. $u_{1,h}^{P^{*}}$)
which is positive within $\Omega$,
\item any other eigenvalue $\lambda$ of $P_h$ (resp. of $P_{h}^{*}$) satisfies 
 $\Re \lambda >  \lambda^{P}_{1,h}$ (resp. $\Re \lambda >  \lambda^{P^{*}}_{1,h}$). 
 \end{itemize}
 
Since in addition $\sigma(P^{*}_{h})=\overline{\sigma(P_{h})}$
 (see e.g. \cite[Section 6.6 in Chapter 3]{Kato}), we have $\lambda^{P}_{1,h}=\lambda^{P^{*}}_{1,h}$
 and
it thus only remains to show that $\lambda_{1,h}^P>0$, which
is a consequence of  the weak maximum principle~\cite[Theorem 1 in Section 6.4.1]{Eva}. Indeed,
according to \eqref{eq.unitary},
if it was not the case, 
the second-order elliptic operator 
without zeroth-order term 
$L_{h}=-\frac h2 \Delta +(\nabla f + \div\boldsymbol{\ell})\cdot\nabla $ would satisfy
$$
L_{h}\big(e^{\frac fh} u_{1,h}^{P}\big)=\frac{\lambda_{1,h}^P}{2h}e^{\frac fh} u_{1,h}^{P}\leq 0\ \ \text{in}\ \ \Omega,
$$
which would imply by the weak maximum principle that $\max_{\overline\Omega} \big(e^{\frac fh} u_{1,h}^{P}\big)=\max_{\pa\Omega} \big(e^{\frac fh} u_{1,h}^{P}\big)=0$, contradicting 
$u_{1,h}^{P}> 0$
 in $\Omega$.  
\end{proof}

 \medskip

\begin{sloppypar}
\noindent
\textbf{Acknowledgement}\\
This work was  supported by the ANR-19-CE40-0010, Analyse Quantitative de Processus  Métastables (QuAMProcs).
B.N. is supported by the grant IA20Nectoux from the Projet I-SITE Clermont CAP 20-25. 
\end{sloppypar}

\bibliographystyle{amsplain}
\bibliography{nonreversible}
\end{document}